\documentclass[a4paper,11pt]{amsart}

\usepackage{amssymb}
\usepackage{amsfonts}
\usepackage{amsmath}

\usepackage{graphicx}
\usepackage[all]{xy}
\usepackage{array}

\usepackage{hyperref}

\usepackage{amsthm}

\newtheorem{theorem}{Theorem}[section]
\newtheorem*{theoremnn}{Theorem}
\newtheorem{problem}[theorem]{Problem}

\newtheorem{lemma}[theorem]{Lemma}
\newtheorem{proposition}[theorem]{Proposition}
\newtheorem{fact}[theorem]{Fact}

\newtheorem{corollary}[theorem]{Corollary}

\theoremstyle{definition}
\newtheorem{definition}[theorem]{Definition}

\newtheorem{notation}[theorem]{Notation}

\newtheorem{remark}[theorem]{Remark}

\newtheorem{example}[theorem]{Example}
\newtheorem{fundexample}[theorem]{Fundamental example}

\newtheorem{convention}[theorem]{Convention}
\newtheorem{openpbs}[theorem]{Problems}

\theoremstyle{remark}

\newcommand{\Si}{\mathfrak{S}}
\newcommand{\Ext}{{\mathrm{Ext}}}
\newcommand{\Tor}{{\mathrm{Tor}}}

\newcommand{\End}{{\mathrm{End}}}
\newcommand{\Id}{{\mathrm{Id}}}

\newcommand{\FF}{\mathcal{F}}
\newcommand{\PP}{\mathcal{P}}

\newcommand{\Cr}{\mathrm{cr}}

\newcommand{\V}{\mathcal{V}}
\newcommand{\A}{\mathcal{A}}
\newcommand{\AC}{\mathcal{AC}}

\newcommand{\Z}{\mathbb{Z}}

\newcommand{\D}{\mathcal{D}}

\newcommand{\EML}{\mathrm{EML}}
\newcommand{\cw}{\mathrm{cw}}
\newcommand{\w}{\mathrm{w}}
\newcommand{\even}{\mathrm{ev}}

\DeclareMathOperator*{\colim}{colim}

\newcommand{\Fp}{{\mathbb{F}_p}}
\newcommand{\Fq}{{\mathbb{F}_q}}
\renewcommand{\V}{\mathcal{V}}
\newcommand{\HH}{\mathcal{H}}
\newcommand{\OO}{\mathcal{O}}
\newcommand{\C}{\mathcal{C}}
\newcommand{\U}{\mathcal{U}}

\newcommand{\strict}{\mathrm{strict}}
\newcommand{\ord}{\mathrm{ord}}
\newcommand{\kk}{\Bbbk}
\newcommand{\Exp}{\mathrm{Exp}}
\newcommand{\Alg}{\mathrm{Alg}}
\newcommand{\Coalg}{\mathrm{Coalg}}
\renewcommand{\H}{\mathrm{Hopf}}
\newcommand{\Proj}{\mathrm{P}}
\newcommand{\add}{\mathrm{add}}
\newcommand{\conn}{\mathrm{conn}}

\newcommand{\Fct}{\mathrm{Fct}}
\newcommand{\Mod}{{\mathrm{Mod}_\kk}}
\newcommand{\Hom}{\mathrm{Hom}}
\newcommand{\Modgr}{{\mathrm{Mod}^*_\kk}}

\title{On the structure of graded commutative exponential funtors}

\author[A. Touz\'e]{Antoine Touz\'e} 
\address{Universit\'e Lille\\
Laboratoire Painlev\'e\\
Cit\'e Scientifique - B\^atiment M2\\
F-59655 Villeneuve d'Ascq Cedex, France}
\email{antoine.touze@univ-lille.fr}

\date{\today}

\begin{document}

\begin{abstract}
We investigate the structure of graded commutative exponential functors. We give applications of these structure results, including computations of the homology of the symmetric groups and of extensions in the category of strict polynomial functors.
\end{abstract}

\maketitle

 \setcounter{tocdepth}{1}

\section{Introduction}

Let $\V$ be an additive category, and let $\Modgr$ be the category of graded modules over a commutative ring $\kk$. An exponential functor is a strong monoidal functor $E:(\V,\oplus, 0)\to (\Modgr,\otimes,\kk)$. The name `exponential functor' comes from the analogy between the fundamental property of exponential functions and the so-called `exponential isomorphism'
$$E(V)\otimes E(W)\simeq E(V\oplus W)\;.$$

Exponential functors appear in many places in algebra and topology. For example, classical universal constructions such as the free graded commutative algebra on a graded vector space $V$, or the free divided power algebra on $V$, yield examples of exponential functors with source $\V=\Modgr$ (see section \ref{subsec-ex}). As these first examples suggest, any exponential functor has a canonical algebra structure. To be more explicit, the product on $E(V)$ is the composite map (where $\sigma_V:V^{\oplus 2}\to V$ is the sum map, i.e. $\sigma_V(v,w)=v+w$):
$$E(V)\otimes E(V)\simeq E(V\oplus V)\xrightarrow[]{E(\sigma_V)} E(V)\;.$$
There is also a canonical coalgebra structure on $E(V)$. Thus exponential functors are tightly connected to Hopf algebras. In fact, as we explain it in section \ref{sec-ordvsHopf}, bicommutative Hopf algebras and commutative affine group schemes naturally yield examples of exponential functors. 
Other kinds of examples of exponential functors come from some usual constructions in homological algebra and topology. As a typical example, the Hochschild Homology of an exponential functor is again an exponential functor, see section \ref{subsec-exHT}. Another typical example is provided by applying a generalized homology theory with K\"unneth isomorphisms to Eilenberg-Mac Lane spaces, see example \ref{ex-dualMorava}. Sometimes, the exponential structure is slightly less obvious, as in the case of the homology of symmetric groups explained in section \ref{sec-sym}.
Finally, 
let us mention that the name `exponential functor' comes from functor homology. To the author's knowledge, the name first appeared in \cite{Franjou}. Hence, the notion of exponential functor plays a role in homological algebra computations related to unstable modules over the Steenrod algebra, to the cohomology of finite groups $GL_n(\Fq)$, and to the cohomology of algebraic groups schemes $GL_n$, see e.g. \cite{FFSS,Chalupnik2,TouzeBar,Drup}.

Although they appear in many places, and although they play an important role in some fundamental computations, exponential functors do not seem to have been studied on their own, and their use seems to have essentially been limited to a few computational tricks. The purpose of this paper is to write down the basic theory of exponential functors, in order to place these tricks into a wider perspective, and to allow for further and more powerful applications. 
As already explained, an exponential functor $E$ yields a family of graded algebras and also of coalgebras $E(V)$, depending functorially on the variable $V$. In general, the whole structure is overdetermined. The main problem that we investigate in this article is the following one.
\begin{problem}\label{pb-main}
Which part of the structure (multiplication, comultiplication, functoriality with respect to $V$\dots) on the graded vector spaces $E(V)$ is sufficient to completely determine an exponential functor $E$?
\end{problem}
We are interested in problem \ref{pb-main} because it has practical implications. 
It is often the case in practice that only a part of the whole structure can be easily described or computed, while the rest of the structure is interesting as well - sometimes it is even the most interesting part. An archetypal example is given by the article \cite{CHN}, in which Cohen, Hemmer and Nakano need to determine the homology of symmetric groups $H_i(\Si_d,V^{\otimes d})$ as representations of $GL(V)$. As we explain it in section \ref{sec-sym}, these cohomology groups are part of an exponential  functor $E$, whose algebra structure is well-known. Elementary results on exponential functors then allows an effortless  reconstruction of the sought-after representations of $GL(V)$.

\bigskip
 
Let us describe in more details the content of the article. 
As mentioned in the title, we essentially restrict ourselves to \emph{graded commutative} exponential functors, i.e. the exponential functors whose product is graded commutative\footnote{As specified in section \ref{subsec-conventions}, from section \ref{sec-elementary} until the end of the article, all exponential functors are implicitly graded commutative and we drop the words `graded commutative' for the sake of concision.}. In view of applications, we also consider the variant of \emph{strict exponential functors}. They are algebro-geometric analogues of exponential functors. In order to define them, one simply replaces ordinary functors $E:\V\to \Modgr$ in the definition by some highly structured functors $E:\Proj_\kk\to \Modgr$, namely the strict polynomial functors of \cite{FS} (here $\Proj_\kk$ is the category of finitely generated projective $\kk$-modules, see the standard notations given at the end of the introduction), or rather the immediate generalization of strict analytic functors which are simply infinite direct sums of strict polynomial functors. All the relevant definitions and basic examples are given in section \ref{sec-def}.

Problem \ref{pb-main} has two main directions.  
Given an exponential functor $E$, one can forget functoriality but still remember some other structures, such as the algebra or coalgebra structures on $E(V)$, for some $V$ in the domain category $\V$.  On the other extreme, one may forget everything but functoriality. This is depicted by the following diagram:
$$\left\{
\text{
\begin{tabular}{c}
$\kk$-modules $+$ \\
extra structure\\
($\mu$, $\Delta$\dots)
\end{tabular}}
\right\}
\xleftarrow{\;(I)\;}
\left\{
\text{
\begin{tabular}{c}
exponential\\ functors 
\end{tabular}
}
\right\}
\xrightarrow[]{\;(II)\;}
\left\{
\text{
\begin{tabular}{c}
functors\\
$\V\to\Modgr$
\end{tabular}}
\right\}.
$$

Knowing how much information is preserved (or lost) by the forgetful functor $(I)$ is the main concern of sections \ref{sec-elementary}, \ref{sec-ordvsHopf} and \ref{sec-strictvsHopf}. In these sections we prove some reconstruction theorems, which allow in certain situations to reconstruct an exponential functor from the description of some algebra, coalgebra or Hopf algebra structures. As an elementary example, we quote theorem \ref{thm-classif-ord}. 
\begin{theoremnn}
Let $\Exp_c$ denote the category of graded commutative exponential functors with domain $\V=\Proj_R$ and codomain graded $\kk$-modules, where $R$ a ring and $\kk$ a is commutative ring. Let $\HH$ be the category of graded bicommutative Hopf $\kk$-algebras, and let $_R\HH$ denote the category of $R$-modules in $\HH$. Evaluation on $R$ yields an equivalence of categories:
$$\begin{array}[t]{ccc}\Exp_c & \xrightarrow[]{\simeq} & {_R}\HH\\
 E & \mapsto & E(R)
\end{array}
\;.$$
\end{theoremnn}

In theorem \ref{thm-classif-strict}, we give an analogue of theorem \ref{thm-classif-ord} for strict exponential functors. However, the proof is technically more involved in the case of strict exponential functors, and we need to assume that $\kk$ is a perfect field (although we suspect that theorem \ref{thm-classif-strict} is at least valid over an arbitrary field). Other reconstruction results which are useful in applications are given in theorems \ref{thm-alg} and \ref{thm-coalg}. The latter only consider algebra or coalgebra structures.

Theorems \ref{thm-classif-ord} and \ref{thm-classif-strict} bridge the gap between graded commutative exponential functors and graded bicommutative Hopf algebras. In section \ref{sec-indecomp}, we exploit this link, together with the classical theory of Dieudonn\'e modules \cite{Schoeller} and the representation theory of string algebras \cite{CB} to obtain a clear portrayal of connected exponential functors, under some reasonable hypotheses. 
The next theorem summarizes the results obtained (see proposition \ref{prop-defiEAW}, theorem \ref{thm-classif-indecomp} and corollary \ref{cor-KRS} for more precise statements). In this theorem, the reflexivity hypothesis is a mild finiteness condition on exponential functors. For example, if $E$ is a connected strict exponential functor and if the algebra $E(\kk)$ is finitely generated, then $E(\kk)$ is reflexive. The hypothesis on homological dimension is strong, but it is satisfied in many cases of interest. For example if $\V=\Proj_R$, it is equivalent to the fact that $R\otimes_\Z\kk$ has homological dimension zero. (This is satisfied when $\kk=\Z$, $\Z/n\Z$ or a finite field, which is sufficient for most of the applications described in this article).

\begin{theoremnn}
Let $\kk$ be a perfect field of positive characteristic, and let $E$ be either a strict exponential functor, or an exponential functor with source $\V$ and target $\Modgr$. In the latter case, assume in addition that the category $\V$ is such that the category of additive functors $\V\to\Mod$ has homological dimension zero.

If $E$ is graded commutative, connected (in the sense that for all $V$, the graded algebra $E(V)$ equals $\kk$ in degree zero) and reflexive, then it can be written in a unique way (up to permutation of the indices) as a countable tensor product:
$$E\simeq \bigotimes_{i\in I}E_i$$
where the exponential functors $E_i$ are indecomposable with respect to the tensor product. 
Moreover, such indecomposable exponential functors are classified. They are exterior algebras $\Lambda(A[d])$ on a simple additive functor $A$ placed in odd degree $d$, or they are certain exponential functors $\Sigma(A[d],w)$ where $A$ is a simple additive functor, $d$ is a positive even integer, and $w$ is a word (possibly infinite) on the alphabet $\{F,V\}$. 
\end{theoremnn}
  
This theorem shows that the situation in positive characteristic is more complicated than when $\kk$ is a field of characteristic zero (where any connected exponential functor is a free graded commutative algebra over a graded additive functor,  by classical results of Milnor and Moore \cite{MilnorMoore}, cf. lemma \ref{lm-classif-carzero}). The situation is not too complicated however, and this theorem is a great help to study the forgetful functor $(II)$. Building on the results of section \ref{sec-indecomp}, we obtain the following strong rigidity result for strict exponential functors in theorem \ref{thm-uniqueness-strict}. The hypothesis of virtual reflexivivity in the statement is a mild finiteness condition on strict exponential functors, slightly weaker than reflexivity. (Example \ref{ex-infinity} shows that the statement of the theorem is not true if we remove this finiteness condition).

\begin{theoremnn}
Let $E$ and $E'$ be graded commutative strict exponential functors over a perfect field $\kk$. Assume that $E$ is virtually reflexive. If the underlying strict analytic functors of $E,E':\Proj_\kk\to \Modgr$ are isomorphic then $E$ and $E'$ are isomorphic as strict exponential functors.
\end{theoremnn}

Theorem \ref{thm-uniqueness-strict} shows that there is at most one way of endowing a given graded strict analytic functor (satisfying mild finiteness hypotheses) with the structure of an exponential functor. 
We have not been able to prove the analogous result for ordinary (i.e. non strict) exponential functors. However, in section \ref{sec-filtr} we investigate filtrations of exponential functors, and we show that much of the structure of an exponential functor $E$ can be determined from the structure of the underlying functor $E:\V\to \Modgr$. For example, we prove in theorem \ref{thm-aug} that the augmentation filtration of an exponential functor coincides with the copolynomial filtration of the underlying functor $E:\V\to \Modgr$. Building on these results, we obtain the following weak analogue of theorem \ref{thm-uniqueness-strict} in corollary \ref{cor-weak}. 
\begin{theoremnn}
Assume that $\kk$ is a perfect field, and that the category of additive functors $\V\to\Mod$ has homological dimension zero. Let $E$ and $E'$ be two graded commutative exponential functors from $\V$ to $\Modgr$. Assume that $E$ is absolutely reflexive. If the underlying functors $E,E':\V\to \Modgr$ are isomorphic, then $\mathrm{gr}\, E$ and $\mathrm{gr}\, E'$ are isomorphic as exponential functors, where `$\mathrm{gr}$' refers to the graded object associated to the augmentation filtration, or to the coradical filtration.
\end{theoremnn}

In order to illustrate the interest of problem \ref{pb-main} and of our results, we have included various examples and applications all along the article. In particular,
\begin{itemize}
\item in section \ref{sec-sym} we give an approach to the computations of \cite{CHN} which does not use Dyer-Lashof operations. 
\item In section \ref{sec-bar} we generalize the functor homology computations of \cite{TouzeBar}, thereby computing many extension groups in the category of rational representations of $GL_n$.
\item In section \ref{sec-strict-vs-ord} we investigate the exponential variant of the problem\footnote{This problem is closely related to the problem of understanding the forgetful functor from polynomial representations of the group schemes $GL_{n,\Fq}$ to the representations of the finite groups $GL_n(\Fq)$. Indeed $\PP_{d,\kk}$ is a model for polynomial representations of the group scheme $GL_{n,\Fq}$ by \cite{FS}, while the target category is closely related \cite{Kuhn2} to the category of modular representation of $GL_n(\Fq)$.} of understanding the forgetful functor $\PP_{d,\kk}\to \Fct(\Proj_\kk,\Mod)$ where $\kk$ is a finite field and $\PP_{d,\kk}$ is the category homogeneous strict polynomial functors of weight $d$. We obtain surprisingly simple and exhaustive results for the exponential variant of this problem.
\item In examples \ref{ex-dualMorava} and \ref{ex-Morava} we obtain in a simple way some basic properties of the functor  $V\mapsto \overline{K}(n)_{\overline{*}} BV$ (the cyclically graded $n$-th Morava $K$-theory of a classifying space $BV$, viewed as an endofunctor of finite $\Fp$-vector spaces). These results are analogues for $p$ odd of some results of \cite{theseQuyet}.
\end{itemize}
These are only a sample of applications of the idea of an exponential functor, but there are others. For example, a variant of theorem \ref{thm-classif-ord} appears in \cite{PV} where the idea of exponential functor plays a key role to analyze Higher Hochschild Homology.

The article ends with the short section \ref{sec-counter}, which describes a family of exponential functors exhibiting some pathologies over imperfect fields (in particular this shows that the hypothesis that $\kk$ is perfect cannot be dropped in the theorems above), as well as three appendices which recall the basics of three theories which are heavily used in the article: Frobenius twists, Dieudonn\'e modules (in the graded setting, as in \cite{Schoeller}), and polynomial functors (as defined by Eilenberg and Mac Lane \cite{EML2}).

\subsection*{Some notations used in the article}

\begin{description}
\item[$\kk$] 
is a commutative ring (often a field, sometimes a perfect field). Unadorned tensor products are taken over $\kk$. If it is a prime, the characteristic of $\kk$ is denoted by $p$.
\item[$\V$] is a small additive category (often $\V$ is $\Proj_R$, the category of projective finitely generated right modules over a ring $R$). An object of $\V$ is typically denoted by $V$ or $W$.
\item[$\mathbb{N}$] is the set of nonnegative integers, $\mathbb{Q}$ is the set of rational numbers.
\item[$\Alg$] is the category of graded $\kk$-algebras. 
Similarly $\Coalg$ is the category of graded $\kk$-coalgebras, $\H$ is the category of graded Hopf $\kk$-algebras, and $\HH$ is the category of bicommutative graded Hopf algebras. In all the article, graded means `nonnegatively graded', see section \ref{subsec-Nota-graded}.
\end{description}

\subsection*{Funding}

This work was supported by the Agence Nationale de la Recherche, via the Projet ANR ChroK  [ANR-16-CE40-0003] and the labex CEMPI [ANR-11-LABX-0007-01].

\tableofcontents

\section{Definitions}\label{sec-def}
\subsection{The functor category $\mathcal{F}$}\label{subsec-setting}
All the constructions of this article involve a certain functor category $\mathcal{F}$. We consider two possibilities for $\mathcal{F}$.
\begin{enumerate}
\item\label{it-set-1} Either $\V$ is a small additive category, $\kk$ is a commutative ring and $\mathcal{F}=\Fct(\V,\Mod)$ is the category of functors with domain $\V$ and codomain $\Mod$, and natural transformations.
\item\label{it-set-2}  Or $\kk$ is a commutative ring and $\mathcal{F}=\PP_{\omega,\kk}$ is the category of strict analytic functors over $\kk$. That is,
$$\PP_{\omega,\kk}=\prod_{d\ge 0}\PP_{d,\kk}\,,$$
where $\PP_{d,\kk}$ is the category of homogeneous strict polynomial functors of weight\footnote{Here we follow the terminology adopted in \cite{TouzeBar,TouzeFund} and we use the term `weight' instead of the more commonly used term `degree' for the integer $d$. We do this because we already use the term degree for two other notions, namely for homological gradings, as well as for Eilenberg-Mac Lane degree of functors.} $d$. If $\Proj_\kk$ denotes the category of finitely generated projective $\kk$-modules, there is a (faithful) forgetful functor:
$$\U_d\;:\;\PP_{d,\kk}\to \Fct(\Proj_\kk,\Mod)\;,$$
hence strict polynomial functors can be thought of as functors with domain $\Proj_\kk$ and codomain $\Mod$, equipped with some additional structure. Strict polynomial functors were initially defined by Friedlander and Suslin \cite{FS,SFB}, here we make a slight change by allowing functors with arbitrary $\kk$-modules as codomain. Thus $\PP_{d,\kk}$ is equivalent to the category of \emph{all} modules over the Schur algebra $S(n,d)$ if $n\ge d$. Our category $\PP_{d,\kk}$ is denoted by $\mathrm{Rep}\,\Gamma^d_\kk$ in \cite{Krause}.
\end{enumerate}
An object of  $\mathcal{F}$ is called an `ordinary functor' in the first case, a `strict analytic functor' in the second case, or simply a `functor' if we don't need to distinguish between the two settings. A functor is typically denoted by $F$ or by $F(V)$ if we want to make the variable explicit. Similarly, a morphism of functors is be denoted by $f:F\to G$ or $f_V:F(V)\to G(V)$.
 
By definition, a strict analytic functor $F$ decomposes as a (possibly infinite) direct sum of homogeneous strict polynomial functors of weight $d$, for $d\ge 0$. These are the homogeneous summands of weight $d$ of $F$, and we denote them by $w_dF$:
$$F=\bigoplus_{d\ge 0} w_dF\;.$$
Note that by definition of $\PP_{\omega,\kk}$, morphisms of strict analytic functors preserve the weight decomposition. Apart from this weight decomposition (which is specific to strict analytic functors) we will mainly use the following basic constructions and properties which hold for the two possible choices for $\mathcal{F}$.

\subsubsection{Abelian and symmetric monoidal category}
The category $\mathcal{F}$ is abelian, complete and cocomplete (limits and colimits are defined objectwise, e.g. $(\ker f)(V)=\ker f_V$ for a morphism of functors $f:F\to G$), with exact filtrant colimits. 
There is a symmetric monoidal category structure $(\mathcal{F},\otimes,\kk,\tau)$, whose tensor product is defined objectwise:
$$(F\otimes G)(V)=F(V)\otimes G(V),$$
whose unit is the constant functor $\kk$, and whose symmetry operator $\tau$ is:
$$\begin{array}[t]{cccc}
\tau :&  F(V)\otimes G(V) & \to &G(V)\otimes F(V)\\
&v\otimes w & \mapsto & w\otimes v
\end{array}.
$$

\subsubsection{Additive functors}\label{subsubsec-setting-add}
A functor $F$ is additive if the canonical morphisms $F(V)\oplus F(W)\to F(V\oplus W)$ are isomorphisms for all $V$, $W$. We let
$\mathcal{F}_{\add}$ be the full subcategory of $\mathcal{F}$ supported by the additive functors. This is a localizing subcategory of $\mathcal{F}$ (stable by subobjects, quotients, extensions and colimits). In many cases of interest, the objects of $\mathcal{F}_{\add}$ are classified and under control. In good cases, $\mathcal{F}_{\add}$ has homological dimension zero, which means that all epimorphisms are split, or equivalently that all monomorphisms are split.  

To be more specific, in the case of ordinary functors with source the category $\V=\Proj_R$ of projective finitely generated modules over a ring $R$, we can control the additive functors via the Eilenberg-Watts theorem (see fact \ref{fact-EW}), which takes the following form.
\begin{lemma}\label{lm-control-ordinary}
Let $R$ be a ring and let $\FF=\Fct(\Proj_R,\Mod)$. Evaluation on $R$ yields an equivalence of categories:
$$\begin{array}[t]{ccc}
\FF_\add & \xrightarrow[]{\simeq} & {_R}\Mod\\
F & \mapsto & F(R) 
\end{array}.$$
(Its inverse sends an $(R,\kk)$-bimodule $M$ to the functor $-\otimes_R M$.) Thus, $\mathcal{F}_{\add}$ has homological dimension zero if and only if the ring $R\otimes_\mathbb{Z}\kk$ has homological dimension zero.
\end{lemma}

One of the most elementary case where $R\otimes_\mathbb{Z}\kk$ has homological dimension zero, which is also one of the most important for our applications (especially section \ref{sec-strict-vs-ord}), is detailed in the following example.
\begin{example}\label{ex-Fadd-Fq}
If $R=\kk=\Fq$ is a field with $q=p^\ell$ elements ($p$ is a prime), then $\mathcal{F}_{\add}$ has homological dimension zero. Since $\Fq\otimes_\Z\Fq$ has finite dimension and is absolutely semisimple, one can even say more. Any object of $\mathcal{F}_{\add}$ decomposes as a direct sum of Frobenius twists functors $I^{(i)}$, $0\le i<\ell$ defined in appendix \ref{subsec-app-VF-norm} (note that $I^{(i)}\simeq I^{(j)}$ if and only if $j=i\mod \ell$ because $x^q=x$ in $\Fq$). Moreover, the Frobenius twist functors have endomorphism rings isomorphic to $\Fq$.
\end{example}

If $\FF=\PP_{\omega,\kk}$, the additive functors are classified in \cite[Section 3.1]{TouzeFund}. We recall the case where $\kk$ is a field, which is the most important for the purposes of this article.
\begin{lemma}\label{lm-control-strict}  
Assume that $\kk$ is a field, and let $\FF=\PP_{\omega,\kk}$. Then $\FF_\add$ has homological dimension zero. 

If $\kk$ has characteristic zero, then any additive strict analytic functor is isomorphic to a direct sum of copies of the functor $I=S^1=\Lambda^1=\Gamma^1$. 

If $\kk$ has positive characteristic $p$, let $I^{(r)}$ be the $r$-th Frobenius twist functor for $r\ge 0$. It is the unique additive functor of weight $p^r$ such that $I^{(r)}(\kk)$ has dimension $1$. In particular, the Frobenius twists are simple and have endomorphism rings isomorphic to $\kk$ (see appendix \ref{subsec-app-VF-strict} for a more explicit definition and further details). Then any additive strict analytic functor decomposes as a direct sum of copies of Frobenius twists. 
\end{lemma}

\begin{remark}
If $\kk$ is finite field, there is a strong similarity between the categories of additive functors of $\Fct(\PP_\kk,\Mod)$ and $\PP_{\omega,\kk}$. This similarity will be exploited in section \ref{sec-strict-vs-ord}.
\end{remark} 
 
\subsection{Notations for graded objects, Hopf monoids, etc.}\label{subsec-Nota-graded}
Let $(\C,\otimes,\kk,\tau)$ be a symmetric monoidal category, such that $\C$ is additive with all colimits and $\otimes$ commute with colimits in each argument. 
Our two main examples are $(\Mod,\otimes,\kk,\tau)$ and the functor category $(\FF,\otimes,\kk,\tau)$.

A graded object in $\C$ is an object of $\C$ equipped with a decomposition $X\simeq\bigoplus_{i\in\mathbb{N}}X^i$, and a morphism of graded objects $f:X\to Y$ is a morphism of $\C$ preserving the decompositions. We let $\C^*$ be the category of graded objects of $\C$ and their morphisms. The monoidal product $\otimes$ induces a symmetric monoidal product on $\A^*$ as usual: a summand $X^i\otimes Y^j$ of $X\otimes Y$ is placed in degree $i+j$, and the graded symmetry morphism:
$$\tau^*: X^i\otimes Y^j\to Y^j\otimes X^i$$
is defined by the Koszul rule, i.e. it equals the symmetry morphism of $\A$ up to a $(-1)^{ij}$ sign. We can define monoids, comonoids, bimonoids and Hopf monoids in $\A^*$ as usual \cite[Chap. 1]{AM}. We use the following notations:
\begin{itemize}
\item $\C-\Alg$ is the category of monoids in $\C^*$,
\item $\C-\Coalg$ is the category of comonoids in $\C^*$,  
\item $\C-\H$ is the category of Hopf monoids in $\C^*$,
\item $\C-\HH$ is the category of bicommutative Hopf monoids in $\C^*$.
\end{itemize}
The notations are suggested by the case $\C=\Mod$, where these categories are the categories of graded $\kk$-algebras, graded $\kk$-coalgebras, graded Hopf $\kk$-algebras, and graded bicommutative Hopf $\kk$-algebras. In order to avoid cumbersome notations, we will simply denote these categories by $\Alg$, $\Coalg$, $\H$ and $\HH$ when $\C=\Mod$.

If $\C=\FF$, we often refer to these categories as the categories of \emph{functorial} graded $\kk$-algebras, \emph{functorial} graded $\kk$-coalgebras and so on. This terminology is inspired by the case $\FF=\Fct(\V,\Mod)$, where (co)monoids in $\FF$ can be equivalently described as functors from $\V$ to graded $\kk$-(co)algebras.  For example, the category $\FF-\Alg$ identifies with the category $\Fct(\V,\Alg)$.

We usually don't indicate the grading in the notation, because unless explicitly specified, we consider graded objects. Instead, if we need to consider ungraded objects, we use a `$^0$' superscript (ungraded objects are identified with objects concentrated in degree zero). For example, the category of ungraded Hopf monoids in $\C$ is denoted by $\C-\H^0$. More generally, we use a variety of superscripts to indicate conditions on degrees: 
\begin{itemize}
\item `$^0$' is the full subcategory of objects concentrated in degree zero,
\item `$^\even$' is the full subcategory of objects concentrated in even degrees,
\item `$^\conn$' is the full subcategory of connected objects, i.e. which are equal to $\kk$ in degree zero,
\item `$^+$' is the full subcategory of objects which are connected and concentrated in even degrees.
\end{itemize}

\subsection{Exponential functors and functorial (co)algebras}
Let $\FF$ be either $\Fct(\V,\Mod)$ or $\PP_{\omega,\kk}$, as specified in section \ref{subsec-setting}.

\begin{definition}\label{def-exp}
An exponential functor is a graded functor $E$ (i.e. an object $E$ of $\FF^*$), equipped with an isomorphism $\phi$ of graded bifunctors and  an isomorphism $u$ of graded $\kk$-modules ($\kk$ is placed in degree zero):
$$\phi_{V,W}: E(V)\otimes E(W)\xrightarrow[]{\simeq} E(V\oplus W)\;,\qquad u:\kk\xrightarrow[]{\simeq} E(0)\;,$$
such that $\phi$ and $\eta$ are associative and unital in the obvious sense. A morphism of exponential functors $(E,\phi,u)\to (E',\phi,u)$ is a morphism of graded functors $f:E\to E'$ commuting with  $\phi$ and $u$. We denote by $\mathcal{F}-\Exp$ the category of exponential functors.
\end{definition}

\begin{remark}
An ordinary exponential functor is nothing but a strong monoidal functor from $(\V,\oplus,0)$ to $(\mathrm{Mod}_\kk^*,\otimes,\kk)$. All the results of section \ref{sec-def} except lemmas \ref{lm-precision-defbis} and \ref{lm-stabilite} remain valid if the target category is replaced by an arbitrary symmetric monoidal category. However we have chosen to stick to $(\mathrm{Mod}_\kk^*,\otimes,\kk)$ in order to treat at the same time the case of strict exponential functors. Indeed, the latter are some kind of enriched strong monoidal functors, and a uniform treatment of the two cases with a general target category would require to introduce a heavier categorical framework, unnecessary for the applications we have in mind in this article.
\end{remark}

Given an exponential functor $(E,\phi,\eta)$, we construct morphisms of graded functors $\mu$, $\eta$, $\Delta$, $\epsilon$, $\chi$ as follows. If $\delta_V:V\to V\oplus V$ is the diagonal map and $\sigma_V:V\oplus V\to V$ is the fold map, then we let
\begin{align*}
&\mu_V:=E(V)\otimes E(V)\xrightarrow[]{\phi_{V,V}}E(V\oplus V)\xrightarrow[]{E(\sigma_V)}E(V)\;,&&&&\\&\eta_V:=\kk\xrightarrow[]{u}E(0)\xrightarrow[]{E(0)} E(V)\;,\\
&\Delta_V:=E(V)\xrightarrow[]{E(\delta_V)} E(V\oplus V)\xrightarrow[]{\phi_{V,V}^{-1}}E(V)\otimes E(V)\;,&&&&\\&\epsilon_V:=E(V)\xrightarrow[]{E(0)} E(0)\xrightarrow[]{u^{-1}}\kk\;,\\
&\chi_V:=E(V)\xrightarrow[]{E(-\Id_V)} E(V)\;.
\end{align*}
It is immediate to check that $(E,\mu,\eta)$ is a functorial graded algebra and that sending an exponential functor $(E,\phi,u)$ to $(E,\mu,\eta)$ yields a functor
$$\Theta_{\Alg}:\, \mathcal{F}-\Exp\to \mathcal{F}-\Alg\;.$$
\begin{lemma}\label{lm-defbis}
The functor $\Theta_{\Alg}$ is fully faithful. 
Moreover $(A,\mu,\eta)$ lies in the image of $\Theta_{\Alg}$ if and only if the following conditions are satisfied. 
\begin{enumerate}
\item[(A1)] The $\kk$-linear map $\eta_0:\kk\to A(0)$ is an isomorphism.
\item[(A2)] For all $V$, $W$ the following composite is an isomorphism (where $\iota_V$ and $\iota_W$ are the canonical inclusions of $V$ and $W$ into $V\oplus W$) 
$$A(V)\otimes A(W)\xrightarrow[]{A(\iota_V)\otimes A(\iota_W)} A(V\oplus W)^{\otimes 2}\xrightarrow[]{\mu_{V\oplus W}} A(V\oplus W)\;.$$
\end{enumerate}

\end{lemma}
\begin{proof}  
Let $\C$ be the full subcategory of $\mathcal{F}-\Alg$ supported by the functorial algebras satisfying (A1) and (A2). If $(A,\mu,\eta)$ is an object of $\C$, let $\phi$ be the composite map appearing in condition (A2). Then $(A,\phi,\eta_0)$ is an exponential functor, and we have a functor:
$$\begin{array}[t]{cccc}
\Phi: &\C&\to &\mathcal{F}-\Exp\\
& (A,\mu,\eta)&\mapsto & (A,\phi,\eta_0)
\end{array}
\;.$$
If $(E,\phi,u)$ is an exponential functor, we have a commutative diagram:
$$\xymatrix{ 
E(V)\otimes E(W)\ar[d]^-{\phi_{V,W}}\ar[rr]^-{E(\iota_V)\otimes E(\iota_{W})}&& E(V\oplus W)^{\otimes 2}\ar[d]^-{\phi_{V\oplus W,V\oplus W}}\ar[rrd]^-{\mu_{V\oplus W}} &\\
E(V\oplus W)\ar[rr]^-{E(\iota_V\oplus\iota_{W})}\ar@/_1pc/[rrrr]_-{\Id} && E(V\oplus W\oplus V\oplus W)\ar[rr]^-{E(\sigma_{V\oplus W})}&& E(V\oplus W)
}.$$
Let $(E,\mu,\eta)=\Theta_{\Alg}(E,\phi,u)$. Then the commutativity of the outer diagram means that for $(E,\mu,\eta)$ the composition of condition (A2) equals $\phi_{V,W}$, hence it is an isomorphism. Thus the image of $\Theta_{\Alg}$ is contained in $\C$ and $\Phi\circ \Theta_{\Alg}=\Id$. 
Conversely, given $(E,\mu,\eta)$ in $\C$, the diagram
$$\xymatrix{ 
&& E(V\oplus V)\ar[rr]^-{E(\sigma_V)}&& E(V) &\\
E(V)^{\otimes 2}\ar[rru]^-{\phi_{V,V}}\ar[rr]^-{E(\iota_V)^{\otimes 2}}\ar@/_1pc/[rrrr]_-{\Id} && E(V\oplus V)^{\otimes 2}\ar[u]^-{\mu_{V\oplus W}}\ar[rr]^-{E(\sigma_{V})^{\otimes 2}}&& E(V)^{\otimes 2}\ar[u]^-{\mu_V}
}$$
commutes. The commutativity of the outer diagram means that the composite $\Theta_{\Alg}\circ\Phi$ is the identity. 
\end{proof}

Similarly $(E,\Delta,\epsilon)$ is a functorial graded coalgebra and sending an exponential functor $(E,\phi,u)$ to $(E,\Delta,\epsilon)$ yields a functor:
$$\Theta_{\Coalg}:\, \mathcal{F}-\Exp\to \mathcal{F}-\Coalg\;.$$ 
The proof of the next lemma is similar to the proof of lemma \ref{lm-defbis} and is omitted.
\begin{lemma}\label{lm-defbisbis}
The functor $\Theta_{\Coalg}$ is fully faithful. 
Moreover $(C,\Delta,\epsilon)$ lies in the image of $\Theta_{\Coalg}$ if and only if the following conditions are satisfied. 
\begin{enumerate}
\item[(C1)] The $\kk$-linear map $\epsilon_0:C(0)\to \kk$ is an isomorphism.
\item[(C2)] For all $V$, $W$ the following composite is an isomorphism (where $p_V$ and $p_W$ are the canonical projections of $V\oplus W$ onto $V$ and $W$) 
$$C(V\oplus W)\xrightarrow[]{\Delta_{V\oplus W}} C(V\oplus W)^{\otimes 2}\xrightarrow[]{C(p_V)\otimes C(p_W)}C(V)\otimes C(W)\;.$$
\end{enumerate}
\end{lemma}
We can actually weaken the conditions appearing in lemmas \ref{lm-defbis} and \ref{lm-defbisbis}.

\begin{lemma}\label{lm-precision-defbis}
Condition (A1) in lemma \ref{lm-defbis} can be replaced by the weaker condition (A1'): the $\kk$-linear map $\eta_0:\kk\to A(0)$ has a retract. Similarly, condition (C1) in lemma \ref{lm-defbisbis} can be replaced by the weaker condition (C1'): the $\kk$-linear map $\epsilon_0:C(0)\to \kk$ has a section.
In particular, if $\kk$ is a field and $A(0)\ne 0$, conditions (A1) and (C1) are superfluous.
\end{lemma}
\begin{proof}
We prove that (A1') implies (A1), the proof that (C1') implies (C1) is similar.
The map $\eta_0$ and its retract induce an isomorphism $\kk\oplus S\simeq A(0)$. Since $\eta_0$ is a unit for $A(0)$ the composite map
$$\kk\otimes A(0)\hookrightarrow (\kk\oplus S)\otimes A(0)\simeq A(0)\otimes A(0)\xrightarrow[]{\mu_0} A(0)$$
is an isomorphism. But condition (A2) implies that $\mu_0$ is an isomorphism, so we must have $S\otimes A(0)=0$. As $\kk$ is a direct summand of $A(0)$, this implies that $S=0$, hence that $\eta_0$ is an isomorphism. 
\end{proof}

We define a tensor product on $\mathcal{F}-\Alg$ and $\mathcal{F}-\Coalg$ by the usual formulas \cite[1.2.2]{AM}, \cite[VI.4]{ML}. Thus, the tensor product of two functorial graded algebras $(A_1,\mu_1,\eta_1)$ and $(A_2,\mu_2,\eta_2)$ is the functorial graded algebra 
$$(A_1\otimes A_2, (\mu_1\otimes\mu_2)\circ(\Id\otimes \tau^*\otimes\Id), \eta_1\otimes\eta_2)\;.$$
Similarly, we define the tensor product of two exponential functors $(E_1,\phi_1,u_1)$ and $(E_2,\phi_2,u_2)$ as the triple: 
$$(E_1\otimes E_2, (\phi_1\otimes\phi_2)\circ(\Id\otimes \tau^*\otimes\Id), u_1\otimes u_2)\;.$$
The next lemma is an immediate check.

\begin{lemma}\label{lm-theta-tens}
The functors $\Theta_{\Alg}$ and $\Theta_{\Coalg}$ preserve tensor products.
\end{lemma}

\subsection{Commutativity and functorial Hopf algebras}

\begin{definition}
An exponential functor $(E,\phi,u)$ is graded commutative if for all $V$, $W$ the following diagram of graded $\kk$-modules commutes (here $\tau$ is the symmetry operator in $\V$ obtained from the universal property of the direct sum). 
$$
\xymatrix{
E(V)\otimes E(W)\ar[d]^-{\tau^*}\ar[r]^-{\phi_{V,W}}& E(V\oplus W)\ar[d]^-{E(\tau)}\\
E(W)\otimes E(V)\ar[r]^-{\phi_{W,V}}&E(W\oplus V) 
}.
$$
We denote by $\mathcal{F}-\Exp_c$ the full subcategory of $\mathcal{F}-\Exp$ supported by the graded commutative exponential functors.
\end{definition}

\begin{lemma}\label{lm-commut}
Given an exponential functor $(E,\phi,u)$, the following conditions are equivalent:
\begin{enumerate}
\item[(1)] $(E,\phi,u)$ is a graded commutative exponential functor,
\item[(2)] $(E,\mu,\eta)$ is a functorial graded commutative algebra,
\item[(3)] $(E,\Delta,\eta)$ is a functorial graded cocommutative coalgebra,
\item[(4)] $(E,\mu,\eta,\Delta,\epsilon,\chi)$ is a functorial graded bicommutative Hopf algebra,
\item[(5)] $(E,\mu,\eta,\Delta,\epsilon,\chi)$ is a functorial graded Hopf algebra,
\item[(6)] $(E,\mu,\eta,\Delta,\epsilon)$ is a functorial graded bialgebra.
\end{enumerate}
\end{lemma}
\begin{proof} Condition (1), resp. (2), resp. (3) is equivalent to the morphism of functors $\tau:E\otimes E\to E\otimes E$ being a morphism of exponential functors, resp. functorial algebras, resp. functorial coalgebras. So the first three conditions are equivalent by lemmas \ref{lm-defbis}, \ref{lm-defbisbis} and \ref{lm-theta-tens}. 

Let us prove (1)$\Rightarrow$(4). Since the first three conditions are equivalent, we need to prove (i) $\chi: E\to E$ is an antipode
(ii) $\Delta:E\to E\otimes E$ is a morphism of algebras. Property (i) follows from the fact that for all $V$, $(\mu\circ (\chi\otimes \Id)\circ \Delta)_V$ equals $E(f_V)$ where $f_V=\sigma_V\circ (\Id_V,-\Id_V)\circ \delta_V=0$. Property (ii) follows from the fact that for all $V$, $((\mu\otimes\mu)\circ (\Id\otimes\tau\otimes\Id)\circ (\Delta\otimes\Delta))_V$ equals $E(g_V)$ where $g_V= (\sigma_V,\sigma_V)\circ (\Id_V,\tau,\Id_V)\circ (\delta_V,\delta_V)$, that $(\Delta\circ \mu)_V$ equals $E(h_V)$ where $h_V=\delta_V\circ\sigma_V$ and that $h_V=g_V$.
 
The implications (4)$\Rightarrow$(5)$\Rightarrow$(6) are obvious. It remains to prove (6)$\Rightarrow$(2). Since $\mu:E\otimes E\to E$ is a morphism of functorial graded coalgebras, and since $\phi^{-1}$ can be reconstructed from $\Delta$ as the composite map appearing in condition (C2) of lemma \ref{lm-defbis}, we have a commutative diagram
$$\xymatrix{
E(V\oplus W)^{\otimes 2}\ar[rr]^-{\mu_{V\oplus W}}&&E(V\oplus W)\\
(E(V)\otimes E(W))^{\otimes 2}\ar[u]^-{(\phi_{V, W})^{\otimes 2}}\ar[rr]^-{\Id\otimes\tau^*\otimes\Id}&&
E(V)^{\otimes 2}\otimes E(W)^{\otimes 2}\ar[rr]^-{\mu_{V}\otimes \mu_W}&& E(V)\otimes E(W)\ar[ull]_-{\phi_{V,W}}
}.$$
Assume that $V=W$, and let $t$, resp. $d$ be the composite map provided by the top left corner of the diagram, resp. bottom right corner of the diagram. Then $E(\sigma_V)\circ t\circ(\eta_V\otimes\Id\otimes\Id)=\mu_V$ while $E(\sigma_V)\circ d\circ(\eta_V\otimes\Id\otimes\Id)$ equals $\mu_V\circ \tau^*$, thus these two morphisms are equal, which proves (2).
\end{proof}
 
One obvious way to construct exponential functors which are not graded commutative is to take a graded commutative exponential functor which is not concentrated in even degrees, and to forget the gradings. We are grateful to Aur\'elien Djament for communicating to us the following less obvious example.
\begin{example}
Assume that $\kk$ is a field.
Let $E$ be a connected exponential functor, and let $M$ be a vector space of dimension greater or equal to $2$. We define a functorial graded algebra $A$ by letting $A^n(V)= E^n(V)\otimes M^{\otimes n}$
with product given by 
$$(e\otimes m_1\otimes\dots\otimes m_k)\cdot (e'\otimes m_1'\otimes\dots\otimes m_\ell')= (ee')\otimes m_1\otimes\dots\otimes m_k\otimes m_1'\otimes\dots\otimes m_\ell'\;.$$
Then $A$ is an exponential functor, which is not graded commutative.
\end{example}

In the remainder of the paper, we shall consider only graded commutative exponential functors (and from section \ref{sec-elementary}, all exponential functors are graded commutative without further notice, see the conventions of section \ref{subsec-conventions}).
By lemma \ref{lm-commut}, sending a graded commutative exponential functor $(E,\phi,u)$ to the tuple $(E,\mu,\eta,\Delta,\epsilon,\chi)$ yields a functor:
$$\Theta_{\H}: \mathcal{F}-\Exp_c\to \mathcal{F}-\H\;.$$
If we define a tensor product of functorial Hopf algebras by the usual formulas \cite[VI.4]{ML} or \cite[1.2.7]{AM}, then $\Theta_{\H}$ commutes with tensor products.

\begin{lemma}\label{lm-defter}
The functor $\Theta_{\H}$ is fully faithful, and its image is contained in $\mathcal{F}-\HH$. A (a priori not necessarily graded commutative or cocommutative) functorial graded Hopf algebra $H$ lies in the image of  $\Theta_{\H}$ if and only if it satisfies one of the following equivalent conditions.
\begin{enumerate}
\item[(a)] For all $V$, $W$ the following composition yields an isomorphism
$$H(V)\otimes H(W)\xrightarrow[]{H(\iota_V)\otimes H(\iota_W)} H(V\oplus W)^{\otimes 2}\xrightarrow[]{\mu_{V\oplus W}} H(V\oplus W)\;, $$
\item[(b)] For all $V$, $W$ the following composition yields an isomorphism
$$H(V\oplus W)\xrightarrow[]{\Delta_{V\oplus W}} H(V\oplus W)^{\otimes 2}\xrightarrow[]{H(p_V)\otimes H(p_W)}H(V)\otimes H(W)\;.$$
\end{enumerate}
\end{lemma}
\begin{proof}
That the image of $\Theta_{\H}$ is contained in $\mathcal{F}-\HH$ follows from lemma \ref{lm-commut}. To prove that $\Theta_{\H}$ is fully faithful, we observe that the composition:
$$ \mathcal{F}-\Exp_c\xrightarrow[]{\Theta_{\H}} \mathcal{F}-\H\xrightarrow[]{\mathrm{Forget}}\mathcal{F}-\Alg\qquad(*)$$
equals $\Theta_{\Alg}$, hence it is fully faithful by lemma \ref{lm-defbis}. Since the forgetful functor $\mathcal{F}-\H\to \mathcal{F}-\Alg$ is faithful, $\Theta_{\H}$ must be fully faithful.

Next we show that an object of $\mathcal{F}-\H$ is in the image of $\Theta_\H$ if and only if  condition (a) holds (we leave the analogous proof for (b) to the reader). 
Since the composition $(*)$ equals $\Theta_{\Alg}$, condition (a) is necessary by lemma \ref{lm-defbis}.
Conversely, let $(H,\mu,\eta,\Delta,\epsilon,\chi)$ be an object of $\mathcal{F}-\H$ satisfying condition (a). Then $(H,\mu,\eta)$ is an object of $\mathcal{F}-\Alg$ satisfying condition (A1') and (A2) of lemma \ref{lm-precision-defbis} (the retract of $\eta_0$ is $\epsilon_0$), hence there is an exponential functor $(H,\phi',u')$ such that $\Theta_\Alg(H,\phi',u')=(H,\mu,\eta)$. Let $\Delta',\epsilon',\chi'$ be the structural morphisms built from $\phi'$ and $u'$. To prove that  $(H,\mu,\eta,\Delta,\epsilon,\chi)$ is in the image of $\Theta_\H$, it remains to prove that $\Delta=\Delta'$, $\epsilon=\epsilon'$ and $\chi=\chi'$. First, $\epsilon_V,\epsilon'_V:H(V)\to \kk$ factor as $H(0)\circ \epsilon_0$ and $H(0)\circ \epsilon'_0$ and the axioms of the Hopf algebra $H(0)=\kk$ imply that $\epsilon_0=\eta_0^{-1}=\epsilon'_0$, hence $\epsilon=\epsilon'$.
Second, $\Delta:H\to H\otimes H$ is a morphism of algebras, hence a morphism of exponential functors (by lemma \ref{lm-defbis}), hence a morphism of coalgebras between $(H,\Delta',\epsilon')$ and $(H,\Delta',\epsilon')^{\otimes 2}$. Thus $\Delta=\Delta'$ by the Eckmann-Hilton argument. Third, the antipode of a Hopf algebra is uniquely determined, so $\chi=\chi'$.
Thus, satisfying condition (a) implies being in the image of $\Theta_\H$. 
\end{proof}
\begin{remark}\label{rk-uniqueDelta}
Lemma \ref{lm-defter} (or its proof) shows a uniqueness statement. Namely, if $A$ is a graded commutative functorial algebra which is exponential, then the comultiplication $\Delta$ produced from the exponential structure is the unique one such that $A$ is a functorial Hopf algebra. 
\end{remark}

The next lemma complements lemma \ref{lm-defter} by giving an elementary stability property of the image of $\Theta_\H$. 
\begin{lemma}\label{lm-stabilite}
Let $f:H\to H'$ be a morphism of functorial graded Hopf algebras. The following holds.
\begin{enumerate}
\item[(i)] If $H$ lies in the image of $\Theta_\H$ and for all $V$, $f_V:H(V)\to H'(V)$ is an epimorphism of $\kk$-modules, then $H'$ lies in the image of $\Theta_\H$.
\item[(ii)] If $H'$ lies in the image of $\Theta_\H$ and for all $V$,  $f_V:H(V)\to H'(V)$ is a monomorphism of $\kk$-modules and $H(V)$ and $H'(V)$ are $\kk$-flat, then $H$ lies in the image of $\Theta_\H$.
\end{enumerate}
\end{lemma}
\begin{proof}
We only prove (i), the proof of (ii) being similar. For all $H$ we let $a_H$ be the composite morphism of lemma \ref{lm-defter}(a), and we let $b_H$ be the composite morphism of lemma \ref{lm-defter}(b). Since $f$ is a morphism of functorial Hopf algebras, for all $V$ and $W$ we have a commutative diagram of $\kk$-modules:
$$
\xymatrix{
H(V)\otimes H(W)\ar[r]^-{a_H}\ar[d]^-{f_V\otimes f_W} & H(V\oplus W)\ar[r]^-{b_H}\ar[d]^-{f_{V\oplus W}} & H(V)\otimes H(W)\ar[r]^-{a_H}\ar[d]^-{f_V\otimes f_W} & H(V\oplus W)\ar[d]^-{f_{V\oplus W}}\\
H'(V)\otimes H'(W)\ar[r]^-{a_{H'}} & H'(V\oplus W)\ar[r]^-{b_{H'}} & H'(V)\otimes H'(W)\ar[r]^-{a_{H'}} & H'(V\oplus W)\\
}$$
If $H=\Theta_\H(E,\phi,u)$, then $a_{H}=\phi$ and $b_{H}=\phi^{-1}$, so $a_{H}\circ b_{H}$ and $b_{H}\circ a_{H}$ are identity maps. If moreover $f$ is objectwise surjective then all the vertical maps of the diagram are surjective, so that $a_{H'}\circ b_{H'}$ and $b_{H'}\circ a_{H'}$ must be identity maps too. Hence $H'$ is in the image of $\Theta_\H$ by lemma \ref{lm-defter}.
\end{proof}

\subsection{Conventions and notations for exponential functors}\label{subsec-conventions}
In the rest of the article, we only consider \emph{graded commutative} exponential functors. We therefore systematically drop the terms `graded commutative' and simply write  `exponential functor'. There are four equivalent definitions of an exponential functor, namely
\begin{itemize}
\item a triple $(E,\phi,u)$ as in definition \ref{def-exp},
\item a functorial graded algebra $(E,\mu,\eta)$ satisfying the two conditions of lemma \ref{lm-defbis},
\item a functorial graded coalgebra $(E,\Delta,\epsilon)$ satisfying the two conditions of lemma \ref{lm-defbisbis},
\item a functorial graded Hopf algebra $(E,\mu,\eta,\Delta,\epsilon,\chi)$ satisfying an additional condition as in lemma \ref{lm-defter}.
\end{itemize}
Thus an exponential functor will be denoted by $(E,\phi,u)$, $(E,\mu,\eta)$, $(E,\Delta,\epsilon)$ or $(E,\mu,\eta,\Delta,\epsilon,\chi)$ according to which point of view is more convenient. If no confusion is possible, it will simply be denoted by $E$. The morphisms $\phi,u,\mu,\eta,\Delta,\epsilon,\chi$ attached to an exponential functor $E$ will often be loosely called `the exponential structure' of $E$, or `the structure morphisms' of $E$. 

Although we try to treat in parallel the setting of strict analytic functors (i.e. when $\FF=\PP_{\omega,\kk}$) and the setting of ordinary functors (i.e. when $\FF=\Fct(\V,\Mod)$), we will sometimes have to distinguish between the two settings. In such a situation, we will  use the terms `strict exponential functors (over $\kk$)' or `ordinary exponential functors (with domain $\V$ and codomain $\kk$-modules)' to indicate in which setting the proof or the statement is valid. 

\subsection{Basic examples of exponential functors}\label{subsec-ex}

Let $M$ be a graded $\kk$-module over a commutative ring $\kk$. We let $S_{\pm}(M)$ be the free graded commutative algebra on $M$ \cite[3(c), ex 5 and 6]{FHT}. Let $S(M)$ and $\Lambda(M)$ denote the symmetric and the exterior algebras on $M$ \cite[A.2.3]{Eisenbud}. Then $S_{\pm}(M)=S(M)$ if $2=0$ in $\kk$, and $S_{\pm}(M)=S(M^{\mathrm{even}})\otimes \Lambda(M^\mathrm{odd})$ if $2$ is invertible in $\kk$.

Dually, if $\kk$ is a field and $V$ is a graded vector space, we let
$\Gamma_{\pm}(V)$ be the cofree graded cocommutative connected
coalgebra on $V$. Thus, $\Gamma_{\pm}(V)$ is the subcoalgebra of the tensor coalgebra $T^c(V)$, supported by the vector space $\bigoplus_{d\ge 0}\Gamma^{d}_{\pm}(V)$ where $\Gamma^d_{\pm}(V)=(V^{\otimes d})^{\Si_d}$ is the space of invariants under the canonical action of the symmetric group $\Si_d$ (with Koszul signs coming from the signs of the graded symmetry operator $\tau^*$). We let $\Gamma(V)$ and $\Lambda(V)$ be the cofree cocommutative connected coalgebra on $V$ and the exterior coalgebra on $V$. In characteristic $2$, one has $\Gamma_{\pm}(V)=\Gamma(V)$ while $\Gamma_{\pm}(V)=\Gamma(V^{\mathrm{even}})\otimes \Lambda(V^{\mathrm{odd}})$ in characteristic different from $2$. If $V$ has finite dimension, the graded coalgebra $\Gamma_{\pm}(V)$ is the restricted dual of the graded algebra $S_{\pm}(\Hom_\kk(V,\kk))$.

Finally we denote by $\kk G$ the $\kk$-algebra of a group $G$.

Let $A:\V\to \mathrm{Mod}_\kk^*$, and $G:\V\to\mathrm{Ab}$ be additive functors, it is easy to prove from their universal properties that the following functorial algebras or coalgebras are ordinary exponential functors:
$$S_{\pm}(A):=S_{\pm}\circ A\,,\quad  \Gamma_{\pm}(A):=\Gamma_{\pm}\circ A\,,\quad \kk G\,.$$
\begin{remark}
The letter `$\Gamma$' is supposed to remind the reader that if $V$ is a graded vector space, then $\Gamma_{\pm}(V)$ equipped with its canonical algebra structure (coming from the exponential structure) coincides with the universal divided power algebra on $V$, see \cite[Chap III]{Roby} or \cite[Exp. 7]{Cartan}.
\end{remark}

If $A$ is a graded additive strict analytic functor, then $S_{\pm}(A)$ and $\Gamma_{\pm}(A)$ have a canonical structure of a strict exponential functor. If $A$ is homogeneous of weight $p^r$, then each summand $S^i(A^{\mathrm{even}})\otimes\Lambda^j(A^{\mathrm{odd}})$ or $\Gamma^i(A^{\mathrm{even}})\otimes\Lambda^j(A^{\mathrm{odd}})$ is homogeneous of weight $p^r(i+j)$.

\section{Elementary structure over a field}\label{sec-elementary}

In this section, $\kk$ is a field and $\FF=\PP_{\omega,\kk}$ or $\Fct(\V,\Mod)$. Following the conventions of section \ref{subsec-conventions}, from this section until the end of the paper, all exponential functors are graded commutative without further notice. We give basic structure results for exponential functors. The main results are theorems \ref{thm-twist}, \ref{thm-alg} and \ref{thm-coalg}. An application of these results is given in section \ref{sec-sym}. 

\subsection{The abelian category $\mathcal{F}-\Exp_c$}
It is well-known \cite{DG,Newman,Takeuchi} that graded bicommutative Hopf algebras over a field $\kk$ yield a cocomplete abelian category $\HH$. To be more specific, the coproduct is the tensor product. The sum of two morphisms $f,g:H\to H'$ is the convolution product:
$$f\star g\;:=\; H\xrightarrow[]{\Delta}H\otimes H\xrightarrow[]{f\otimes g}H'\otimes H'\xrightarrow[]{\mu'}H'\;. $$
The zero morphism $H\to H'$ is the composition $\eta'\epsilon$, while the inverse of $f$ is the composition $f\chi$. 
The cokernel of a morphism $f:H\to H'$ is the canonical map $\pi: H'\to H'\otimes_H\kk$. The kernel of $f:H\to H'$ is supported by the kernel in the category of graded $\kk$-vector spaces of the morphism 
$$(f\otimes \Id_H)\Delta - \eta'\otimes \Id_H \;: \;H\to H'\otimes H\;.$$
For cocompleteness, one needs in addition arbitrary coproducts $\bigotimes_{\alpha\in A} H_\alpha$. Such coproducts are supported by the colimit (union) in the category of graded $\kk$-vector spaces of the finite tensor products $\bigotimes_{\alpha\in B} E_\alpha$ indexed by the finite subsets $B\subset A$, and the various inclusions of subsets $B\subset B'$.

All these constructions can be performed in the category $\mathcal{F}-\HH$ of functorial graded bicommutative Hopf algebras, which is abelian and cocomplete as well. 
Now by lemma \ref{lm-defter}, the category $\mathcal{F}-\Exp_c$ of exponential functors identifies with a full subcategory of $\FF-\HH$. 
\begin{lemma}\label{lm-abelian}
The category $\mathcal{F}-\Exp_c$ is stable under colimits and finite limits calculated in $\FF-\HH$. Thus, it is abelian and cocomplete. Moreover, the tensor product induces an equivalence of categories:
$$\begin{array}[t]{ccc}
\mathcal{F}-\Exp^0_c\;\times\; \mathcal{F}-\Exp^{\conn}_c & \xrightarrow[]{\simeq}&
\mathcal{F}-\Exp_c\\
(E, E')&\mapsto & E\otimes E'
\end{array}.$$
\end{lemma}
\begin{proof}
Any (infinite) tensor product of exponential functors is clearly an exponential functor. Stability of exponential functors by (co)kernels follows from lemma \ref{lm-stabilite}. Moreover, for all exponential functors $E$, the composition $e:E\to E^0\to E$ of the projection onto the degree zero component of $E$ and the inclusion of this degree zero component in $E$ is an idempotent of $E$, hence there is a natural isomorphism $E\simeq E^0\otimes E'$ where $E'=\ker e$ is connected, which provides the asserted categorical decomposition.
\end{proof}

An exponential functor $E$ can be seen as a functorial graded Hopf algebra, thus standard constructions of graded Hopf algebras make sense for $E$. In particular, $E$ has an augmentation ideal $\overline{E}=\ker \epsilon$, a graded subfunctor of primitives $PE$, (defined as the kernel of the map $\overline{E}\to \overline{E}^{\otimes 2}$ induced by $\Delta$), and a graded quotient functor of indecomposables $QE$ (defined as the cokernel of the map $\overline{E}^{\otimes 2}\to \overline{E}$ induced by $\mu$).
The next lemma is a direct consequence of \cite[Prop 7.21]{MilnorMoore}.

\begin{lemma}[The splitting principle]\label{lm-splitting-principle}
If $\kk$ is a field of characteristic different from $2$ let $\mathcal{F}-\Exp_c^+$, resp. $\mathcal{F}-\Exp_c^-$, be the full subcategory of $\mathcal{F}-\Exp_c^{\mathrm{conn}}$ supported by the exponential functors which are concentrated in even degrees, resp. whose primitives are concentrated in odd degrees. The tensor product induces an equivalence of categories:
$$\begin{array}[t]{ccc}
\mathcal{F}-\Exp^+_c\;\times\; \mathcal{F}-\Exp^{-}_c & \xrightarrow[]{\simeq}&
\mathcal{F}-\Exp^{\conn}_c\\
(E, E')&\mapsto & E\otimes E'
\end{array}.$$
Moreover, 
for all objects $E$ and $E'$ of $\mathcal{F}-\Exp^{-}_c$ there are isomorphisms 
$$E\simeq \Lambda(PE)\simeq \Lambda(QE)\;,\qquad 
\Hom_{\FF-\Exp^{-}_c}(E,E')\simeq \Hom_{\FF^*}(PE,PE')\;.$$
\end{lemma}
As a direct consequence of \cite[Cor 4.18 and Thm 5.18]{MilnorMoore}, one also has that connected exponential functors are quite boring in characteristic zero:
\begin{lemma}\label{lm-classif-carzero}
If $\kk$ is a field of characteristic zero, for all connected exponential functors $E$ there are isomorphisms $E\simeq S_{\pm}(PE)\simeq S_{\pm}(QE)$.
\end{lemma}

\begin{remark}\label{Rk-classif-carzero}
For strict exponential functors, lemma \ref{lm-classif-carzero} is true without any connectedness hypothesis. Indeed, define a new grading on $E$ by placing each summand $w_kE^i$ in degree $i+2k$. Let $\mathbf{E}$ be the resulting strict exponential functor. Since $\mathbf{E}$ is connected, $\mathbf{E}\simeq S_{\pm}(P\mathbf{E})\simeq S_{\pm}(Q\mathbf{E})$ by lemma \ref{lm-classif-carzero}. Coming back to the original grading (by placing each summand $w_k\mathbf{E}^i$ in degree $i-2k$) we get isomorphisms $E\simeq S_{\pm}(PE)\simeq S_{\pm}(QE)$. 
\end{remark}

We say that an additive category $\C$ has characteristic $n\in \mathbb{N}\setminus\{0\}$, if for all morphisms $f$ in $\C$ one has 
$$n\cdot f:= \underbrace{f+\dots+f}_{n\;times} =0\;,$$
and if furthermore for all $k$ such that $0<k<n$ there is a morphism $f_k\in \C$ such that $k\cdot f\ne 0$. Equivalently, $\C$ has characteristic $n$ if the characteristic of the endomorphism rings of all its objects have positive characteristic dividing $n$, and if there is an object of $\C$ whose endomorphism ring is exactly of characteristic $n$. 
For example, the category $\Proj_R$ of finitely generated projective right $R$-modules has the same characteristic as the ring $R$. Also, $\C$ has prime characteristic $p$ if and only if its $\Hom$ groups are $\Fp$-vector spaces.

\begin{lemma}\label{lm-charact}
If $\FF=\Fct(\V,\Mod)$ with $\V$ of positive characteristic $n$, or if $\FF=\PP_{\omega,\kk}$ with $\kk$ of positive (prime) characteristic $n$, then $\mathcal{F}-\Exp_c$ also has characteristic $n$. 
\end{lemma}
\begin{proof}
Let $k\cdot$ be the multiplication by $k$ in the category $\V$. By definition of the structure morphisms of $E$ we have for all integers $k$:
$$(\Id_E)^{\star k}=E(k\cdot)\qquad (*)$$ 
Thus $(\Id_E)^{\star n}=E(n\cdot)=E(0\cdot)=\eta\epsilon$ (the zero morphism of $\mathcal{F}-\Exp_c$). Thus, it remains to find an object of $\mathcal{F}-\Exp_c$ whose endomorphism ring has characteristic $n$. For strict exponential functors, such an object is provided by $S$ (the symmetric algebra of a $\kk$-vector space). 
For ordinary exponential functors, let $V$ be an object of $\V$ whose endomorphism ring has characteristic $n$. If $0<k<n$ then $k\cdot:V\to V$ is nonzero, hence the induced morphism $E_V\to E_V$ is not trivial, so that the endomorphism ring of $E_V$ has characteristic $n$. 
\end{proof}

\subsection{Some connections with additive functors}

The structure of exponential functors is tightly connected with additive functors. The next lemma is a first explanation of this fact.

\begin{lemma}\label{lm-additivite} Let $E$ be an exponential functor. The functor of primitives $PE$ is the largest additive subfunctor of $E$, and the functor of indecomposables $QE$ is the largest additive quotient of $E$.
\end{lemma}
\begin{proof}
Given two augmented coalgebras $C_1$ and $C_2$ we have $P(C_1\otimes C_2)= PC_1\oplus PC_2$. By taking $C_1=E(V)$ and $C_2=E(W)$, one obtains that $PE$ is additive. Moreover, if $A$ is a graded additive functor there is no nonzero morphism $A\to \overline{E}^{\otimes 2}$ by Pirashvili's vanishing lemma \ref{lm-cancel}. Thus all morphisms $A\to E$ factor through $PE$. The proof for $QE$ is similar.
\end{proof}

Let $E$ be a strict exponential functor over a field $\kk$ of positive characteristic $p$. Since Frobenius twist functors are additive, the composition $E^{(r)}:=E\circ I^{(r)}$ (with all structure morphisms precomposed by $I^{(r)}$) is again an exponential functor. 
Thus, for all nonnegative integers $r$ we have a functor
$$\begin{array}[t]{cccc}
-\circ I^{(r)}: & \PP_{\omega,\kk}-\Exp_c &\to &\PP_{\omega,\kk}-\Exp_c\\
& E & \mapsto & E^{(r)}
\end{array}.
$$
The degrees and weights of $E^{(r)}$ are related to the ones of $E$ by the equation: 
$$w_k(E^{(r)})^i=\begin{cases}0 & \text{if $k$ is not divisible by $p^r$,}\\
(w_\ell E^i)^{(r)} & \text{if $k=p^r\ell$.}
\end{cases}$$
In particular, $w_k(E^{(r)})=0$ for $0<k<p^r$.
\begin{theorem}\label{thm-twist}
The functor $-\circ I^{(r)}$ is fully faithful. A strict exponential functor $E$ lies in the image of $-\circ I^{(r)}$ if and only if $w_kE=0$ for $0<k<p^r$. 
\end{theorem}
\begin{proof}
That $-\circ I^{(r)}$ is fully faithful is a formal consequence of fact \ref{fact-thickF}. We have already seen that if $E$ is in the image of $-\circ I^{(r)}$, then $w_kE=0$ for $0<k<p^r$. Conversely, let $E$ be a strict exponential functor such that $w_kE=0$ for $0<k<p^r$.
The powers of $\overline{E}$ yield the augmentation filtration of $E$, with associated object: 
$$\mathrm{gr}E=\bigoplus_{n\ge 0}\overline{E}^n/\overline{E}^{n+1}\;.$$
Since $E$ is graded, each $\overline{E}^n/\overline{E}^{n+1}$ is a graded strict analytic functor, so $\mathrm{gr}E$ is a graded strict analytic functor (we don't consider the filtration grading on $\mathrm{gr}E$). Then $\mathrm{gr}E$ inherits an algebra structure from $E$, and the canonical morphism yields an epimorphism of functorial graded algebras
$$S_{\pm}(QE)=S_{\pm}(\overline{E}/\overline{E}^{2})\to\mathrm{gr}E\;.\quad(*)$$
By lemma \ref{lm-additivite} and the classification of additive strict analytic functors of lemma \ref{lm-control-strict}, $QE$ is a direct sum of Frobenius twist functors, and the hypothesis that $w_kE=0$ for $0<k<p^r$ implies that only Frobenius twist $I^{(k)}$ with $k\ge r$ appear as summands of $QE$. Since $I^{(i)}\circ I^{(r)}=I^{(r+i)}$, we have $QE=A^{(r)}$ for some graded additive functor $A$. Thus $(*)$ reinterprets as an epimorphism of graded strict analytic functors:
$$S_{\pm}(A^{(r)})=S_{\pm}(A)^{(r)}\to \mathrm{gr} E\;.\qquad(**)$$
By lemma \ref{fact-thickF}, the epimorphism $(**)$ implies that each graded strict analytic functor $\overline{E}^n/\overline{E}^{n+1}$ is of the form ${G^n}^{(r)}$ for some graded strict analytic functor $F^n$. By lemma \ref{fact-thickF} again, this implies in turn that each quotient $E/\overline{E}^n$ is of the form ${G^n}^{(r)}$ for some graded strict analytic functor $G^n$. Now $w_0\overline{E}=0$, and the multiplication preserves weights, hence $w_k\overline{E}^{n}=0$ for $k<n$. Thus $w_kE=w_k(E/\overline{E}^k)$. Hence the graded strict analytic functor $E$ equals the precomposition by $I^{(r)}$ of the graded strict analytic functor $E'=\bigoplus_{i\ge 0} w_iG^i$. To finish the proof, we observe that precomposition by the Frobenius twist functor $I^{(r)}$ being fully faithful by fact \ref{fact-thickF}, the structure morphisms of the exponential functor $E$ are obtained by precomposition by $I^{(r)}$ from uniquely determined structure morphisms for $E'$.
\end{proof}

We have seen in lemma \ref{lm-classif-carzero} that when $\kk$ is a field of characteristic zero all connected exponential functors have the form $E\simeq S_{\pm}(QE)$. The next theorem gives a characterization of exponential functors of this form in positive characteristic. This characterization is useful in practice because it is stated in terms of graded $\kk$-algebras, i.e. it does not require any information regarding comultiplication or functoriality.

\begin{theorem}\label{thm-alg} Let $\kk$ be a field of positive characteristic $p$ and let $\mathcal{F}$ be either $\mathcal{P}_{\omega,\kk}$ or $\Fct(\Proj_R,\Mod)$ for some ring $R$ of characteristic $p$. Assume that $\mathcal{F}_{\mathrm{add}}$ has homological dimension zero. 
If $E$ is a connected exponential functor and if there is a $V$ such that $E(V)$ is a free graded commutative $\kk$-algebra, then there is an isomorphism of exponential functors $E\simeq S_{\pm}(QE)$.
\end{theorem}

\begin{remark}
The assumption on  the homological dimension of $\mathcal{F}_{\mathrm{add}}$ is always satisfied in the strict setting, see section \ref{subsubsec-setting-add}. We show in section \ref{sec-ordvsHopf} that theorem \ref{thm-alg} becomes false if we remove this assumption or the assumption on the characteristic of $R$. 
\end{remark}

\noindent
{\bf Proof of theorem \ref{thm-alg}. }
If $\kk$ has odd characteristic then $E\simeq E'\otimes \Lambda(QE^{\mathrm{odd}})$ by the splitting principle of lemma \ref{lm-splitting-principle}. If $\kk$ has characteristic $2$,  we let $E=E'$. In both cases $E'$ is commutative in the ungraded sense, and in order to prove theorem \ref{thm-alg} it suffices to prove that $E'\simeq S(QE')$. 

We first prove that $E'$ is primitively generated, i.e. for all $W$ the Hopf algebra $E'(W)$ is primitively generated. We use the Frobenius $\mathbf{F}_H:{^{(1)}}H\to H$ and the Verschiebung $\mathbf{V}_H:H\to {^{(1)}}H$ recalled in appendix \ref{app-VF}. We have $\mathbf{F}_H \mathbf{V}_H= \Id_H^{\star p}$ \cite[IV \S 3 4.10]{DG} and $H$ is primitively generated if and only if $\mathbf{V}_H$ is trivial \cite[IV \S 3 Thm 6.6]{DG}. Since the endomorphism ring of $E'(W)$ has characteristic $p$ by lemma \ref{lm-charact}, it therefore suffices to prove that the Frobenius of $E'(W)$ is injective. The latter follows from the fact that the algebra $E'(W)$ is a direct summand of the free commutative algebra $E'(V^{\oplus k})\simeq E'(V)^{\otimes k}$ for $k$ sufficiently big.

Since $E'$ is primitively generated and connected, the canonical morphism $PE'\to QE'$ is an epimorphism. But $PE'$ and $QE'$ are additive by lemma \ref{lm-additivite} and the category of additive functors has homological dimension zero, so $PE'\to QE'$ has a section $s$. 
The composite morphism:
$QE'\xrightarrow[]{s}PE'\hookrightarrow E'$
induces a morphism of connected functorial graded algebras:
$f:S(QE')\xrightarrow[]{}E'$.
By construction, $Qf:QE'\to QE'$ is an isomorphism. Since $Qf$ is epi and $E'$ connected, $f$ is epi \cite[Prop 3.8]{MilnorMoore}. Since $Qf$ is mono and both the source and the target of $f$ are free graded commutative algebras, $f$ must be mono. Indeed, this follows from the fact that $S(Qf)=\mathrm{gr} f: \mathrm{gr} S(QE')\to
\mathrm{gr} E $ (where $\mathrm{gr}$ refers to the graded object associated to the augmentation filtration) is mono, and the fact that the augmentation filtration on $S(QE')$ is Hausdorff and complete in the category of connected graded algebras. 
\qed

We also have a dual theorem of \ref{thm-alg}, whose proof is similar.

\begin{theorem}\label{thm-coalg} Let $\kk$ be a field of positive characteristic $p$ and let $\mathcal{F}$ be either $\mathcal{P}_{\omega,\kk}$ or $\Fct(\Proj_R,\Mod)$ for some ring $R$ of characteristic $p$. Assume that $\mathcal{F}_{\mathrm{add}}$ has homological dimension zero. 
If $E$ is a connected exponential functor and if there is a $V$ such that $E(V)$ is a cofree connected graded cocommutative $\kk$-coalgebra, then there is an isomorphism of exponential functors $E\simeq \Gamma_{\pm}(PE)$.
\end{theorem}


\section{Application to the homology of symmetric groups}\label{sec-sym}

In this section $\kk$ is a field. We apply the results of  section \ref{sec-elementary} to the determination of the homology groups $H_*(\Si_d,V^{\otimes d})$ of the symmetric groups as representations of $GL_\kk(V)$. These representations where first described by Cohen-Hemmer-Nakano \cite{CHN}, relying on classical algebraic topology computations. Our purpose here is to simplify their study by using simple facts regarding exponential functors instead of Dyer-Lashof operations.
Let us consider the homologically graded functor of the vector space $V$: 
$$E(V)=\bigoplus_{i\ge 0} E^i(V)\quad\text{ where }\quad E^i(V)=\bigoplus_{d\ge 0} H_i(\Si_d,V^{\otimes d})\;.$$
By convention, $H_*(\Si_0,V^{\otimes 0})$ is a graded vector space of dimension one, concentrated in degree zero. We make $E(V)$ into a functorial graded algebra by letting the unit be the isomorphism $\eta:\kk\simeq H_0(\Si_0,V^{\otimes 0})$, and the multiplication be the following composition (the first map is induced by the tensor product and the second one  by the canonical inclusion $\Si_d\times \Si_e\subset \Si_{d+e}$): 
$$ H_i(\Si_d,V^{\otimes d})\otimes H_j(\Si_e,V^{\otimes e})\to H_{i+j}(\Si_d\times\Si_e,V^{\otimes d+e})\to H_{i+j}(\Si_{d+e},V^{\otimes d+e})\;.$$

\begin{lemma}\label{lm-hom-sym-exp}
The functorial algebra $E(V)$ is canonically a strict exponential functor of the variable $V$, such that $w_dE^i(V)=H_i(\Si_d,V^{\otimes d})$.
\end{lemma}
\begin{proof}
Since $V^{\otimes d}$ is a degree $d$ homogeneous strict polynomial functor of the variable $V$ and the category $\PP_{d,\kk}$ of homogeneous strict polynomial functors of degree $d$ is abelian, the $H_i(\Si_d,V^{\otimes d})$ are objects of $\PP_{d,\kk}$ and $E(V)$ is a strict analytic functorial graded algebra. 
Given two vector spaces $V$ and $W$, the subspace of $(V\oplus W)^{\otimes d+e}$ spanned by the  elementary tensors $v_1\otimes\dots\otimes v_{d+e}$ such that $d$ vectors $v_i$ belong to $V$ and $e$ vectors $v_i$ belong to $W$ identifies with the induced module $\mathrm{ind}_{\Si_d\times\Si_e}^{\Si_{d+e}}(V^{\otimes d}\otimes W^{\otimes e})$. Thus the following composition (the first isomorphism is given by the K\"unneth formula and the second one by the induction) defines an exponential structure on $E$: 
\begin{align*}\bigoplus_{d+e=k}H_*(\Si_d,V^{\otimes d})\otimes H_*(\Si_e,W^{\otimes e})&\xrightarrow[]{\simeq}\bigoplus_{d+e=k} H_*(\Si_d\times \Si_e,V^{\otimes d}\otimes W^{\otimes e})\\
&\xrightarrow[]{\simeq}\bigoplus_{d+e=k} H_*(\Si_k,\mathrm{ind}_{\Si_d\times\Si_e}^{\Si_{k}}(V^{\otimes d}\otimes W^{\otimes e}))\\
&= H_*(\Si_k,(V\oplus  W)^{\otimes k})\;.
\end{align*}
If $V=W$, the following composition (where the first map is the unit of adjunction and $\sigma_V:V\oplus V\to V$ is the fold map) is the identity map:
$$V^{\otimes d}\otimes V^{\otimes e}\hookrightarrow \mathrm{ind}_{\Si_d\times\Si_e}^{\Si_{k}}(V^{\otimes d}\otimes V^{\otimes e})\subset (V\oplus V)^{\otimes k}\xrightarrow[]{(\sigma_V)^{\otimes k}} V^{\otimes k}\;.$$
So the multiplication associated to the exponential structure on $E$ coincides with the multiplication defined before lemma \ref{lm-hom-sym-exp}, whence the result.
\end{proof}

We first prove a structure result, for which we do not need any computation at all, whereas this result is obtained in \cite{CHN} as a consequence of a full computation of $E$. As observed in \cite[Remark 8.2.2]{CHN}, several of the results of \cite{CHN} depend on theorem \ref{thm-forme-generale} rather than on the full computation of $E$.
\begin{theorem}\label{thm-forme-generale}
There is an isomorphism of strict exponential functors, where $S$ is the symmetric algebra concentrated in degree zero and $F$ is a connected strict exponential functor:
$$E\simeq S\otimes F^{(1)}\;.$$
\end{theorem}
\begin{proof}
By lemma \ref{lm-abelian}, $E$ decomposes as a tensor product $E^0\otimes E'$ where $E'$ is connected. By definition of symmetric powers $E^0=S$. Moreover, the symmetric groups $\Si_d$ have no cohomology for $d<p$ by Maschke's theorem, hence $E'$ is zero in weights $d<p$, so that $E'\simeq F^{(1)}$ by theorem \ref{thm-twist}.
\end{proof}

Now we compute $E$ explicitly. By theorem \ref{thm-forme-generale}, it suffices to compute the connected strict exponential functor $F^{(1)}$. We use the following method.
\begin{enumerate}
\item We prove an isomorphism of graded $\kk$-algebras $F^{(1)}(\kk)\simeq H_*(\Si_\infty,\kk)$. 
\item The homology of the infinite symmetric group is well-known \cite{Nakaoka}. Since it is a free graded commutative algebra, theorem \ref{thm-alg} yields an isomorphism of strict exponential functors $F^{(1)}\simeq S_{\pm}(QF^{(1)})$.
\item Finally, we use the explicit computation of \cite{Nakaoka} and the classification of additive strict analytic functors to compute $QF^{(1)}$.
\end{enumerate}
We obtain the following result.
\begin{theorem}\label{thm-calcul}
There is an isomorphism of strict exponential functors
$$F^{(1)}\simeq \bigotimes_{k\ge 1} S_{\pm}(N_{k}\otimes I^{(k)})$$
where the graded vector spaces $N_k$ have a basis given by the homogeneous elements 
$a_{(j_1,\dots,j_k)}$ of degree $j_1+\dots+j_k$, indexed by the $k$-tuples $(j_1,\dots,j_k)$ of positive integers satisfying the following three conditions:
\begin{enumerate}
\item the $j_\ell$ are congruent to $0$ or $-1$ modulo $2(p-1)$,
\item $j_1\le pj_2$, \dots, $j_{k-1}\le pj_k$,
\item $j_1>(p-1)(j_2+\dots+j_k)$.
\end{enumerate}
\end{theorem}
\begin{proof}
First, we notice that the $\Si_d$-module $\kk^{\otimes d}$ is isomorphic to the trivial $\Si_d$-module $\kk$. Hence the graded $\kk$-algebra $E(\kk)$ is nothing but the homology of symmetric groups with trivial coefficients. 
Theorem \ref{thm-forme-generale} yields an isomorphism of graded $\kk$-algebras, where $t$ is placed in degree zero:
$$E(\kk)\simeq \kk[t]\otimes F^{(1)}(\kk)\;.$$
Thus if we quotient by the ideal generated by $1-t$ we obtain an isomorphism of graded $\kk$-algebras
$$E(\kk)/\langle 1-t\rangle\simeq F^{(1)}(\kk)\;.\qquad(*)$$
On the other hand, $t$ is the basis of $S^1(\kk)=H_0(\Si_1,\kk)$, hence multiplication by $t$, considered as a morphism from $H_*(\Si_d,\kk)$ to $H_*(\Si_{d+1},\kk)$ equals the morphism induced by the standard inclusion $\Si_d=\Si_d\times\Si_1\subset \Si_{d+1}$. so we have another isomorphism of graded $\kk$-algebras 
$$E(\kk)/\langle 1-t\rangle \simeq  \colim_d H_*(\Si_d,\kk)=H_*(\Si_\infty,\kk)\qquad(**)$$ 
where the product on $H_*(\Si_\infty,\kk)$ is induced by the products
$H_i(\Si_d,\kk)\otimes H_j(\Si_e,\kk)\to H_{i+j}(\Si_{d+e},\kk)$. 
Now \cite[Thm 7.1]{Nakaoka} asserts that $H_*(\Si_\infty,\kk)$ is a connected free graded commutative algebra (the results of \cite{Nakaoka} are stated for $\kk=\Fp$, but they hold for all $\kk$ by the universal coefficient theorem). Hence $F^{(1)}(\kk)$ is also free by the isomorphisms $(*)$ and $(**)$, hence theorem \ref{thm-alg} yields an isomorphism of strict exponential functors:
$$F^{(1)}\simeq S_{\pm}(Q(F^{(1)}))= S_{\pm}(QF)^{(1)}\;.$$

It remains to identify the additive graded strict polynomial functor $QF^{(1)}$.  By lemma \ref{lm-additivite}, $QE$ is the biggest additive quotient of $E$, and by the classification of additive strict analytic functors of lemma \ref{lm-control-strict}, it has the form 
$$QE=\bigoplus_{k\ge 0} M_k\otimes I^{(k)} $$
where the $M_k$ are the graded vector spaces such that $M_k\otimes I^{(k)}$ is the biggest additive quotient of $w_{p^k}E(V)=H_*(\Si_{p^k},V^{\otimes p^k})$. Since $E\simeq S\otimes F^{(1)}$ we get
$$QF^{(1)}=\bigoplus_{k\ge 1} M_k\otimes I^{(k)} \;.$$
Thus, to finish the proof, we have to show that each $M_k$ is isomorphic to the $N_k$ described in theorem \ref{thm-calcul}. Since the $M_k$ are finite dimensional in each degree (they are quotients of $H_*(\Si_{p^k},\kk)$), this is equivalent to find for each positive $k$ an isomorphism of graded vector spaces: 
$$\bigoplus_{1\le i\le k}N_i\simeq \bigoplus_{1\le i\le k}M_i\;.\qquad(\dag)$$

Let $F_d$ be the image of the canonical map $H_*(\Si_d,\kk)\to H_*(\Si_{\infty},\kk)$ (by \cite[Prop 7.3]{Nakaoka} this map is injective, so $F_d$ is actually isomorphic to $H_*(\Si_d,\kk)$). The $F_d$, $d\ge 0$, yield an exhaustive filtration of $H_*(\Si_{\infty},\kk)$. Similarly, we define an exhaustive filtration $F'_d$, $d\ge 0$ of $E(\kk)/\langle 1-t\rangle$ by letting $F_d'$ be the image of the composite map (the first map is the inclusion provided by the definition of $E$, the second map is the projection onto the quotient):
$$ H_*(\Si_d,\kk)\hookrightarrow E(\kk)\twoheadrightarrow E(\kk)/\langle 1-t\rangle\;.$$
Then the isomorphism $(**)$ sends $F_d'$ isomorphically onto $F_d$ (in particular the composite map $H_*(\Si_d,\kk)\to E(\kk)$ is injective). Thus, the isomorphisms $(\dag)$ will be induced by the isomorphism $(**)$, provided we can prove that
\begin{enumerate}
\item[(i)]  the graded vector space $\bigoplus_{1\le i\le k}N_i$ is isomorphic to the image of $F_{p^k}$ by the canonical map $H_*(\Si_{\infty},\kk)\to QH_*(\Si_{\infty},\kk)$, and
\item[(ii)]  the graded vector space $\bigoplus_{1\le i\le k}M_i$ is isomorphic to the image of $F_{p^k}'$ by the canonical map $E(\kk)/\langle 1-t\rangle\to Q\left(E(\kk)/\langle 1-t\rangle\right)$.
\end{enumerate}
By \cite[Thm 7.1]{Nakaoka}, there is an isomorphism of graded $\kk$-algebras  
$$H_*(\Si_\infty,\kk)\simeq \bigotimes_{k\ge 1}S_{\pm}(N_k)\;.$$
Moreover, if we define a bigrading on $H_*(\Si_\infty,\kk)$ by letting the elements of $N_k$ have a second degree $p^k$, then each $F_d$ identifies with the vector space spanned by the homogeneous elements of second degree less or equal to $d$. Statement (i) easily follows from this description. In order to prove (ii), we 
choose sections of the defining maps $H_*(\Si_{p^k},\kk)\twoheadrightarrow M_k$, i.e. we view each $M_k$ as a subspace of $H_*(\Si_{p^k},\kk)\subset E(\kk)$. Now the subspace $t\cdot H_*(\Si_{p^k-1},\kk)$ of $E(\kk)$ consists of decomposable elements, hence it intersects $M_k$ trivially. Thus, in the quotient $E(\kk)/\langle 1-t\rangle$, $M_k$ is a subspace of $F'_{p^k}$ intersecting trivially $F'_{p^k-1}$. Statement (ii) follows.
\end{proof}

\section{Ordinary exponential functors and Hopf algebras}\label{sec-ordvsHopf}

We now return to the general setting for ordinary functors: $\kk$ is a commutative ring, $\V$ is a small additive category and $\mathcal{F}=\Fct(\V,\Mod)$.

Since $\kk$ is not a field, the category of graded bicommutative Hopf algebras over $\kk$ is not abelian anymore. However, it is still an additive category (the categorical (co)product is the tensor product) with all colimits (the formulas for (infinite) coproducts and cokernels hold over an arbitrary commutative ring $\kk$). Just as in lemma \ref{lm-abelian}, the structure of graded bicommutative Hopf $\kk$-algebras has the following consequence.

\begin{lemma}\label{lm-additive-split}
The category $\mathcal{F}-\Exp_c$ is additive and cocomplete. In particular, it has split idempotents. Hence any exponential functor $E$ decomposes as a tensor product $E\simeq E^0\otimes E'$ where $E'$ is connected.
\end{lemma}

The category $\FF-\HH$ of functorial graded bicommutative Hopf $\kk$-algebras identifies
with the category $\Fct(\V,\HH)$ of functors with domain $\V$ and codomain $\HH$. But $\HH$ is an additive category whose categorical (co)product is the tensor product. Thus, lemma \ref{lm-defter} can be reformulated by saying that an ordinary exponential functor $E$ is a functorial Hopf algebra such that the canonical morphism $E(V)\otimes E(W)\to E(V\oplus W)$ is an isomorphism. Note that this morphism is a isomorphism of graded bicommutative Hopf algebras, since $E(V\oplus W)$ is graded bicommutative. So we obtain the following result.

\begin{lemma}\label{lm-intelligent}
The category $\mathcal{F}-\Exp_c$ is equivalent to $\Fct_{\mathrm{add}}(\V,\HH)$.
\end{lemma} 

Recall from appendix \ref{app-Fct} that if $R$ is a ring, ${_R}\HH$ stands for the category of $R$-modules in $\HH$, whose objects are graded bicommutative Hopf $\kk$-algebras $H$ equipped with a ring morphism $\phi:R\to  \End_{\HH}(H)$, and whose morphisms are morphisms of Hopf algebras commuting with the actions of $R$. If $\V=\Proj_R$, the classification of additive functors given by the Eilenberg-Watts theorem (fact \ref{fact-EW}) yields the following result. 
\begin{theorem}\label{thm-classif-ord}
Let $R$ be a ring, let $\kk$ be a commutative ring, and let $\mathcal{F}=\Fct(\Proj_R,\Mod)$. Then evaluation on $R$ yields an equivalence of categories
$$\begin{array}[t]{cccc}
^\ord\Theta: & \FF-\Exp_c&\xrightarrow[]{\;\simeq\;} & {_R}\HH \\
& E & \mapsto & E(R) 
\end{array}\;.
$$
\end{theorem}

\begin{remark}
Alternative approaches to special cases (as well as noncommutative variants) of theorem \ref{thm-classif-ord} can be found in the literature. An approach relying on the theory of PROPs is given in \cite{Piraprop}, the case $R=\Z$ being explicitly treated in \cite[Thm 5.2 iv]{Piraprop}.
Another approach to the case $R=\Z$ can be found in \cite[Thm 5.4]{PV}.
\end{remark}

\begin{example}[Algebraic groups]
If $R=\Z$, theorem \ref{thm-classif-ord} says that the category of ungraded exponential functors with source $\Proj_R$ and target $\Mod$ is antiequivalent to the category of affine commutative group schemes over $\kk$. If $R=\Fq$, then it is antiequivalent to the category of affine commutative group schemes over $\kk$, equipped with an action of $\Fq$. Such objects naturally appear in algebraic geometry, they include affine Shtukas, see e.g. \cite{Poguntke}.
\end{example}

Though the proof of  theorem \ref{thm-classif-ord} is elementary, it has many concrete applications. In the rest of the section, we give a few immediate applications, others will be given in sections \ref{sec-strict-vs-ord} and \ref{sec-indecomp}. The first corollary shows that the hypothesis in theorem \ref{thm-alg} that $R$ has characteristic $p$  is necessary.

\begin{corollary}
Assume that $R=\Z$, and $\kk$ is a field of positive characteristic $p$. There exists an ordinary exponential functor $E$ with domain $\Proj_\Z$ and codomain $\Mod$ which is connected and such that $E(V)$ is a free graded commutative algebra, but which is not isomorphic to $S_{\pm}(QE)$.
\end{corollary}
\begin{proof}
It suffices to find a Hopf algebra which is connected, free as a graded $\kk$-algebra, but not primitively generated. An example is the Hopf algebra representing the group scheme of Witt vectors, see e.g. \cite[3.1]{Schoeller}.
\end{proof}

Recall that if $\V=\Proj_R$, then $\FF_{\mathrm{add}}$ is equivalent to the category of $R\otimes_\mathbb{Z}\kk$-modules (see lemma \ref{lm-control-ordinary}). Thus, it is easy to find rings $R$ such that $\FF_{\mathrm{add}}$ does not have homological dimension zero, the ring $R=\Fp[t]/\langle t^p-1\rangle$ (also known as the $\Fp$-group algebra of the cyclic group of order $p$) being one of the simplest examples.  
Next corollary shows that the hypothesis on the homological dimension of $\mathcal{F}_{\mathrm{add}}$ in theorem \ref{thm-alg} is necessary. 
\begin{corollary}\label{cor-cex}
Let $R=\Fp[t]/\langle t^p-1\rangle$, and let $\kk$ be a field of positive characteristic $p$. There exists a connected ordinary exponential functor $E$ with domain $\Proj_R$ and codomain $\Mod$ which is not of the form $S_{\pm}(A)$ for any additive functor $A$, but such that the graded Hopf $\kk$-algebra $E(R)$ is a polynomial algebra on primitive generators (hence the graded Hopf $\kk$-algebra $E(V)$ is a polynomial algebra on primitive generators for all $V$).
\end{corollary}
\begin{proof}
We first claim that if $(E,E')$ is a pair of ordinary exponential functors satisfying the following three conditions, there is no graded additive functor $B$ such that $E'\simeq S_{\pm}(B)$.
\begin{enumerate}
\item[(i)] $E(R)$ and $E'(R)$ are not isomorphic as $R$-Hopf modules,
\item[(ii)] $QE(R)$ and $QE'(R)$ are isomorphic as $R-\kk$-bimodules,
\item[(iii)] $E$ is isomorphic to $S_{\pm}(A)$ for some graded additive functor $A$.
\end{enumerate}
Indeed, $QE$ and $QE'$ are additive by lemma \ref{lm-additivite}, hence they are isomorphic by (ii) and the classification of additive functors of lemma \ref{lm-control-ordinary}. If $E'$ was of the form $S_{\pm}(B)$, then by (iii) $E$ and $E'$ would be isomorphic, which would contradict (i) by theorem \ref{thm-classif-ord}. 

We are now ready to construct our counter-example. Let $H=\kk[x,y]$ with $x$ and $y$ primitive of degrees $2$ and $2p$. We are going to construct a pair $(E,E')$ satisfying the three conditions above and such that $E(R)=H=E'(R)$. Then by the claim above, $E'$ will be the sought-after counterexample.

\smallskip

{\bf Construction of $E$:} there is a unique ring morphism $R\to \kk$, and we consider the ordinary exponential functor (here the brackets indicate in which degree the copies of the additive functor $V\mapsto V\otimes_R\kk$ are placed):
$$E(V)= S(V\otimes_R\kk[2]\oplus V\otimes_R\kk[2p])\;.$$ 
Then $E(R)=H$, and the $R$-module structure on $H$ is the unique ring morphism $\phi:R\to \End_{\HH}(H)$ which factors through $\Fp$. The ring morphism $\psi:R\to\End_\kk(QH)$ is the unique ring map factoring through $\Fp$.
\smallskip 

{\bf Construction of $E'$:} let $\alpha:H\to H$ be the unique morphism of graded Hopf algebras such that $\alpha(x)=x$ and $\alpha(y)=y+x^p$. We note that $\alpha\circ\dots\circ\alpha$ ($p$-times) is equal to the identity. Thus the ring map $\Fp[t]\to \End_{\HH}(H)$ sending $t$ to $\alpha$ induces an injective ring map 
$$\phi': R\hookrightarrow \End_{\HH}(H)\;.$$ 
Note that $Q\alpha=\Id_{QH}$, hence $Q\circ \phi'=\psi'$ is the unique ring map factoring through $\Fp$, that is $\psi'=\psi$. 
\end{proof}

Finally we mention an elementary application to self-duality. For all $R$-modules $V$, we let $V^\circ:=\Hom_R(V,R)$. Similarly we let $W^*=\Hom_\kk(W,\kk)$. If $E$ is an exponential functor with source $P_R$ values in finitely generated projective $\kk$-modules, we define its dual $E^\sharp$ by 
$$E^\sharp(V):= E(V^\circ)^*\;.$$
We say that $E$ is \emph{self-dual} if there is an isomorphism of exponential functors $E\simeq E^\sharp$.
If $R=\Z/n\Z$, then $_R\HH$ identifies with the full subcategory of $\HH$ supported by the Hopf algebras whose endomorphism ring has characteristic $n$. Thus, to check that $E$ is self-dual, it suffices to check that the Hopf algebra $E(R)\in\HH$ is self-dual. 
We illustrate this by two examples.
\begin{example}
Let $\kk$ be an algebraically closed field and let $R=\Z/\ell\Z$ with $\ell$ invertible in $\kk$. Then  the group ring $\kk R$ of the additive subgroup of the ring $R$ is self-dual. Thus the exponential functor $E(V)=\kk V$ is self-dual, as well as its tensor powers. Hence the standard projectives $P^n(V)=\kk (V^{\oplus n})$ of the functor category $\Fct(\Proj_{\Z/\ell\Z},\kk)$ are also injective. 
\end{example}

\begin{example}[Morava K-theory of classifying spaces]\label{ex-dualMorava}
Let $\kk=\Fp$, $p$ an odd prime. Let $\overline{K}(n)_{\overline{*}}$ be the $\Z/2(p^n-1)\Z$-graded $n$-th Morava K-theory at $p$  \cite[Def 4.2]{RavenelWilson}. For all spaces $X$ and $Y$ we have a canonical isomorphism
$$ \overline{K}(n)_{\overline{*}}X\otimes \overline{K}(n)_{\overline{*}}Y\to \overline{K}(n)_{\overline{*}}(X\times Y)\;.$$
Thus, if 
$X$ is a symmetric monoidal functor from $(\V,\oplus,0)$ to the category of spaces up to homotopy $(\mathrm{hTop},\times,\mathrm{pt})$, then $E(V)=\overline{K}(n)_{\overline{*}}X(V)$ is an exponential functor of the variable $V\in \V$, with codomain $\Fp$-vector spaces. 

Let us take $n=2$, $\V=\Proj_{\Fp}$ and $X(V)=BV$, the classifying space of the finite $\Fp$-vector space $V$. The Hopf algebra $E(\Fp)$ is computed in \cite[Thm 5.7]{RavenelWilson}. It is the Hopf algebra of lemma \ref{lm-selfdual} below, in particular it is self-dual. Thus $E$ is self-dual.
Such self-duality phenomena were first observed in \cite{theseQuyet} (for $p=2$). Our proof has the advantage of being rather short. 
\end{example}

\begin{lemma}\label{lm-selfdual} Let $H$ be a $\Z/2(p^2-1)\Z$-graded bicommutative Hopf $\Fp$-algebra. Assume that the graded vector space $H$ has a homogeneous basis $(a_0,\dots,a_{p^2-1})$, with $\deg a_i=2i$, such that: $a_p$ generates the algebra $H$, $(a_p)^p = a_1$, $(a_1)^p=0$, and for all $m$
$$\Delta(a_m)=\sum_{0\le k\le m} a_k\otimes a_{m-k}\;.$$
Then there is an isomorphism of Hopf algebras $H\simeq H^*$ which multiplies the degrees by $p$.
\end{lemma}
\begin{proof}
First, $H$ is uniquely determined up to isomorphism. Indeed, $a_0$ is the only grouplike element of $H$ so $a_0=1$. Let $y=a_p$. The surjective morphism of algebras $\kk[y]/y^{p^2}\to H$ is an isomorphism for dimension reasons. By comparing $\Delta(a_k)$ and $\Delta(a_1)^k$ we prove by induction on $k$ that $y^{pk}=a_1^k=k!a_k$, hence the comultiplication of $H$ is given by:
$$\Delta(y)=y\otimes 1+1\otimes y+\sum_{1\le k\le p-1}\frac{1}{k!(p-k)!}y^{pk}\otimes y^{p(p-k)}\;.$$
Now we prove self-duality. By uniqueness of $H$, it suffices to find a basis $(b_i)$ of $H^*$ with $\deg b_i=2pi$ and such that the $b_i$ satisfy the same conditions as the $a_i$. Take $(b_i)$, $0\le i\le p^2-1$ be dual basis of the $(y^i)$, $0\le i\le p^2-1$. 
We have $\Delta(b_m)=\sum_{k,\ell}\lambda_{k,\ell} b_k\otimes b_\ell$ with 
$$\lambda_{k,\ell}=\langle \Delta(b_m), y^k\otimes y^\ell\rangle = \langle b_m,y^{k+\ell}\rangle =\delta_{m,k+\ell}\;.$$
Thus it remains to prove that $(b_p)^p=b_1$, $(b_1)^{p}=0$ and $(b_p)^{p^2-1}\ne 0$. 
For degree reasons $(b_p)^p=\lambda b_{1}$ with
$$\lambda= \langle (b_p)^p, y\rangle = \langle b_p^{\otimes p}, \Delta^p(y)\rangle\;,$$
hence $\lambda$ is the coefficient of $(y^{p})^{\otimes k}=a_1^{\otimes k}$ in $\Delta^p(y)=\Delta^p(a_p)$, that is $\lambda=1$. Similarly, $(b_1)^{p}=0 b_p=0$. Finally the coefficient of $b_{p^2-1}$ in $(b_p)^{p^2-1}$ is (up to a nonzero scalar) the coefficient of $a_1^{\otimes p^2-1}$ in $\Delta^{p^2-1}(a_{p^2-1})$ hence it is nonzero. 
\end{proof}

\section{Strict exponential functors and Hopf algebras}\label{sec-strictvsHopf}

In this section $\mathcal{F}=\PP_{\omega,\kk}$ with $\kk$ a field of positive characteristic $p$ (strict exponential functors are completely understood in characteristic zero by lemma \ref{lm-classif-carzero} and remark \ref{Rk-classif-carzero}). 

Our purpose is to establish a comparison theorem between the strict exponential functors $E$ and the Hopf algebras $E(\kk)$ with the same flavor as theorem \ref{thm-classif-ord}.
If $E$ is a strict exponential functor, then it has a weight decomposition, hence $E(\kk)$ is canonically equipped with a weight decomposition. Moreover, by lemma \ref{lm-additivite} and the classification of additive strict analytic functors of lemma \ref{lm-control-strict}, the primitives and the indecomposables of $E$, hence of $E(\kk)$ are concentrated in weights $p^r$, $r\ge 0$. Finally, the endomorphism ring of $E$, hence of $E(\kk)$ is an $\mathbb{F}_p$-algebra by lemma \ref{lm-charact}. That is, $E(\kk)$ is a \emph{strict Hopf algebra} according to the following definition.
\begin{definition}\label{def-strict-Hopf}
Let $\kk$ be a field of characteristic $p>0$.
A \emph{strict Hopf algebra} is a graded bicommutative Hopf $\kk$-algebra $H$, whose endomorphism ring is an $\mathbb{F}_p$-algebra, equipped with a weight decomposition in each degree:
$$H^i=\bigoplus_{k\ge 0}w_kH^i$$
satisfying the following conditions.
\begin{enumerate}
\item All the Hopf algebra operations preserve the weights, e.g. the multiplication restricts to maps $w_kH^i\otimes w_\ell H^j\to w_{k+\ell}H^{i+j}$.
\item The unit and counit induce an isomorphism $\kk\simeq w_0H^0$.
\item The primitives and the indecomposables of $H$ are concentrated in weights $p^r$, $r\ge 0$.
\end{enumerate}
We denote by $^\strict\HH$ the category whose objects are the strict Hopf algebras, and whose morphisms are the morphisms of graded Hopf algebras which preserve the weight decompositions.
\end{definition}

The main result of the section is the following analogue of theorem \ref{thm-classif-ord}.
\begin{theorem}\label{thm-classif-strict}
Assume that $\kk$ is a perfect field of positive characteristic. Evaluation on $\kk$ yields an equivalence of categories
$$\begin{array}[t]{cccc}
^\strict\Theta: & \PP_{\omega,\kk}-\Exp_c&\xrightarrow[]{\;\simeq\;} & ^\strict\HH \\
& E & \mapsto & E(\kk) 
\end{array}\;.
$$
\end{theorem}

The remainder of the section is devoted to the proof of theorem \ref{thm-classif-strict}. In sections \ref{subsec-decomp} and \ref{subsec-reduction} the field $\kk$ is an arbitrary field of positive characteristic $p$. For section \ref{subsec-Gn} we need to assume that $\kk$ is perfect because we rely on the theory of Dieudonn\'e modules from \cite{Schoeller} (recalled in appendix \ref{App-Dieu}) at exactly one place, namely in the proof of lemma \ref{lm-prop-3}.

\subsection{A categorical decomposition}\label{subsec-decomp} Let $\kk$ be an arbitrary field of positive characteristic $p$. We let $\C$ be either $\PP_{\omega,\kk}-\Exp_c$ or $^\strict\HH$. Thus $\C$ is an abelian category with all colimits, by lemma \ref{lm-abelian} (the proof adapts without change for strict Hopf algebras).

We use the following notations and terminology within section \ref{subsec-decomp}. 
If $\C=\PP_{\omega,\kk}-\Exp_c$ we let $F_{r,i}$ be a copy of the Frobenius functor $I^{(r)}$ placed in degree $i$, and if $\C={^\strict\HH}$ we let $F_{r,i}$ be a $\kk$-vector space of dimension one, placed in weight $p^r$ and in degree $i$. We denote by $S_{r,i}$, resp. $\Lambda_{r,i}$, resp. $T_{r,i}$, the primitively generated symmetric algebra, resp. exterior algebra, resp. $p$-th truncated polynomial algebra, on $F_{r,i}$:
\begin{align*}
&S_{r,i}=S(F_{r,i}),\\
&\Lambda_{r,i}=\Lambda(F_{r,i}),\\
&T_{r,i}=\mathrm{coker}\,\left[\,{^{(1)}}S(F_{r,i})\xrightarrow[]{\mathbf{F}}S(F_{r,i})\right]\;.
\end{align*}
Here $\mathbf{F}$ is the Frobenius map recalled in appendix \ref{app-FV-Hopfalg}, thus $T_{r,i}$ is simply the quotient of $S(F_{r,i})$ modulo the ideal generated by $p$-th powers.
 Note that $T_{r,i}=\Lambda_{r,i}$ in characteristic $2$. If $p\ne 2$ then for graded commutativity reasons $S_{r,i}$ and $T_{r,i}$ only make sense when $i$ is even, while $\Lambda_{r,i}$ only makes sense when $i$ is odd. Note also that for $X=S$, $T$ or $\Lambda$, we have $^{(1)}X_{r,i}\simeq X_{r+1,pi}$. 
\begin{lemma}
The objects $T_{r,i}$ and $\Lambda_{r,i}$ are the only simple objects of $\C$.
\end{lemma} 
\begin{proof}The proof follows classical lines.
Observe first that $T_{r,i}$ and $\Lambda_{r,i}$ are simple. Indeed, if $X$ was a simple subobject of $T_{r,i}$, then it must have nontrivial primitives (namely its summand of lowest weight), but $PX\subset PT_{r,i}=F_{r,i}$ and since $F_{r,i}$ is simple (in the category of graded vector spaces with weights or in the category of strict analytic functors), $PX=F_{r,i}$. Since $F_{r,i}$ generate $T_{r,i}$, we have $X=T_{r,i}$. The proof for $\Lambda_{r,i}$ is the same. 
Now let $X$ be a simple object of $\C$. Then let $r$ be the lowest positive integer such that  $w_kX\ne 0$, and let $i$ be the lowest positive integer such that $w_kX^i\ne 0$. For degree and weight reasons, the elements of $w_kX^i$ are primitive, thus $w_kX^i$ is a direct sum of copies of $F_{r,i}$. Choose one of these copies, and let $f:S_\pm(F_{r,i})\to X$ be the associated morphism. If $i$ is odd, then $f$ is nonzero, hence an isomorphism, between $\Lambda_{r,i}$ and $X$. If $i$ is even, then $f$ must be zero on $F_{r,i}$ (otherwise $X$ would contain a submodule with several composition factors) thus $f$ induces a nonzero morphism, hence an isomorphism between $T_{r,i}$ and $X$.
\end{proof} 
 
We say that an object of $^\strict\HH$ is finitely generated if its underlying $\kk$-algebra is finitely generated, and that an object of $\PP_{\omega,\kk}-\Exp_c$ is finitely generated if the $\kk$-algebra $E(V)$ is finitely generated for all $V$. (Since $E(\kk^n)\simeq E(\kk)^{\otimes n}$, $E$ is finitely generated if and only if the algebra $E(\kk)$  is finitely generated.)
The next lemma is a variant of the well-known fact that connected graded bicommutative Hopf algebras over $\kk$ yield a locally noetherian category \cite{Gabriel}. 

\begin{lemma}\label{lm-locnoeth}
The category $\C$ is locally noetherian. Given an object $E$ of $\C$, the following assertions are equivalent:
\begin{enumerate}
\item\label{cen1} $E$ is noetherian,
\item\label{cen2} $E$ is finitely generated,
\item\label{cen3} $E$ has a finite composition series whose factors are $S_{r,i}$, $\Lambda_{r,i}$ or $T_{r,i}$.
\end{enumerate}
\end{lemma}
\begin{proof}
The proof is the same as in the classical case of the category $\HH$. Since we know no reference for the classical proof, we provide some details.

We already know that $\C$ is a cocomplete abelian category. Filtrant colimits are exact (use e.g. criterion given in \cite[I.6, Prop. 6 c]{Gabriel} and that this criterion can be checked in the underlying category of $\kk$-vector spaces or of strict analytic functors, where filtrant colimits are known to be exact). 

Next we examine the noetherian objects of $\C$. We first observe that any object $E$ of $\C$ has an exhaustive increasing filtration $\kk=F_0\subset F_1\subset \dots$ whose factors $F_i/F_{i-1}$ are primitively generated. (Take $F_i$ the subobject generated by   the part of weight $\le i$ of $E$, so that $F_i/F_{i-1}$ is generated by the homogeneous part of weight $i$ of $QE$). Now \eqref{cen1} holds if and only if the filtration is finite and the $F_i/F_{i-1}$ are finitely generated. Similarly, \eqref{cen2} holds if and only if the filtration is finite and the $F_i/F_{i-1}$ are noetherian. Similarly, \eqref{cen3} holds if and only if the filtration is finite and the $F_i/F_{i-1}$ satisfy \eqref{cen3}. Thus it suffices to prove \eqref{cen1}$\Rightarrow$\eqref{cen2}$\Rightarrow$\eqref{cen3} when $E$ is primitively generated, the latter being easy.

It remains to prove that the objects of $\C$ are the union of their finitely generated subobjects. First, any object of $\C$ is the union of its reflexive subobjects (an object $E$ is reflexive if for all $k$ and $i$, $w_kE^i$ is finite. In the case of $\HH$ finite means a finite dimensional vector space, while in the case of $\PP_{\omega,\kk}-\Exp_c$, this means that $w_kE^i(V)$ is a finite dimensional vector space for all $V$). This follows from \cite[Prop 4.13]{MilnorMoore}, whose proof can be adapted to the case of strict exponential functors (replace `an element $x$ of least degree in $B-A$' by `a finite functor of least degree which is not contained in $A$', and use that a functorial Hopf subalgebra of a strict exponential functor is a strict exponential functor by lemma \ref{lm-stabilite}).  Now by considering indecomposables, it is easy to see that any reflexive object $E$ is the union of finitely generated subobjects.  
\end{proof}

Let $\mathbb{N}[p^{-1}]\subset \mathbb{Q}$ be the set of nonnegative rational numbers $n/p^r$ with $r\ge 0$ and $n\in \mathbb{N}$. For all $a\in \mathbb{N}[p^{-1}]$ we let $\C(a)$ be the full subcategory of $\C$ supported by the objects $E$ whose weight decomposition satisfy:
$$w_kE^i = 0  \text{ if } i\ne ak\;.$$
Thus, the weight decomposition of an object of $\C(a)$ is determined by its grading and vice-versa. The categories $\C(a)$ are localizing subcategories of $\C$, i.e. stable under subobjects, quotients, extensions and colimits. Each object $S_{r,i}$, $\Lambda_{r,i}$, or $T_{r,i}$ belongs to only one of these categories, namely $\C(ip^{-r})$. 

The next proposition is an analogue in our context of the categorical decomposition of \cite[Section 2]{Schoeller}.

\begin{proposition}\label{prop-categ-decomp}
Tensor products induce an equivalence of categories:
$$\prod_{a\in\mathbb{N}[p^{-1}]}\C(a) \;\simeq\;\C \;.$$
\end{proposition}
\begin{proof}
Since $\C$ is locally noetherian, it suffices to prove that 
\begin{enumerate}
\item[i)] there is no nontrivial homomorphism between a noetherian object of $\C(a)$ and a noetherian object of $\C(b)$ if $a\ne b$, and
\item[ii)] any noetherian object of $\C$ can be uniquely written as a tensor product of objects of $\C(a)$, for a finite number of $a$.
\end{enumerate}
Condition i) is satisfied for weight and degree reasons. It remains to prove ii). The description of noetherian objects in lemma \ref{lm-locnoeth} shows that ii) is equivalent to: 
\begin{enumerate}
\item[iii)] we have $\Ext^1_{\C}(A,B)=0$ if $A=X_{r,i}$ and $B=Y_{s,j}$ for $i=ap^r$, $j=bp^s$ and $a\ne b$, and $X$ and $Y$ stand for $S$, $T$ or $\Lambda$.
\end{enumerate}
Let us prove iii). Let $B$ be an object of $\C(b)$ and let $A$ be a primitively generated object of $\C(a)$ with $a\ne b$. We observe that an extension 
$$\kk\to B\xrightarrow[]{\iota} E\xrightarrow[]{\pi} A\to \kk\qquad (*)$$
splits as soon as the canonical map $E\to QA$ has a section $s$. Indeed, this section $s$ induces a morphism $\overline{s}: S_{\pm}(QA)\to E$ and we can consider the commutative square (in which $\mathrm{proj}$ denotes the canonical projection):
$$\xymatrix{
B\otimes S_\pm(QA)\ar@{->>}[r]^-{\mathrm{proj}}\ar@{->>}[d]^-{\iota\otimes\overline{s}} &S_\pm(QA)\ar@{->>}[d]^-{\pi\circ\overline{s}} \\
E\ar@{->>}[r]^-{\pi} & A
}\;.$$
The morphism $\iota\otimes \overline{s}$ restricts to an identity map between $B=\ker\mathrm{proj}$ and $B=\ker\pi$. Thus by the snake lemma the morphism $\mathrm{proj}$ restricts to an isomorphism between $\ker (\iota\otimes \overline{s})$ and $\ker (\pi\circ\overline{s})$.
Hence, $\ker (\iota\otimes \overline{s})$ is an object of $\C(a)$, so the component of the map $\ker (\iota\otimes \overline{s})\to B\otimes S_\pm(QA)$ with target $B$ is trivial. Thus $E$ is isomorphic to $B\otimes A$.

If $A=\Lambda_{r,i}$, $S_{r,i}$ or $T_{r,i}$, the homogeneous summand of weight $p^r$ of $(*)$ is a short exact sequence (of graded $\kk$-vector spaces or of graded functors):
$$0\to \bigoplus_{d\ge 0}w_{p^r}B^d\xrightarrow[]{w_{p^r}\iota}\bigoplus_{d\ge 0}w_{p^r}E^d \xrightarrow[]{w_{p^r}\pi} F_{r,i}\to 0\;.
$$
Since $F_{r,i}=QA$ the canonical map $E\to QA$ has a section if and only if $w_{p^r}\pi$ has a section. But $b\ne a$, hence $w_{p^r}B^i=0$, hence $w_{p^r}\pi$ is an isomorphism between $w_{p^r}E^i$ and $F_{r,i}$. Thus the existence of the section is obvious.
\end{proof}

\subsection{Reduction of theorem \ref{thm-classif-strict} to the ungraded case}\label{subsec-reduction}
Evaluation on $\kk$ preserves the categorical decompositions of proposition \ref{prop-categ-decomp}, that is, the functor ${^\strict}\Theta$ can be written as the product of its restrictions, for all $a\in \mathbb{N}[p^-1]$:
$${^\strict}\Theta(a):\PP_{\omega,\kk}-\Exp_c(a)\to {^\strict}\HH(a)\;.$$ 
\begin{lemma}
If $p$ is odd and $a=np^{-r}$ for an odd integer $n$, then ${^\strict}\Theta(a)$ is an equivalence of categories.
\end{lemma}
\begin{proof}
The primitives of objects of $\PP_{\omega,\kk}-\Exp_c(a)$ and ${^\strict}\HH(a)$ are concentrated in odd degrees. Hence by lemma \ref{lm-splitting-principle}, ${^\strict}\Theta(a)$ is an equivalence of categories if and only if $P({^\strict}\Theta(a))$ is an equivalence of categories.  By lemma \ref{lm-additivite}, primitives are additive, thus, $P({^\strict}\Theta(a))$ is an equivalence of categories by the classification of additive strict analytic functors given in lemma \ref{lm-control-strict}.
\end{proof}

Now any $a\in\mathbb{N}[p^{-1}]$ can be written in a unique way as $a=np^{-r}$ with $r\in\mathbb{N}$, $n\in\mathbb{N}$ and $n$ is prime to $p$ if $r>0$. If $p$ is odd, we assume furthermore that $n$ is even. We define a regrading functor $\mathcal{R}_a$ as the following composition, where $\mu_n$ multiplies all the degrees of an exponential functor by $n$:
$$\PP_{\omega,\kk}-\Exp^0_c\xrightarrow[]{-\circ I^{(r)}}\PP_{\omega,\kk}-\Exp_c(p^{-r})\xrightarrow[]{\mu_n} \PP_{\omega,\kk}-\Exp_c(a)\;.$$
Thus $\mathcal{R}_aE = E^{(r)}$ as ungraded exponential functors, and the grading on $\mathcal{R}_aE$ is defined by placing each summand $(w_k E)^{(r)}$ in degree $nk$.
\begin{lemma}\label{lm-Ra}
The functor $\mathcal{R}_a$ is an equivalence of categories.
\end{lemma}
\begin{proof} The functor $-\circ I^{(r)}$ is an equivalence by theorem \ref{thm-twist}, and the functor $\mu_n$ is an equivalence (whose inverse multiplies the degrees by $\frac{1}{n}$).
\end{proof}

Similarly, if $H$ is a strict Hopf algebra concentrated in degree zero, one defines $\mathcal{R}_aH$ as the same ungraded Hopf algebras as $H$, but in which the weights and degrees are defined by 
$w_{kp^r}(\mathcal{R}_aH)^{nk}=w_kH$. Then $\mathcal{R}_a$ clearly defines an equivalence of categories:
$${^\strict}\HH^0\xrightarrow[]{\simeq}{^\strict}\HH(a)\;.$$
Moreover, we have a commutative diagram of categories:
$$\xymatrix{
\PP_{\omega,\kk}-\Exp_c(a)\ar[rr]^-{{^\strict}\Theta(a)}&&{^\strict}\HH(a)\\
\PP_{\omega,\kk}-\Exp^0_c\ar[rr]^-{{^\strict}\Theta^0}\ar[u]^-{\simeq}_-{\mathcal{R}_a}&&{^\strict}\HH^0\ar[u]^-{\simeq}_-{\mathcal{R}_a}
}.$$
As a consequence:
\begin{proposition}\label{prop-red-0}
Theorem \ref{thm-classif-strict} holds if and only if $^\strict\Theta^0$ is an equivalence of categories.
\end{proposition}

\subsection{Proof of theorem \ref{thm-classif-strict}}\label{subsec-Gn}
In this section we assume that $\kk$ is perfect. By proposition \ref{prop-red-0} we may restrict our attention to the category $\PP_{\omega,\kk}-\Exp^0_c$ of strict exponential concentrated in degree zero, or equivalently of ungraded strict exponential functors. So all our exponential functors are commutative in the ungraded sense, and implicitly graded by the weight. 
Since  $^\strict\Theta^0$ commutes with colimits, it is an equivalence of categories as a straightforward formal consequence of the following statement.

\begin{proposition}\label{prop-Gn}
There is a family of strict exponential functors $G_n$, $n\ge 0$, satisfying the following properties. 
\begin{enumerate}
\item\label{it-cond-2} For all strict exponential functors $E$ and all $n\ge 0$, evaluation on $\kk$ induces an isomorphism:
$$\Hom_{\PP_{\omega,\kk}-\Exp^0_c}(G_n,E)\xrightarrow[]{\simeq} \Hom_{{^\strict}\HH^0}(G_n(\kk),E(\kk)) \;.$$
\item\label{it-cond-3} The $G_n(\kk)$, $n\ge 0$ form a generator of ${^\strict}\HH^0$.
\item\label{it-cond-4} The $G_n$, $n\ge 0$ form a generator of $\PP_{\omega,\kk}-\Exp^0_c$.
\end{enumerate}
\end{proposition}

The remainder of section \ref{subsec-Gn} is devoted to the proof of proposition \ref{prop-Gn}. The first statement of proposition \ref{prop-Gn} is proved in lemma \ref{lm-prop-2}, the second one in lemma \ref{lm-prop-3} and the last one in lemma \ref{lm-prop-4}. 

For $n\ge 0$, we let $\Gamma_n$ be the ungraded strict exponential functor obtained as the kernel of the iterated Verschiebung 
$\mathbf{V}^n:\Gamma\to {^{(n)}}\Gamma$. Equivalently $\Gamma_n$ is the functorial subalgebra of $\Gamma$ generated by the subfunctor $\bigoplus_{d<p^n}\Gamma^d$. In particular, $Q\Gamma_n = \bigoplus_{0\le r<n}I^{(r)}$. If $\pi_n:\Gamma^{p^n}\to I^{(n)}$ denotes the quotient morphism (up to scalar multiplication $\pi_n$ is the unique nonzero-morphism between these functors), then:
$$w_i\Gamma_n =\Gamma^i\quad \text{for $i<p^n$}\,,\quad \text{and}\quad  w_{p^n}\Gamma_n = \ker \pi_n\;. $$

We are now ready to construct the exponential functors $G_n$. We first define a functorial commutative algebra $A_n$:
$$A_n(V)=\Gamma_n(V)\otimes S(\Gamma^{p^n}(V))\;.$$
Thus $w_{p^n}A_n = w_{p^n}\Gamma_n\oplus \Gamma^{p^n}$. We let $K^{p^n}$ be the kernel of the morphism 
$$\iota_n+\Id_{\Gamma^{p^n}}\,:\, w_n\Gamma_n\oplus \Gamma^{p^n}\xrightarrow[]{}\Gamma^{p^n}\;.$$
where $\iota_n$ refers to the inclusion of $w_{p^n}\Gamma_n$ into $w_{p^n}\Gamma=\Gamma^{p^n}$. Thus $K^{p^n}$ is a subfunctor of $w_{p^n}A_n$, abstractly isomorphic to $w_{p^n}\Gamma_n$.
We define the functorial algebra $G_n$ as the quotient of $A_n$ by the ideal generated by $K^{p^n}$.
\begin{lemma}\label{lm-Gn-exp}
For all $n\ge 0$,  $G_n$ is an exponential functor. 
\end{lemma}
\begin{proof}
Given a functorial algebra $A$, we let $\psi_A:A(V)\otimes A(W)\to A(V\oplus W)$ be the composite map of lemma \ref{lm-defbis}(A2). We have to show $\psi_A$ is an isomorphism if $A=G_n$.
The exponential isomorphisms of $S$, $\Gamma_n$ and $\Gamma$ induce an isomorphism of algebras, natural with respect to $V$ and $W$,  
$$\phi: A_n(V\oplus W)\simeq A_n(V)\otimes A_n(W)\otimes S(F(V,W))\;,$$
where 
$F(V,W)=\bigoplus_{0< i<p^n} \Gamma^i(V)\otimes \Gamma^{p^n-i}(W)$.
One readily checks that the ideal generated by $K^{p^n}(V\oplus W)$ is sent by $\phi$ onto the ideal generated by 
$$(K^{p^n}(V)\otimes \kk\otimes\kk)\; \oplus \; (\kk\otimes K^{p^n}(W)\otimes \kk)\;  \oplus \; (\kk\otimes\kk\otimes F(V,W))\;.$$
Moreover $\phi\circ\psi_{A_n}=\Id_{A_n(V)}\otimes \Id_{A_n(W)}\otimes \eta$, where $\eta$ is the unit of $S(F(V,W)$. Thus, $\psi_{A_n}$ induces an isomorphism $\psi_{G_n}$ on the quotient algebras.
\end{proof}

The next lemma is trivial, but it is interesting since it shows that checking exactness of sequences of exponential functors can be done from the algebra structure of $E'(\kk)$, $E(\kk)$ and $E''(\kk)$ only.

\begin{lemma}\label{lm-exact-crit}
Let $f:E'\to E$ and $g:E\to E''$ be two morphisms of exponential functors. Then $$\kk\to E'\xrightarrow[]{f}E\xrightarrow[]{g}E''\to \kk$$
is a short exact sequence of exponential functors if and only if $f_\kk:E'(\kk)\to E(\kk)$ is injective, $g_\kk:E(\kk)\to E''(\kk)$ is surjective, $g_\kk\circ f_\kk=\eta\epsilon$, and the canonical map $\kk\otimes_{E'(\kk)}E(\kk)\to E''(\kk)$ is an isomorphism.
\end{lemma}

\begin{lemma}\label{lm-exact-seq}
For all $n\ge 0$, there is a non-split exact sequence:
$$\kk\to \Gamma_n\to G_n\to S^{(n)}\to \kk\;.$$
\end{lemma}
\begin{proof}
Note that $w_n G_n=\Gamma^{p^n}$ and $w_n S^{(n)}=I^{(n)}$. But
$$\Hom_{\PP_{\omega,\kk}}(I^{(n)},\Gamma^{p^n})\subset \Hom_{\PP_{\omega,\kk}}(I^{(n)},\otimes^{p^n})\subset \Hom_{\Fct(\Proj_\kk,\Mod)}(I^{(n)},\otimes^{p^n})$$ and the latter is zero by Pirashvili's vanishing lemma \ref{lm-cancel}. So there is no injective morphism $S^{(n)}\to G_n$, hence an exact sequence as in lemma \ref{lm-exact-seq} cannot be split. Now we construct the exact sequence.
From the definition of the functorial algebra $A_n$ as a tensor product, we have a morphism of functorial algebras $\Gamma_n\to A_n$. Similarly, by using the morphisms of functorial algebras $\Gamma_n\to \kk$ and $S(\pi_n):S(\Gamma^{p^n})\to S^{(n)}$, we define a morphism of functorial algebras $A_n\to S^{(n)}$. These two morphisms induce morphisms of functorial algebras (hence of exponential functors by lemma \ref{lm-defter}) $f:\Gamma_n\to G_n$ and $g:G_n\to S^{(n)}$. 
Since $K^{p^r}(\kk)=0$, the quotient map $A_n(\kk)\to G_n(\kk)$ is an isomorphism. Thus, we have an isomorphism of algebras $G_n(\kk)=\Gamma_n(\kk)\otimes S(\Gamma^{p^r}(\kk))=\Gamma_n(\kk)\otimes S^{(n)}(\kk)$. In particular, one easily sees that $f$ and $g$ satisfy the conditions of lemma \ref{lm-exact-crit}.
\end{proof}

\begin{lemma}\label{lm-exact-seq-bis}
For all $n\ge 0$, there is a short exact sequence:
$$\kk\to S^{(n+1)}\to G_n\to \Gamma_{n+1}\to \kk\;.$$
\end{lemma}
\begin{proof}
By tensoring the inclusion of functorial algebras $\Gamma_n\hookrightarrow \Gamma_{n+1}$ and the morphism of functorial algebras $S(\Gamma^{p^n})\to \Gamma_{n+1}$ induced by the identity map $\Gamma^{p^n}\xrightarrow[]{=}w_n\Gamma_{n+1}$, we obtain a morphism of functorial algebras $A_n\to \Gamma_{n+1}$, which induces a morphism $g:G_n\to \Gamma_{n+1}$. By tensoring the inclusion of functorial algebras $S^{(1)}(\Gamma^{p^n})\hookrightarrow S(\Gamma^{p^n})$ (induced by the monomorphism $I^{(1)}(\Gamma^{p^n})\hookrightarrow S^p(\Gamma^{p^n})$) and the unit of $\Gamma_n$, we get an inclusion of functorial algebras $S^{(1)}(\Gamma^{p^n})\to A_n$. Since $S^{(1)}(\Gamma^{p^n})$ is the symmetric algebra on the $p$-th powers of $\Gamma^{p^n}$, The quotient map $A_n\to G_n$ sends $I^{(1)}(\ker \pi_n)\subset I^{(1)}(\Gamma^{p^n})$ to zero. Hence we obtain a morphism of functorial algebras $f:S^{(n+1)}\to G_n$. But $G_n(\kk)=\Gamma_n(\kk)\otimes S^{(n)}(\kk)$, $f_\kk$ is induced by the canonical inclusion of $S^{(n+1)}(\kk)$ into $S^{(n)}(\kk)$, and $g_\kk$ is the quotient $\Gamma_n(\kk)\otimes S^{(n)}(\kk)\to \Gamma_n(\kk)\otimes T^{(n)}(\kk)\simeq \Gamma_{n+1}(\kk)$, where $T$ denotes the $p$-truncated symmetric powers. Thus $f$ and $g$ clearly satisfy the conditions of lemma \ref{lm-exact-crit}.
\end{proof}

\begin{lemma}\label{lm-iso-S}
For all ungraded strict exponential functors $E$, and for all $n\ge 0$,
the map induced by evaluation on $\kk$:
$$\Ext^i_{\PP_{\omega,\kk}-\Exp^0_c}(S^{(n)},E)\to \Ext^i_{^\strict\HH^0}(S^{(n)}(\kk),E(\kk))$$
is an isomorphism if $i=0$ and is injective if $i=1$.
\end{lemma}
\begin{proof}We have commutative diagram in which the two horizontal isomorphisms are provided by the universal property of the symmetric algebra and thetwo vertical morphisms are given by evaluation on $\kk$. By lemma \ref{lm-additivite}, $PE$ is additive. Hence by the classification of additive strict analytic functors of lemma \ref{lm-control-strict}, the vertical map on the right is an isomorphism (the category $\mathrm{Vect}_{\omega,\kk}$ refers to the category of vector spaces graded by the weights and concentrated in degrees $p^r$ for $r\ge 0$).
$$\xymatrix{
\Hom_{\PP_{\omega,\kk}-\Exp^0_c}(S^{(n)},E)\ar[d]\ar[r]^-{\simeq}&
\Hom_{\PP_{\omega,\kk}}(I^{(n)},PE)\ar[d]^-{\simeq}\\
\Hom_{^\strict\HH^0}(S^{(n)}(\kk),E(\kk))\ar[r]^-{\simeq}&\Hom_{\mathrm{Vect}_{\omega,\kk}}(I^{(n)}(\kk),PE(\kk))
}\,.$$
For the injectivity in degree one, we have to prove that if an exact sequence $\kk\to E\to E'\xrightarrow[]{\pi} S^{(n)}\to \kk$ is split after evaluation on $\kk$, then it is split. But a splitting after evaluation on $\kk$ is equivalent to the existence of a map $f:S^{(n)}(\kk)\to E'(\kk)$ such that $\pi_\kk f=\Id$. By the isomorphism statement in degree $i=0$, $f$ must be of the form $f=g_\kk$ for some $g:S^{(n)}\to E'$ and we must have $\pi\circ g=\Id$.
\end{proof}

\begin{lemma}\label{lm-prop-2}
For all strict exponential functors $E$ and all $n\ge 0$, evaluation on $\kk$ induces an isomorphism:
$$\Hom_{\PP_{\omega,\kk}-\Exp^0_c}(G_n,E)\xrightarrow[]{\simeq} \Hom_{{^\strict}\HH^0}(G_n(\kk),E(\kk)) \;.$$
\end{lemma}
\begin{proof}
We proceed by induction on $n$. For $n=0$, $G_0=S$ and the isomorphism is given by lemma \ref{lm-iso-S}. We now assume that the isomorphism holds for $G_n$ for a given $n$. 
For all exact sequences  of strict exponential functors $\mathbb{F}:=\kk\to F'\to F\to F''\to \kk$, there is a comparison ladder from the $\Ext^*_{\PP_{\omega,\kk}-\Exp^0_c}(-,E)$-long exact sequence associated to $\mathbb{F}$ and the $\Ext^*_{^\strict\HH^0}(-,E(\kk))$-long exact sequence associated to $\mathbb{F}(\kk)$, in which the vertical morphisms are given by evaluation on $\kk$. 
By inspecting the ladder associated to the exact sequence of lemma \ref{lm-exact-seq-bis}, we obtain (using the induction hypothesis and lemma \ref{lm-iso-S}) an isomorphism
$$\Hom_{\PP_{\omega,\kk}-\Exp^0_c}(\Gamma_{n+1},E)\simeq \Hom_{^\strict\HH^0}(\Gamma_{n+1}(\kk),E(\kk))\;.$$
Then by inspecting the ladder associated to the exact sequence of lemma \ref{lm-exact-seq} (using the isomorphism for $\Gamma_{n+1}$ just obtained and lemma \ref{lm-iso-S}) we obtain that the isomorphism of lemma \ref{lm-prop-2} holds for $G_{n+1}$.
\end{proof}

\begin{lemma}\label{lm-prop-3}
The $G_n(\kk)$, $n\ge 0$, form a generator of $^\strict\HH^0$.
\end{lemma}
\begin{proof}
We first make the elementary observation that there is an equivalence of categories between $^\strict\HH^0$ and the category $\HH_1'$ of appendix \ref{App-Dieu}: the equivalence is given by doubling the weights of strict Hopf algebras and calling them degrees. In particular, by fact \ref{fact-Gen-dieu} for all $n\ge 0$, there is (up to isomorphism) a unique strict Hopf algebra $H_n$ which fits into a non-split short exact sequence of strict Hopf algebras
$$\kk\to \Gamma_n(\kk)\to H_n \to S^{(n)}(\kk)\to \kk\;,$$
and the $H_n$, $n\ge 0$, generate $^\strict\HH^0$. Now the non-split exact sequence of lemma \ref{lm-exact-seq}, together with lemma \ref{lm-iso-S}, imply that $G_n(\kk)=H_n$.
\end{proof}

\begin{lemma}\label{lm-prop-4}
The $G_n$, $n\ge 0$, form a generator of $\PP_{\omega,\kk}-\Exp^0_c$.
\end{lemma}
\begin{proof}
By lemma \ref{lm-prop-3}, for all objects $E$ of 
$\FF-\Exp^0_c$ we can find a strict exponential functor $G$ which is a tensor product of the $G_n$, and an epimorphism of Hopf algebras $g:G(\kk)\to E(\kk)$. 
By lemma \ref{lm-prop-2} there is a (unique) morphism of strict exponential functors $f:G\to E$ such that $f_\kk=g$. Since $f_{\kk^n}=(f_\kk)^{\otimes n}$, $f$ is an epimorphism.
\end{proof}

\section{Application to functor homology and higher torsion groups}\label{sec-bar}

In this section, we apply our results on strict exponential functors to obtain some functor homology computations, which generalize the computations of \cite{TouzeBar} over a field $\kk$.
Let us first recall the context of these computations. 
Given two strict analytic functors $F$ and $G$, homogeneous of weight $d$, we denote by $\underline{\Ext}^*(F,G)$ the $\Ext$ with parameter:
$$\underline{\Ext}^*(F,G)(V)=\Ext^*_{\PP_{d,\kk}}(F,G_V)$$
where $G_V(U)=G(V\otimes U)$. The internal $\Ext$ is a strict analytic functor homogeneous of weight $d$ \cite[section 4]{TouzeBar}. An important motivation for computing such $\Ext$ with parameters is that they come as an input in a formula computing $\Ext^*_{\PP_{dp^r,\kk}}(F^{(r)}, G^{(r)})$. The latter $\Ext$-groups appear in several fundamental problems related to the cohomology of algebraic groups and finite groups of Lie type, see \cite{TouzeBen} for a survey of the subject.

{Let $E$ be a strict exponential functor concentrated in degree zero. We define\footnote{The definition of $\mathbb{E}(\Lambda,E)$ is not the same as in \cite{TouzeBar}: the grading and the signs in the product have been modified. This new definition ensures that the resulting algebra is graded commutative, and that the results of the computations are easier to read.}
functorial graded commutative algebras $\mathbb{E}(S,E)$, $\mathbb{E}(\Lambda,E)$ by
$$w_k\mathbb{E}(S,E)^i = \underline{\Ext}^{i}(S^k,w_kE)\;,\qquad w_k\mathbb{E}(\Lambda,E)^i = \underline{\Ext}^{i-k}(\Lambda^k,w_kE)$$
and by letting the product $\bullet$ be the convolution product, up to a sign for $\mathbb{E}(\Lambda,E)$. To be more specific, if $X=S$ or $\Lambda$, we require that the following diagram (in which the vertical map is the cross product of extensions, defined by deriving the tensor product)
$$\xymatrix{
\underline{\Ext}^{i}(X^k,w_kE)\otimes \underline{\Ext}^{j}(X^\ell,w_\ell E)\ar[d]^-{\otimes}\ar[rr]^-{\bullet}&& 
\underline{\Ext}^{i+j}(X^{k+\ell},w_{k+\ell}E)\\
\underline{\Ext}^{i+j}(X^k\otimes X^\ell,w_kE\otimes w_\ell E)\ar@/_1pc/[rru]_-{\qquad\underline{\Ext}^{i+j}(\Delta_X,\mu_E)}
}$$
commutes if $X=S$, and that is commutes up to a sign $(-1)^{i\ell+k\ell}$ if $X=\Lambda$. By \cite[Prop 5.8]{TouzeBar}, $\mathbb{E}(X,E)$ is a strict exponential functor. These strict exponential functors were computed in \cite{TouzeBar} for $E=\Gamma$, they were the computations which could not be achieved in \cite{FFSS}. In the next results, we go far beyond: we compute them for an arbitrary $E$. Let $S^{(r)}$ be the symmetric algebra on a copy of $I^{(r)}$, placed in degree zero, and we let $S_n^{(r)}$ be the quotient of $S^{(r)}$ by the ideal of $p^n$-th powers.}

\begin{theorem}\label{thm-calcul-fct-hom}~Let $\kk$ be a perfect field of positive characteristic. Let $X$ be either $S$ or $\Lambda$.
Let $E$ be a strict exponential functor over $\kk$. Assume that $E$ is concentrated in degree zero, and that for all $k$ and all $V$, the vector space $w_kE(V)$ has finite dimension. The following holds. 
\begin{enumerate}
\item (Borel Theorem) There is a family of strict exponential functors $E_\alpha$ of the form $S^{(r_\alpha)}$ or $S_{n_\alpha}^{(r_{\alpha})}$, and an isomorphism of algebras preserving the weights $E(\kk)\simeq \bigotimes_\alpha E_\alpha(\kk)$.
\item There is an isomorphism of strict exponential functors:
$$\mathbb{E}(X,E)\simeq \bigotimes_\alpha \mathbb{E}(X,E_\alpha)\;.$$
\end{enumerate}
\end{theorem}

In view of theorem \ref{thm-calcul-fct-hom}, a large number of explicit $\Ext$-computations can now be deduced from the following small number of computations.

\begin{theorem}\label{thm-calcul-fct-hom2}
Let $\kk$ be a (non necessarily perfect) field of positive characteristic $p$, let $r$ be a nonnegative integer and let $n$ be a positive integer.
If $p=2$, there are isomorphisms of strict exponential functors:
\begin{align*}
&\mathbb{E}(\Lambda,S^{(r)})\simeq \Lambda(G_{r,2p^r-1})\;,\\
&\mathbb{E}(\Lambda,S_n^{(r)})\simeq 
\begin{cases}
\Gamma(G_{r,2p^r-1}) & \text{ if $n=1$,}\\
\Lambda(G_{r,2p^r-1})\otimes\Gamma(G_{r+n,2p^{r+n}-2}) & \text{ if $n>1$,}
\end{cases}
\\
&\mathbb{E}(S,S^{(r)})\simeq \Gamma(G_{r,2p^r-2} )\;,\\
&\mathbb{E}(S,S_n^{(r)})\simeq
\begin{cases}
\displaystyle\bigotimes_{k\ge 0}\Gamma(G_{r+k,2p^{r+k}-p^k}) & \text{ if $n=1$,}\\
\displaystyle\Gamma(G_{r,2p^r-2})\otimes\bigotimes_{k\ge 0}\Gamma(G_{r+k,2p^{r+n+k}-2p^k}) & \text{ if $n>1$.}
\end{cases}
\end{align*}
If $p>2$ there are isomorphisms of strict exponential functors:
\begin{align*}
&\mathbb{E}(\Lambda,S^{(r)})\simeq \Lambda(G_{r,2p^r-1})\;,\\
&\mathbb{E}(\Lambda,S_n^{(r)})\simeq \Lambda(G_{r,2p^r-1})\otimes \Gamma(G_{r+n,2p^{r+n}-2})\;,\\
&\mathbb{E}(S,S^{(r)})\simeq \Gamma(G_{r,2p^r-2} )\;,\\
&\mathbb{E}(S,S_n^{(r)})\simeq 
\begin{array}[t]{ll}
\displaystyle \Gamma(G_{r,2p^r-2})\otimes 
\bigotimes_{k\ge 0}\Lambda(G_{r+n+k,2p^{r+n+k}-2p^k-1})
\displaystyle\otimes \bigotimes_{k\ge 0}\Gamma(G_{r+n+k+1, 2p^{r+n+k+1}-2p^k-2})\;.
\end{array}\;
\end{align*}
In these formulas, each $G_{s,i}$ refers to a copy of $I^{(s)}$ placed in degree $i$. Thus if we let $\epsilon=0$ for $X=\Gamma$ and $\epsilon=1$ for $X=\Lambda$, we have: 
$$G_{s,i} \subset w_{p^s}\mathbb{E}(X,E^{(r)})^i= \underline{\Ext}^{i-\epsilon p^s}(X^{p^s}, E^{p^{s-r}\,(r)})\;.$$
\end{theorem}

The remainder of section \ref{sec-bar} is devoted to the proof of theorems \ref{thm-calcul-fct-hom} and \ref{thm-calcul-fct-hom2}. In section \ref{subsec-exHT}, we first show that (higher) torsion groups of (strict) exponential functors yield examples of (strict) exponential functors. Then in section \ref{subsec-HT} we explain how to combine theorem \ref{thm-classif-strict} with classical computations in order to compute these higher torsion groups. We then translate these higher $\Tor$ computations into functor homology computations by using a result of \cite{TouzeBar}.

\subsection{Higher torsion groups and exponential functors}\label{subsec-exHT}
Let $\kk$ be a field, and $\FF=\Fct(\V,\kk)$ or $\PP_{\omega,\kk}$. 
If $A$ and $B$ are two graded $\kk$-algebras and $M_A$, $M_B$, $_AN$ and $_BN$ are graded modules over $A$ and $B$ (the index indicates which algebra acts on which side), there is a canonical isomorphism of graded $\kk$-vector spaces (as usual, unadorned tensor products are taken over $\kk$):
$$\phi:\Tor^A(M_A,{_A}N)\otimes \Tor^B(M_B,{_B}N)\xrightarrow[]{\simeq} \Tor^{A\otimes B}(M_A\otimes M_B,{_A}N\otimes{_B}N)\;. $$
Here we use the definition of $\Tor$ as in the theory of DGAs, see e.g. \cite{Moore} or \cite[Chap 20]{FHTrat}. 
If $A$, $M_A$ and $_AN$ are ungraded (equivalently concentrated in degree zero) this is just the usual $\Tor$.
If $A_1\leftarrow A\rightarrow A_2$ are morphisms of graded commutative $\kk$-algebras, then $\Tor^A(A_1,A_2)$ is a graded commutative algebra, with product given by the composition
$$\Tor^A(A_1,A_2)^{\otimes 2}\xrightarrow[\simeq]{\phi}\Tor^{A^{\otimes 2}}(A_1^{\otimes 2},A_2^{\otimes 2})\to \Tor^A(A_1,A_2)\;,$$
the last morphism being induced by the morphisms of graded algebras $A^{\otimes 2}\to A$, $A_1^{\otimes 2}\to A_1$ and $A_2^{\otimes 2}\to A_2$ given by the multiplications. Applying this construction to exponential functors, we get the next result.
\begin{proposition}\label{prop-ex-homol-exp}
Let $E_1\leftarrow E\rightarrow E_2$ be two morphisms of exponential functors. Then 
$\Tor^{E(V)}(E_1(V),E_2(V))$ is an exponential functor of the variable $V$. 
\end{proposition}
\begin{proof}
Let $T(V)=\Tor^{E(V)}(E_1(V),E_2(V))$. By functoriality of $\Tor$ and of the canonical morphism $\phi$, the composition
$$T(V)\otimes T(W)\xrightarrow[]{T(\iota_V)\otimes T(\iota_W)}T(V\oplus W)^{\otimes 2}\xrightarrow[]{\mathrm{mult}}T(V\oplus W)$$
equals the composition:
\begin{align*}T(V)\otimes T(W)&\xrightarrow[]{\phi} \Tor^{E(V)\otimes E(W)}(E_1(V)\otimes E_1(W), E_2(V)\otimes E_2(W)\\&\to \Tor^{E(V\oplus W)}(E_1(V\oplus W), E_2(V\oplus W))
\end{align*}
where the last map is induced by the exponential isomorphisms of $E$, $E_1$ and $E_2$. Hence the graded commutative algebra $T$ is an exponential functor by lemma \ref{lm-defbis}.
\end{proof}

The graded commutative algebras $\Tor^A(\kk,\kk)$ can be computed by reduced bar constructions. We refer the reader to \cite[Chap 19]{FHTrat} and \cite[X.12]{ML} for full details on reduced bar constructions, we simply recall here the main properties. Given a commutative differential graded augmented (CDGA) algebra $A$, the reduced bar construction $\overline{B}A$ is a commutative differential graded Hopf (CDGH) algebra. The product on $\overline{B}A$ is constructed in a way similar to the product on $\Tor$, namely there is a canonical comparison map (which is a morphism lifting $\phi$ to the chain level)
$$\overline{\phi}: \overline{B}(A_1)\otimes \overline{B}(A_2)\to \overline{B}(A_1\otimes A_2)$$
and the product is induced by $\overline{\phi}$ and the morphism of CDGA-algebras $A^{\otimes 2}\to A$ given by the multiplication of $A$. The coproduct is the deconcatenation product. (To see that the product and coproduct yield a Hopf algebra structure, use that as a graded $\kk$-module $\overline{B}A$ is the tensor algebra on the augmentation ideal of $A$. The product is also described as the shuffle product \cite[X Thm 12.2]{ML}, hence if we ignore its differential, $\overline{B}A$ is simply the shuffle Hopf algebra). Now since $\overline{B}A$ is itself a CDGA-algebra, this construction can be iterated: we let $\overline{B}^0A=A$ and for $j\ge 1$ we let $\overline{B}^jA= \overline{B}(\overline{B}^{j-1}A)$. We denote the homology of $\overline{B}^jA$ by $\Tor^A_{[j]}(\kk,\kk)$. This is a graded commutative Hopf algebra, natural with respect to the CDGA-algebra $A$. We have the following higher version of proposition \ref{prop-ex-homol-exp}.
\begin{proposition}\label{prop-ex-homol-exp2}
Let $E$ be an exponential functor. For all $i>0$, the graded Hopf algebra $\Tor^{E(V)}_{[j]}(\kk,\kk)$ is an exponential functor of the variable $V$.
\end{proposition}
\begin{proof}
Condition (a) of lemma \ref{lm-defter} is satisfied. 
The proof is similar to that of proposition \ref{prop-ex-homol-exp}, but relies on (iterated) uses of $\overline{\phi}$ instead of $\phi$.
\end{proof}

Propositions \ref{prop-ex-homol-exp} and \ref{prop-ex-homol-exp2} have many variants. For example, $\Tor$ can be replaced by $\mathrm{Cotor}$ and reduced bar constructions by cobar constructions. There are also analogues for $\Ext$, for cyclic bar constructions, etc, with analogous proofs. We leave this to the reader.
Our focus on reduced bar constructions is motivated by the application to functor homology. In view of this application, we need an additional property of  bar constructions relative to (possibly infinite) tensor products of CDGA-algebras. 
\begin{lemma}\label{lm-comp-infinite-prod}
If there is a decomposition of a CDGA-algebra $A$ as a tensor product $A=\bigotimes_\alpha A_\alpha$ then for all positive $j$, there is a quasi-isomorphism of CDGH-algebras
$$\bigotimes_\alpha\overline{B}^j A_\alpha \to \overline{B}^j A\;.$$
If $A$ and the $A_\alpha$ have an additional grading by weights (such that products and differentials preserve the weights), and if the tensor product decomposition preserves the weights, then the bar constructions have a canonical additional grading by weights, and the quasi-isomorphism preserves the weights.
\end{lemma}
\begin{proof}
It suffices to prove the result for $j=1$, the general result then follows by an easy induction on $j$. If $j=1$ and $A=A_1\otimes A_2$, then a suitable quasi-isomorphism is provided by the comparison map. Indeed by functoriality, the comparison map equals the composition 
$$\overline{B}A_1\otimes \overline{B}A_2\to \overline{B}(A_1\otimes A_2)^{\otimes 2}\xrightarrow[]{\mathrm{mult}} \overline{B}(A_1\otimes A_2)$$
where the first map is provided by inclusions of $A_1$ and $A_2$ in $A_1\otimes A_2$. Both maps are morphisms of CDGH-algebras, hence the comparison map is a morphism of CDGH-algebras. It preserves the weights if all algebras have an additional grading by weights. Furthermore the comparison map is known to be a quasi-isomorphism. By iterating this, we obtain suitable quasi-isomorphism for finite tensor products. Taking colimits with respect to finite tensor subproducts yield the result for arbitrary tensor products.
\end{proof}

\subsection{Some explicit higher $\Tor$ computations}\label{subsec-HT}
Now we specialize to $\FF=\PP_{\omega,\kk}$ and $\kk$ is a perfect field of positive characteristic $p$. 
\begin{theorem}\label{thm-calcul-B}
Let $E,E_\alpha$ be strict exponential functors over a perfect field $\kk$ of positive characteristic. Each isomorphism of graded algebras 
$E(\kk)\simeq\bigotimes_\alpha E_\alpha(\kk)$ preserving the weights induces isomorphisms of strict exponential functors for $j>0$:
$$\Tor^{E}_{[j]}(\kk,\kk)\simeq \bigotimes_{\alpha}\Tor^{E_\alpha}_{[j]}(\kk,\kk)\;.$$
\end{theorem}
\begin{proof}
Lemma \ref{lm-comp-infinite-prod} shows that the two sides of the asserted isomorphism 
are isomorphic as strict Hopf algebras after evaluation on $\kk$. Hence the result follows from theorem \ref{thm-classif-strict}.
\end{proof}

In the next statements, for all nonnegative integers $i$, $r$, we let $F_{r,i}$ be a copy of $I^{(r)}$ placed in degree $i$, and for all positive $n$ we let $S_n(F_{r,i})$ be the quotient of $S(F_{r,i})$ by the ideal of $p^n$-th powers. In order that our exponential functors are graded commutative as usual, we always implicitly impose that $i$ is even if $p$ is odd in the sequel. Similarly, when considering $\Lambda(F_{r,i})$ we implicitly impose that $i$ is odd if $p$ is odd. 
 
\begin{proposition}\label{prop-Borel}
Let $E$ be a strict exponential functor over a perfect field $\kk$ of positive characteristic $p$ such that each homogeneous component $w_kE^i(\kk)$ is finite dimensional. Then there is an isomorphism of graded algebras $E(\kk)\simeq \bigotimes_\alpha E_\alpha(\kk)$ preserving the weights, where each $E_\alpha$ is a strict exponential functor of the form $S(F_{r,i})$, $S_n(F_{r,i})$ or $\Lambda(F_{r,i})$ for some nonnegative integers $r$, $i$ depending on $\alpha$.
\end{proposition}
\begin{proof}
By proposition \ref{prop-categ-decomp}, we reduce ourselves to the case where there is an  $a$ such $w_kE^i(\kk)=0$ for $i\ne ak$. Thus we may simply consider $E(\kk)$ as a connected graded algebra, with grading given by the weights. The decomposition is then provided by the theorem of Borel \cite[Thm 7.11]{MilnorMoore}.
\end{proof}

Proposition \ref{prop-Borel} and theorem \ref{thm-calcul-B} essentially reduce higher $\Tor$ computations of strict exponential functors to those of $S_n(F_{r,i})$, $S(F_{r,i})$, and $\Lambda(F_{i,r})$ over a perfect field of characteristic $p$. In the rest of section \ref{subsec-HT}, we explain how to compute them. The approach is as follows. 
The results as graded algebras are classical computations. Our task is to compute the weights in addition. Then theorem \ref{thm-classif-strict} yields computations of strict exponential functors. As we use theorem \ref{thm-classif-strict} to recover functoriality, we need to assume that $\kk$ is perfect when doing the computations. However, the computations will hold over any field by base change (see remark \ref{rk-imperfect}).

\begin{proposition}\label{prop-cal1}
Let $r,i$ be nonnegative integers, and let $n$ be a positive integer. Then
\begin{align*}
&\Tor^{S(F_{r,i})}(\kk,\kk)=\Lambda(F_{r,i+1}),\\
&\Tor^{\Lambda(F_{r,i})}(\kk,\kk)=\Gamma(F_{r,i+1}),\\
&\Tor^{S_n(F_{r,i})}(\kk,\kk)=\Lambda(F_{r,i+1})\otimes \Gamma(F_{r+n,p^{n-1}i+2})\text{ if $p$ is odd or if $n>2$.}
\end{align*}
\end{proposition}

\begin{proof}
It suffices to do the computations for $r=0$, the case $r>0$ follows directly by precomposing by $I^{(r)}$ the equalities obtained for $r=0$. 
The first two computations are already explained in \cite[Section 7.2]{TouzeBar}, they are simple consequences of the theory of Koszul duality of algebras, so we concentrate on the last computation, which is the difficult part. Let $T_{n,i}=\Tor^{S_n(F_{0,i})}(\kk,\kk)$. 

We first compute $T_{n,0}$. Recall \cite[3.2]{Evens} the computation of the cohomology algebra of $G=\Z/p^n\Z$ with coefficients in $\kk$. If $p$ is odd, or if $n\ge 2$:
$$H^*(G)\simeq S(H^{2}(G))\otimes\Lambda(H^1(G))\;.$$
This is actually an isomorphism of primitively generated Hopf algebras (the comultiplication on on $H^*(G)$ is the map $H^*(+_G): H^{*}(G)\to H^*(G\times G)\simeq H^*(G)^{\otimes 2}$). Dually, there is an isomorphism of graded Hopf algebras:
$$\Tor^{\kk G}(\kk,\kk)=H_*(G)\simeq\Lambda(H_1(G))\otimes \Gamma(H_2(G))\;.$$
The isomorphism of algebras $\kk G\simeq \kk[x]/x^{p^n}$ then implies that there is an isomorphism of graded Hopf algebras:
$$T_{n,0}(\kk)\simeq \Lambda(\kk[1])\otimes\Gamma(\kk[2])\;.\qquad (*)$$

By theorem \ref{thm-classif-strict}, in order to compute $T_{n,0}$ it remains to show that the homogeneous summand of degree $1$, resp. $2$, of $T_{n,0}(\kk)$ has weight $1$, resp. $p^n$. Let $I_n$ be the augmentation ideal of $S_n(\kk)=\kk[x]/x^{p^n}$. The complex  $C_n=\overline{B}(S_n(\kk))$ has the form
$$ \dots\to I_n^{\otimes j+1}\to I_n^{\otimes j}\to \dots \to I_n^{\otimes 2}\xrightarrow[]{\mathrm{mult}} I_n\xrightarrow[]{0}\kk $$
with each $I_n^{\otimes j}$ placed in degree $j$. Note that $I_n$ is graded by the weight ($x^k$ is homogeneous of weight $k$). The differentials of $C_n$ are defined from the multiplication of $\kk[x]/x^{p^n}$, so they preserve the weights. So $C_n$ splits as a direct sum of complexes $w_kC_n$, for $k\in \mathbb{N}$, each $w_kC_n$ being homogeneous of weight $k$. Since $w_kI_n^{\otimes j}=0$ for $k<j$, each $w_kC_n$ with $k$ positive is a finite complex of the form:
$$0\to w_kI_n^{\otimes k}\to w_kI_n^{\otimes k-1}\to \dots \to w_kI_n^{\otimes 2}\xrightarrow[]{\mathrm{mult}} w_kI_n\to 0\;.$$
In particular $w_1(C_n)$ is equal to $\kk$ concentrated in degree $1$, hence $H_1(C_n)$ has weight $1$. 
Also, $H_i(C_n)$ contains no element of weight $1$ if $i\ne 1$, thus $H_2(C_n)$ must have weight $p^r$ with $r>0$. We deduce from this and the isomorphism $(*)$ that for all $k\ge 0$ and all $n\ge 1$:
$$\sum_{i\ge 0}\dim_\kk w_k H_i(C_n)\le 1\;.\qquad(**)$$
Now we determine $H_2(C_n)$ by induction on $n$. Note that $I_1^{\otimes 2}$ has no element of weight less than $2$ or greater than $2p-2$. Since $H_2(C_1)$ must be homogeneous of weight $p^r$ for some $r$, the only possibility is $r=1$. Assume that we have proved that $H_2(C_n)$ has weight $p^n$ for a given value of $n$. Because of the isomorphism $(*)$, we then have
$$\chi(w_{p^r}C_n)=
\begin{cases}
0 & \text{if $0<r<n$,}\\
1 & \text{if r=n.}
\end{cases}$$
Observe that $I_{n+1}^{\otimes 2}$ has no element of weight less than $2$ or greater than $2p^{n+1}-2$, thus the possible weight for $H_2(C_{n+1})$ is $p^r$ with $0<r\le n+1$.
Since $w_kI_n^{\otimes j}=w_kI_{n+1}^{\otimes j}$ for $k< p^{n}$ and all $j$, we have
$$\chi(w_{p^r}C_{n+1}) =  \chi(w_{p^r}C_n) = 0$$
for $0< r<n$. Similarly, we have:
$$\chi(w_{p^{n}}C_{n+1}) =  \chi(w_{p^{n}}C_n) -1 = 0\;.$$
Thus $(**)$ implies that $w_{p^r}H_*(C_{n+1})$ is zero for $0<r\le n$. Hence the weight of  $H_2(C_{n+1})$ must be $p^{n+1}$. This finishes the computation of $T_{n,0}$.

To finish the proof, we must compute $T_{n,i}$ for $i>0$. Note that $T_{n,i}$ must be isomorphic to $T_{n,0}$ up to regrading. Hence, one must have 
$$T_{n,i}\simeq \Lambda(F_{0,a})\otimes \Gamma(F_{n,b})$$
and it remains to determine the two integers $a$ and $b$. Let $A$ be the graded algebra $S_n(F_{0,i}(\kk))$, and let $x$ be a generator of $A$ (thus $x$ is homogeneous of degree $i$). By using the usual projective resolution of $\kk$ in graded $A$-modules (the brackets indicate a copy of $A$ shifted in degrees):
$$ \dots \xrightarrow[]{x} A[p^{n-1} i]\xrightarrow[]{x^{p^{n-1}}}A[i]\xrightarrow[]{x}A  $$
we easily compute the graded vector space $T_{n,i}(\kk)$ and we see that $a=i+1$ and $b=p^{n-1} i+2$, which finishes the proof.
\end{proof}

The following explicit computation was already obtained in \cite{TouzeBar} by different methods.

\begin{proposition}\label{prop-cal2}Let $r,i$ be nonnegative integers with $i$ even. Then
$$\Tor^{\Gamma(F_{r,i})}(\kk,\kk)=
\begin{cases}\bigotimes_{k\ge 0}\left(\, \Lambda(F_{r+k,p^ki+1})\otimes \Gamma(F_{r+k+1,p^ki+2})\,\right)& \text{ if $p$ is odd,}\\
\bigotimes_{k\ge 0} \Gamma(F_{r+k,p^ki+1}) & \text{ if $p=2$.}
\end{cases}$$
\end{proposition}
\begin{proof}
Let $E=\Gamma(F_{r,i})$ and $E_k=S_1(F_{r+k,p^ki})$. Then $E(\kk)\simeq \bigotimes_{k\ge 0} E_k(\kk)$ as graded algebras with weights. The result follows from theorem \ref{thm-calcul-B} (Note that $S_1=\Lambda$ if $p=2$)
\end{proof}

The next proposition follows from a standard argument \cite[Lm 1]{Cartanbis}. 
\begin{proposition}\label{prop-cartan}
Let $E$ be a strict exponential functor, and let $j\ge 2$. Assume that $\Tor^E_{[j-1]}(\kk,\kk)$ is a cofree graded coalgebra. There is an isomorphism of strict exponential functors
$$\Tor^E_{[j]}(\kk,\kk)\simeq \Tor^{\Tor^E_{[j-1]}(\kk,\kk)}(\kk,\kk)\;.$$

\end{proposition}
\begin{proof}
Let $T=\Tor^E_{[j-1]}(\kk,\kk)$ and $B=\overline{B}^{j-1}E$. Since $\overline{B}(B(\kk))$ is a CDGA-algebra with divided powers (see \cite[V.6]{Brown}), the equality $T(\kk)=\Tor^{E(\kk)}_{[j-1]}(\kk,\kk)$ can be lifted to a quasi-isomorphism of CDGA-algebras $T(\kk)\to B(\kk)$, preserving the weights.
Bar constructions preserve quasi-isomorphisms \cite[X Thm 11.2]{ML}, so $\Tor_{[j]}^{E(\kk)}(\kk,\kk)$ is isomorphic to $\Tor^{T(\kk)}(\kk,\kk)$ as a graded Hopf algebras with weights.
The isomorphism now follows from theorem \ref{thm-classif-strict}.
\end{proof}

Notice that proposition \ref{prop-cartan} together with propositions \ref{prop-cal1} and \ref{prop-cal2} can be used to compute inductively $\Tor^E_{[j]}(\kk,\kk)$ for all $j$ when $E=S(F_{r,i})$, $\Lambda(F_{r,i})$ or $S_n(F_{r,i})$. We only write down explicitly the case $j=2$ and leave the combinatorics of higher computations to the reader (this combinatorics can be encoded by the combinatorics of words as in \cite{Cartan}).
\begin{example}\label{ex-cal3}
We have:
\begin{align*}
&\Tor^{S(F_{r,i})}_{[2]}(\kk,\kk)=\Gamma(F_{r,i+2})\;,\\
&\Tor^{\Lambda(F_{r,i})}_{[2]}(\kk,\kk)=\Tor^{\Gamma(F_{r,i+1})}(\kk,\kk)\;,\\
&\Tor_{[2]}^{S_n(F_{r,i})}(\kk,\kk)=\Gamma(F_{r,i+2})\otimes \Tor^{\Gamma(F_{r+n,p^{n-1}i+2})}(\kk,\kk)\text{ if $p$ is odd or if $n>1$}\;. 
\end{align*}
\end{example}

\begin{remark}[Imperfect fields]\label{rk-imperfect}
If $E$ is an exponential functor defined over $\Fp$ and if $\kk$ is a field of characteristic $p$, then we have an exact base change functor $\PP_{\omega,\Fp}\to  \PP_{\omega,\kk}$ which preserves tensor products \cite[Section 2]{SFB}. In particular any exponential functor $E$ yields an exponential functor $E_\kk$ after base change and by exactness, $\Tor_{[j]}^{E_\kk}(\kk,\kk)$ is isomorphic to the exponential functor obtained by applying base change to $\Tor_{[j]}^{E}(\Fp,\Fp)$. In particular, the explicit computations of propositions  \ref{prop-cal1}, \ref{prop-cal2} and example \ref{ex-cal3} are valid over imperfect fields as well.
\end{remark}

Now we turn to the proof of theorems theorems \ref{thm-calcul-fct-hom} and \ref{thm-calcul-fct-hom2}. It is proved in \cite[Prop 7.1]{TouzeBar} that the strict exponential functors $\mathbb{E}(\Lambda,E)$ and $\mathbb{E}(S,E)$ are isomorphic to $\Tor^E(\kk,\kk)$ and $\Tor^E_{[2]}(\kk,\kk)$ up to changing the gradings. To be more specific, we have isomorphisms:
$$w_kH_{i}(\overline{B}E)\simeq w_k\mathbb{E}(\Lambda,E)^{2k-i}\text{ and }
w_kH_{i}(\overline{B}^2E)\simeq w_k\mathbb{E}(S,E)^{2k-i}\;.$$
Thus theorem \ref{thm-calcul-fct-hom} follows from theorem \ref{thm-calcul-B} and proposition \ref{prop-Borel}, and theorem \ref{thm-calcul-fct-hom2} follows from the computations of proposition \ref{prop-cal1} and  example \ref{ex-cal3} (the computations are valid over an arbitrary field, as explained in remark \ref{rk-imperfect}).

\section{Application: strict versus ordinary exponential functors over finite fields}\label{sec-strict-vs-ord}

Strict analytic functors can be thought of as ordinary functors with domain $\Proj_\kk$ and codomain $\Mod$, equipped with an additional structure. To be more specific, there is a forgetful functor commuting with colimits:
$$\U:\PP_{\omega,\kk}\to \Fct(\Proj_\kk,\Mod)\;.$$
The restriction of this functor to $\PP_{d,\kk}$ was denoted by $\U_d$ in section \ref{subsec-setting}. 
If $\kk$ is an infinite field, $\U$ is fully faithful. If $\kk$ is a finite field, $\U$ is only faithful.
Throughout section \ref{sec-strict-vs-ord}, we fix a finite field $\Fq$ with $q=p^\ell$ elements, we let
$$\PP_q:=\PP_{\omega,\Fq}\;,\text{ and }\; \FF_q:=\Fct(\Proj_\Fq,\mathrm{Mod}_\Fq)\;. $$ 
The following lemma summarizes the properties of additive objects of $\FF_q$ and $\PP_q$ given in section \ref{subsubsec-setting-add} and shows that the forgetful functor is well-understood on the level of these additive objects.
\begin{lemma}\label{lm-fct-add-U}
The objects of $\PP_{q\,\add}$ or $\FF_{q\,\add}$ are direct sums of simple objects, and the simple objects of $\PP_{q\,\add}$ and $\FF_{q\,\add}$ have endomorphism ring of dimension one. The simple objects in $\PP_q$ (resp. in $\FF_q$) are the Frobenius twists $I^{(i)}$, for $i\ge 0$ (resp. for $i\in \Z/\ell\Z$). The forgetful functor sends $I^{(i)}$ (in $\PP_q$) to $I^{(i)}$ (in $\FF_q$). 
\end{lemma}

However the forgetful functor $\U$ is not as simple as lemma \ref{lm-fct-add-U} might suggest. To illustrate this, we mention three basic open problems.
\begin{openpbs}\label{pb-sec-8}
(1) Find workable conditions for an object of $\FF_q$ to be in the image of $\U$.
(2) Describe the antecedents of an object of $\FF_q$. (3) Find conditions on an indecomposable object of $\PP_q$ such that it remains indecomposable in $\FF_q$. 
\end{openpbs} 
In this section we show that although $\U$ may be quite complicated over a finite field, the induced forgetful functor  
$$\U:\PP_q-\Exp_c\to \FF_q-\Exp_c$$
is surprisingly simpler. The main result of the section is theorem \ref{thm-Uprime} below, which solves the three problems  above in the context of exponential functors. It generalizes  the results of \cite[Section 9]{TouzeBar}, which were motivated by $\Ext$-computations for strict polynomial functors. 
\begin{definition}\label{def-weights}
Let $E$ be an exponential functor of $\FF_q-\Exp_c$. A \emph{weight decomposition of $E$} is a decomposition of each homogeneous functor $E^i$
$$E^i=\bigoplus_{k\ge 0}w_kE^i$$
such that (1) the multiplication restricts to $w_kE^i\otimes w_\ell E^j\to w_{k+\ell}E^{i+j}$, (2) the unit and counit induce an isomorphism $\kk\simeq w_0E^0$ and (3) the primitives and the indecomposables of $E$ are concentrated in weights $p^k$, $k\ge 0$.
The weight decomposition is called \emph{consistent} if in addition each summand $I^{(i)}$ of $QE$ or $PE$ is homogeneous of weight $p^k$ for some $k$ such that $k=i \mod \ell$.
\end{definition}

\begin{fundexample}\label{ex-weights}
If $E=\U(E')$ for some strict exponential functor $E'$, a consistent weight decomposition of $E$ is given by letting
$$w_k E^i:=  \U(w_k{E'}^i)\;.$$
\end{fundexample}

\begin{theorem}\label{thm-Uprime}
Let $E$ be an ordinary exponential functor with domain $\Proj_\Fq$ and codomain $\mathrm{Mod}_\Fq$. The following assertions are equivalent.
\begin{enumerate}
\item[(1)] The exponential functor $E$ lies in the image of the forgetful functor
$$\U:\PP_q-\Exp_c\to \FF_q-\Exp_c\;.$$
\item[(2)] The exponential functor $E$ admits a consistent weight decomposition.
\item[(3)] The exponential functor $E$ admits a weight decomposition.
\item[(4)] The exponential functor $E^0$ can be equipped with a grading, such that $E^0$ is a connected exponential functor with respect to this grading.
\end{enumerate}

Moreover, if $E$ lies in the image of $\U$, then there is a one-to-one correspondence between the isomorphism classes of strict exponential functors $E'$ such that $\U(E')\simeq E$ on the one hand, and the consistent weight decompositions of $E$ on the other hand.

Finally, the forgetful functor sends indecomposable connected strict exponential functors to indecomposable exponential functors. Also, if $E$ is a connected indecomposable ordinary exponential functor in the image of $\U$, then there is a strict exponential functor $E'$ such that the antecedent of $E$ are up to isomorphism  exactly the strict exponential functors $(E')^{(j\ell)}$, $j\ge 0$.
\end{theorem}

\begin{remark}
In the last part of theorem \ref{thm-Uprime}, indecomposability is in the abelian category of exponential functors. Thus, an exponential functor is indecomposable if and only if it cannot be written as the tensor product of two nontrivial exponential functors. 
\end{remark}

The remainder of section \ref{sec-strict-vs-ord} is devoted to the proof of theorem \ref{thm-Uprime}. The following statement follows easily from lemma \ref{lm-splitting-principle} and lemma \ref{lm-fct-add-U}.
\begin{lemma}\label{lm-facile}
Assume that $p$ is odd. If $E=\Lambda(A)$ then $E$ lies in the image of $\U$. Moreover, there is a one-to-one correspondence between (i) the isomorphism classes of strict exponential functors $E'$ such that $\U(E')\simeq E$, (ii) the isomorphism classes of graded strict analytic functors $A'$ such that $\U(A')\simeq A$, and (iii) the consistent weight decompositions of $E$.
\end{lemma} 

Thus in order to prove theorem \ref{thm-Uprime} when $p$ is odd, we may assume that $E$ is concentrated in even degrees by the splitting principle (lemma \ref{lm-splitting-principle}). If $p=2$, we may also assume that $E$ is concentrated in even degrees (by doubling the degrees). In section \ref{subsec-ref-dec} we prove the equivalence between (1)-(4) as well as the last part of theorem \ref{thm-Uprime}. In section \ref{subsec-eq-weights} we prove that it is equivalent to give a consistent weight decomposition of $E$ and an antecedent of $E$ by $\U$. In both cases, we consider functors concentrated in even degrees, and we obtain the result from more precise descriptions of the forgetful functor $\U$. 

\subsection{A refined decomposition of connected ordinary exponential functors}\label{subsec-ref-dec} Let $\kk$ be a field of characteristic $p$. 
If $n$ is prime to $p$, we let $\HH_n$ be the full subcategory of $\HH$ supported by the connected Hopf algebras whose indecomposables (or equivalently whose primitives) are concentrated in degrees $2np^k$, $k\ge 0$. For all additive categories $\V$, we let $\Fct(\V,\Mod)-\Exp_c\langle 2n\rangle$
be the full subcategory of $\Fct(\V,\Mod)-\Exp_c$ supported by the exponential functors $E$ such that for all $V$, the Hopf algebra $E(V)$ lies in $\HH_n$. 
\begin{lemma}\label{lm-decomp-ord-triv}Let $\kk$ be a field of positive characteristic $p$ and let $\V$ be an additive category. 
There is a categorical decomposition
$$\Fct(\V,\Mod)-\Exp^+_c\simeq \prod_{n\text{ prime to }p}\Fct(\V,\Mod)-\Exp_c\langle 2n\rangle\;.$$
\end{lemma}
\begin{proof}
By lemma \ref{lm-intelligent} the category of ordinary exponential functors identifies with $\Fct_\add(\V,\HH)$. The categorical decomposition is induced by the one recalled in fact \ref{thm-schoeller1}.
\end{proof}

We now refine the categorical decomposition of lemma \ref{lm-decomp-ord-triv} in the special case $\kk=\Fq$ and $\V=\Proj_\Fq$ with $q=p^\ell$. For all $i\in \Z/\ell\Z$,
we let 
$$\mathcal{E}\langle 2n,s\rangle\subset  \FF_q-\Exp_c\langle 2n\rangle$$ be the full subcategory of supported by the exponential functors $E$ such that for all $k\ge 0$, $QE^k$ and $PE^k$ are direct sums of Frobenius twists $I^{(j)}$ for integers $j$ satisfying $j=k+s\mod \ell$.

\begin{proposition}\label{prop-decompFq}
For all positive integers $n$ prime to $p$, there is a categorical decomposition: 
$$\FF_q-\Exp_c\langle 2n\rangle\simeq \prod_{s\in\Z/\ell\Z}\mathcal{E}\langle 2n,s\rangle\;.$$
Moreover, let $\HH'_n$ be the full subcategory of $\HH_n$ supported by the Hopf algebras whose endomorphism ring has characteristic $p$. Evaluation on $\Fq$ induces an equivalence of categories:
$$\mathcal{E}\langle 2n,s\rangle\xrightarrow[]{\simeq} \HH'_n\;.$$
\end{proposition}

The proof of proposition \ref{prop-decompFq} relies on lemma \ref{lm-dieu-exp} below, which follows from the classical theory of graded Dieudonn\'e modules recalled in appendix \ref{App-Dieu}. We first state a definition which will be also useful in section \ref{sec-indecomp}. 
\begin{definition}\label{def-DA}
Let $\A$ be an abelian category. The category $\mathcal{D}_\A'$ of graded Dieudonn\'e modules in $\A$ is defined as follows. The objects of $\mathcal{D}_\A'$ are the nonnegatively graded objects in $\A$: $M=\bigoplus_{i\ge 0} M^i$, equipped with morphisms $F_i:M^i\leftrightarrows M^{i+1}:V_i$ such that $F_iV_i=0=V_iF_i$ for all $i\ge 0$. The morphisms of $\mathcal{D}_\A'$ are the (degree preserving) morphisms of graded objects which commute with the operators $F_i$ and $V_i$.
\end{definition}
 
If $\A=\Mod$ then $\mathcal{D}'_\A$ is the category denoted by $\mathcal{D}'_\kk$ in appendix \ref{App-Dieu}.  
\begin{lemma}\label{lm-dieu-exp}
Let $\kk$ be a perfect field of positive characteristic $p$, and let $\V$ be a small additive category of characteristic $p$. Let $n$ be a positive integer prime to $p$. There is an equivalence of categories:
$$\mathbf{M}:\Fct(\V,\Mod)-\Exp_c\langle 2n\rangle\xrightarrow[]{\simeq} \mathcal{D}'_{\Fct_\add(\V,\Mod)}\;.$$
If $E$ is an exponential functor and $\mathbf{M}(E)=(M^*,V_*,F_*)$ then we have:
\begin{align*}
&PE^{2np^i} = {^{(i)}}(\mathrm{Ker}\, V_{i-1})\;,
&QE^{2np^s} = {^{(i)}}(\mathrm{Coker}\,F_{i-1})\;.
\end{align*}
\end{lemma}
\begin{proof}
By lemma \ref{lm-intelligent}, the category $\Fct(\V,\Mod)-\Exp_c$ identifies with the category
$\Fct_\add(\V,\HH_n)$. Moreover, since $\V$ has characteristic $p$, the latter is equivalent to $\Fct_\add(\V,\HH_n')$ (the equivalence being induced by the inclusion $\HH'_n\subset \HH_n$). Now by fact \ref{thm-schoeller-Fp} there is an equivalence of categories 
$$\Fct_\add(\V,\HH_n')\simeq \Fct_\add(\V,\mathcal{D}'_\kk)\;.$$
Finally $\Fct_\add(\V,\mathcal{D}'_\kk)$ readily identifies with $\mathcal{D}'_{\Fct_\add(\V,\Mod)}$. The formulas for primitives and indecomposables follow from facts \ref{fact-cor-P} and \ref{fact-cor-Q}.
\end{proof}

\noindent
{\bf Proof of proposition \ref{prop-decompFq}. }
Since $\FF_q$ is semisimple with simple objects $I^{(s)}$, $0\le s<\ell$, and since each of these simples has an endomorphism ring isomorphic to $\Fq$, we have a decomposition 
$$\mathcal{D}'_{\FF_q}\simeq \prod_{s\in \Z/\ell\Z} \mathcal{D}'_\kk\otimes I^{(s)}\qquad(*)$$
where $\mathcal{D}'_\Fq\otimes I^{(s)}$ is the full subcategory of $\mathcal{D}'_{\FF_q}$ supported by the functors of the form $M\otimes I^{(s)}$ for some object $M$ of $\mathcal{D}'_\Fq$. 
The formulas describing primitives and indecomposables in lemma \ref{lm-dieu-exp} show that an exponential functor $E$ of $\FF_q-\Exp_c\langle 2n\rangle$ lies in $\mathcal{E}\langle 2n,s\rangle$ if and only if $\mathbf{M}(E)$ lies in $\mathcal{D}'_\Fq\otimes I^{(s)}$. Thus $\mathbf{M}^{-1}$ transforms the categorical decomposition $(*)$ into the categorical decomposition of proposition \ref{prop-decompFq}.

Moreover, there is a commutative diagram, in which the horizontal arrows are induced by evaluation on $\Fq$:
$$\xymatrix{
\mathcal{E}\langle 2n,s\rangle\ar[r]^-{\mathrm{ev}_\Fq}\ar[d]^-{\mathbf{M}}& \HH'_n\ar[d]^-{\mathbf{M}}\\
\mathcal{D}'_\Fq\otimes I^{(s)}\ar[r]^-{\mathrm{ev}_\Fq}&\mathcal{D}'_\Fq
}.$$
Since the vertical arrows and the bottom horizontal arrow are equivalences of categories, the top horizontal arrow is also an equivalence of categories.
\qed

Recall from proposition \ref{prop-categ-decomp} the categorical decomposition:
$$\PP_q-\Exp^+_c\simeq \prod_{a\in\mathbb{N}[p^{-1}]\,,\,a\ne 0}\PP_q-\Exp_c(2a)\;,$$
where $\PP_q-\Exp_c(2a)$ is the full subcategory of $\PP_q-\Exp^+_c$ supported by the exponential functors $E$ such that $w_kE^i=0$ if $2ak\ne i$. In particular, any direct summand $I^{(j)}$ of the primitives (or the indecomposables) of $E$ must be placed in degrees $2ap^j$. Thus, if $a=np^{-s}$ with $s\in\mathbb{Z}$ and $n\in\mathbb{N}$ is prime to $p$, the forgetful functor restricts to a functor:
$$\U:\PP_q-\Exp_c(2a)\to \mathcal{E}\langle 2n,s\rangle \;.$$
\begin{proposition}\label{prop-presq-categ}
Let $a=np^{-s}$ with $s\in\mathbb{Z}$ and $n\in\mathbb{N}$ is prime to $p$. Then
$$\U:\PP_q-\Exp_c(2a)\to \mathcal{E}\langle 2n,s\rangle $$
is fully faithful. It is an equivalence of categories if and only if $s\ge 0$. If $s<0$, then its image is the full subcategory supported by the exponential functors $E$ in $\mathcal{E}\langle 2n,s\rangle$ such that $E^k=0$ if $k$ is not divisible by $2np^{-s}$.
\end{proposition}
\begin{proof} We have a commutative diagram in which the vertical arrows are given by evaluation on $\Fq$ (they are equivalences of categories by theorem \ref{thm-classif-strict} and proposition \ref{prop-decompFq}), and the bottom horizontal arrow forgets the weights:
$$\xymatrix{
\PP_q-\Exp_c(2a)\ar[d]^-\simeq\ar[r]^-{\U}& \mathcal{E}\langle 2n,i\rangle \ar[d]^-\simeq\\
^\strict\HH(2a)\ar[r]^-{\U}&\HH_n'
}\;.$$
Thus it suffices to prove that the bottom horizontal arrow is fully faithful, and that it is essentially surjective if and only if $s\ge 0$.
The weights of an object of $^\strict\HH(2a)$ are determined by the degrees because $w_kH^i=0$ for $i\ne 2ak$. Thus, the bottom horizontal arrow is fully faithful (any morphism of graded Hopf algebras automatically preserves the weight decomposition). If $s\ge 0$ and $H$ is an object of $\HH'_n$, then we can define weights on $H$ by placing each summand $H^{2nj}$ in weight $jp^s$, hence $\U$ is essentially surjective. If $s<0$, the equation $w_kH^i=0$ for $i\ne 2ak$ shows that all objects of $^\strict\HH(2a)$ are zero in degrees different from $2np^{-s}k$, $k\ge 0$. Conversely, if $H$ is an object of $\HH'_n$ such that $H^j=0$ if $j$ is not divisible by $2np^{-s}$, then we can define a weight decomposition on $H$ by placing each summand $H^{k2np^{-s}}$ in weight $k$. This gives the image of $\U$, and proves in particular that $\U$ is not essentially surjective.
\end{proof}

\begin{corollary}
Assertions (1), (2), (3), (4) in theorem \ref{thm-Uprime} are equivalent.
\end{corollary}
\begin{proof}
We clearly have (1)$\Rightarrow$(2)$\Rightarrow$(3)$\Rightarrow$(4). Now assume that assertion (4) holds. Write $E\simeq E'\otimes E^0$ with $E'$ connected. In order to prove that $E$ is in the image of $\U$, it suffices to prove that $E'$ and $E^0$ lie in the image of $\U$. 
Moreover, let $\widetilde{E}^0$ denote $E^0$ with its new grading. If $\widetilde{E}^0$ lies in the image of $\U$ then $E^0$ lies in the image of $\U$ (take an antecedent of $\widetilde{E}^0$ and forget its degree grading). Thus we have reduced ourselves to proving that connected ordinary exponential functors lie in the image of $\U$. But any connected ordinary exponential functor can be written as a tensor product of objects of the categories $\mathcal{E}\langle 2n,s\rangle$, which lie in the image of $\U$ by proposition \ref{prop-presq-categ}.
\end{proof}

\begin{corollary}
The forgetful functor preserves indecomposability of connected exponential functors.
\end{corollary}
\begin{proof}
If $E$ is a connected indecomposable strict exponential functor, then $E$ belongs to some $\PP_q-\Exp_c(2a)$ with $a\ne 0$. But the restriction of $\U$ to these categories is fully faithful, hence $\U E$ is indecomposable.
\end{proof}

\begin{corollary}\label{cor-nice}
Let $E$ be an indecomposable ordinary exponential functor. There exists a unique (up to isomorphism) strict exponential functor $E'$ such that (i) $\U(E')\simeq E$ and (ii) if $E''$ is a strict exponential functor such that $\U(E'')\simeq E$ then there exists $j\ge 0$ such that $E''\simeq {E'}^{(j\ell)}$.
\end{corollary}
\begin{proof}
By lemma \ref{lm-facile}, we may assume that $E$ is concentrated in even degrees.
By indecomposability, $E$ lies in one of the categories $\mathcal{E}\langle 2n,s\rangle$ with $0\le s<\ell$, and by proposition \ref{prop-presq-categ} the antecedents of $E$ by $\U$ lies in the categories $\PP_q-\Exp^*_c(2a)$ for $a=np^{-s-k\ell}$ with $k\in\Z$. To be more specific, if $E$ is zero in degrees not divisible by $2nt$ and nonzero in degree $2nt$, then $E$ has exactly one antecedent in $\PP_q-\Exp_c(2a)$ if $s+k\ell\le t$, and no antecedent if $s+k\ell> t$. Let $k_0$ be maximal among the integers $k$ such that $s+k\ell\le t$ and let $a_0=np^{-s-k_0\ell}$. Let $E'$ be the antecedent of $E$ in $\PP_q-\Exp_c(2a)$. For all $j\ge 0$, ${E'}^{(j\ell)}$ is an antecedent (hence the unique one) of $E$ in $\PP_q-\Exp_c(2ap^{-j\ell})$, and we have constructed in this way all the antecedents of $E$.
\end{proof}

\subsection{Weight decompositions}\label{subsec-eq-weights}
We let $^{\w}\Exp^\even_c$ be the category whose objects are the ordinary exponential functors with domain $\Proj_\Fq$ and codomain $\mathrm{Mod}_\Fq$ which are concentrated in even degrees (but not necessarily connected) and equipped with a weight decomposition, and whose morphisms are the morphisms of ordinary exponential functors preserving the weights. 
We observe that if $E$ is an object of $^{\w}\Exp^\even_c$, then for all $V$, $E(V)$ is a strict Hopf algebra in the sense of definition \ref{def-strict-Hopf}. Hence by reasoning is the same way as for lemma \ref{lm-intelligent}, we obtain an equivalence of categories:
$$^\w\Exp^\even_c\simeq\Fct_{\mathrm{add}}(\V,{^\strict}\HH^\even)\;.\qquad(*)$$
We also let $^{\cw}\Exp^\even_c$ be the full subcategory of $^{\w}\Exp^\even_c$ supported by the exponential functors whose weight decomposition is consistent.

Let $a\in \mathbb{N}[p^{-1}]$ and let $\mathrm{x}$ stand for `$\cw$' or `$\w$'. We let $^\mathrm{x}\Exp_c(2a)$ be the full subcategory of $^\mathrm{x}\Exp^\even_c$ supported by the exponential functors $E$ such that $w_k E^i=0$ if $i\ne 2ak$. Thus, $E$ is an object of $^\mathrm{x}\Exp_c(2a)$ if and only if $E(V)$ belongs to the category ${^{\strict}}\HH(2a)$ (defined in section \ref{subsec-decomp}) for all $V$. 

\begin{lemma}\label{lm-dec}
If $\mathrm{x}=\w$ or $\cw$, there is an equivalence of categories: 
$$^{\mathrm{x}}\Exp^\even_c\simeq \prod_{a\in \mathbb{N}[p^{-1}]} {^{\mathrm{x}}}\Exp_c(2a)\;.$$
\end{lemma}
\begin{proof}
If $\mathrm{x}=\w$, the result follows from the equivalence of categories $(*)$ and the categorical decomposition of ${^\strict}\HH^\even$ given in proposition \ref{prop-categ-decomp}. The result for $\mathrm{x}=\cw$ is obtained by restriction of the equivalence of categories for $\mathrm{x}=\w$. 
\end{proof}

\begin{lemma}\label{lm-ff}
For all $a\in\mathbb{N}[p^{-1}]$, evaluation on $\Fq$ yields an equivalence of categories:
$$^{\cw}\Exp^\even_c(2a)\xrightarrow[]{\simeq} {^{\strict}}\HH(2a)\;.$$
\end{lemma}
\begin{proof} Assume first that $a\ne 0$. Then we can write $a=np^{-s}$ with $n$ prime to $p$ and $s\in\mathbb{Z}$. Let $r=\max\{-s,0\}$. 
We have a commutative diagram, in which the vertical arrows are induced by evaluation on $\Fq$, the arrows $\U$ are the forgetful functors, which forget the weights, and the arrows $\mu_r$ modify the degrees by multiplying them by $p^r$:
$$\xymatrix{
^{\cw}\Exp^\even_c(2a)\ar[d]& ^{\cw}\Exp^\even_c(2a/p^r) \ar[l]_-{\mu_r}\ar[r]^-{\U}\ar[d] & \mathcal{E}\langle 2n,s\rangle\ar[d]\\
{^{\strict}}\HH(2a) &{^{\strict}}\HH(2a/p^r)  \ar[l]_-{\mu_r}\ar[r]^-{\U}& \HH_n'
}\;.$$
The functors $\mu_r$ are equivalences (with inverse $\mu_{1/r}$), the forgetful functors $\U$ are equivalences (because in their domain categories, the weights are determined by the degrees). The vertical arrow on the right is an equivalence of categories by proposition \ref{prop-presq-categ}. This proves lemma \ref{lm-ff} for $a\ne 0$.

Assume now that $a=0$. There is a commutative diagram in which the horizontal arrows are given by evaluation on $\Fq$, and the vertical arrows are given by regrading the objects $H$, by placing each summand $w_k H$ in degree $2k$ (these are equivalences of categories, the inverse equivalences being given by forgetting the degrees):
$$\xymatrix{
^{\cw}\Exp^\even_c(2)\ar[r]^-{\simeq}& {^{\strict}}\HH(2)\\
^{\cw}\Exp^\even_c(0)\ar[u]^-{\simeq}_{\mathcal{R}_2}\ar[r]& {^{\strict}}\HH(0)\ar[u]^-{\simeq}_{\mathcal{R}_2}
}$$
Thus the bottom horizontal arrow is an equivalence of categories, which proves lemma \ref{lm-ff} for $a=0$.
\end{proof}

We can factor the forgetful functor $\U$ as the composition
$$\PP_{q}-\Exp^\even_c\xrightarrow[]{\U_1}{^\cw}\Exp^\even_c\xrightarrow[]{\U_2}\FF_q-\Exp^\even_c$$
where $\U_1$ sends a strict exponential functor $E$ to the ordinary exponential functor $\U(E)$ equipped with the weight decomposition of example \ref{ex-weights}, and $\U_2$ forgets the weight decomposition. Giving a consistent weight decomposition of an ordinary exponential functor $E$ amounts to giving an antecedent of $E$ by $\U_2$. Thus, proposition \ref{prop-next} proves the second part of theorem \ref{thm-Uprime}. 
\begin{proposition}\label{prop-next}
The forgetful functor $\U_1$ is an equivalence of categories.
\end{proposition}
\begin{proof}
The equivalence of categories $\PP_{\omega,\Fq}-\Exp^\even_c\xrightarrow[]{\simeq}{^\strict}\HH^\even$ given by evaluation on $\Fq$ (see theorem \ref{thm-classif-strict}) factors as the composition
$$\PP_{q}-\Exp^\even_c\xrightarrow[]{\U_1}{^\cw}\Exp^\even_c\xrightarrow[]{\mathrm{Ev}} {^\strict}\HH^\even$$
where $\mathrm{Ev}$ is given by evaluation on $\Fq$. But the functor $\mathrm{Ev}$ decomposes as the product of the functors ${^\cw}\Exp_c(2a)\to {^\strict}\HH(2a)$ given by evaluation on $\Fq$, for all $a\in\mathbb{N}[p^{-1}]$. Thus it follows from lemma \ref{lm-ff} that $\mathrm{Ev}$ is an equivalence of categories. Since $\mathrm{Ev}$ and $\mathrm{Ev}\circ\U_1$ are equivalences of categories, $\U_1$ is an equivalence of categories as well.
\end{proof}

\section{Indecomposable exponential functors}\label{sec-indecomp}

In this section, $\kk$ is a perfect field of positive characteristic $p$, and $\FF$ denotes $\PP_{\omega,\kk}$ or $\Fct(\V,\Mod)$, where $\V$ is a small additive category of characteristic $p$ (i.e. the morphisms of $\V$ are $\Fp$-vector spaces). 
Our purpose is to study the connected indecomposable exponential functors, i.e. those which are not isomorphic to a tensor product $E'\otimes E''$ where $E'$ and $E''$ are nontrivial exponential functors. 

\begin{definition}
An \emph{$FV$-word} is either the empty word $\epsilon$, a finite word $w=w_1\cdots w_n$ or an infinite word $w=w_1w_2 \cdots$, with $w_i\in\{F,V\}$ for all $i$. 
\end{definition}
The following result will be proved below.

\begin{proposition}\label{prop-defiEAW}
Let $A[d]$ be a simple additive functor placed in degree $d$. If $p$ is odd, assume furthermore that $d$ is even. For all $FV$-words $w$, there is a unique indecomposable exponential functor $E$ such that:
\begin{align*} 
& PE^d=QE^d=E^d=A\;,\\
& PE^{dp^r}={^{(r)}}A \;\text{ if $r\ge 1$ and $w_r=F$,} \\
& QE^{dp^r}={^{(r)}}A \;\text{ if $r\ge 1$ and $w_r=V$,}
\end{align*}
and  $PE^i$ and $QE^i$ are zero otherwise. Moreover, the endomorphism ring of $E$ is isomorphic to $\End_{\FF}(A)$. 
\end{proposition}

\begin{notation}
We let $\Sigma(A[d],w)$ be the indecomposable functor associated to the simple graded functor $A[d]$ and the $FV$-word $w$ as in proposition \ref{prop-defiEAW}. 
\end{notation}

\begin{remark}
The family of exponential functors of the form $\Sigma(A[d],w)$ contains many exponential functors used in the previous sections of this article. For example, by the uniqueness in proposition \ref{prop-defiEAW}, we see that $S(A[d])\simeq\Sigma(A[d],F^\infty)$ and $\Gamma(A[d])\simeq \Sigma(A[d],V^\infty)$. (Here $A$ is simple, $d$ is even if $p$ is odd, and  $x^\infty$ denotes the infinite word built only with the letter $x$). 
If $p=2$, we also have $\Lambda(A[d])\simeq\Sigma(A[d],\epsilon)$.  
\end{remark}

For all exponential functors $E$, we let $PQE$ be the image of the canonical map $PE\to QE$. For example, $PQ\Sigma(A[d],w)=A[d]$. Since $PQ(E\otimes E')=PQE\oplus PQE'$, an exponential functor $E$ such that $PQE$ is an indecomposable functor must be an indecomposable exponential functor. We will see in theorem \ref{thm-classif-indecomp} that the converse holds for a large class of exponential functors, namely the reflexive ones (provided $\FF_\add$ has homological dimension zero). 
\begin{definition}\label{def-reflexive}
An exponential functor $E$ is called \emph{reflexive} if it is connected and for all $i$, the additive functor $QE^i$ has a finite composition series.
\end{definition}

We prove in lemma \ref{lm-fin} below a characterization of reflexivity in terms of primitives. The term `reflexive' is sometimes used for the connected graded Hopf $\kk$-algebras $H$ such that $\dim_\kk H^i<\infty$ for all $i$, see e.g. \cite{Schoeller}. The next lemma explains the relation between the two notions.

\begin{lemma}\label{lm-reflex-Hopf}
Assume that $\FF=\PP_{\omega,\kk}$ or that $\FF=\Fct(\Proj_R,\Mod)$ for some finite dimensional $\Fp$-algebra $R$. Then an exponential functor $E$ is reflexive if and only if for all $V$ the Hopf algebra $E(V)$ is reflexive.
\end{lemma} 
\begin{proof}
The hypothesis on $\FF$ implies that an additive functor $A$ has a finite composition series if and only if $\dim_\kk A(V)<\infty$ for all $V$. Thus $E$ is reflexive if and only if for all $V$, $\dim QE^i(V)<\infty$. Now if $E(V)$ is reflexive, then $\dim_\kk QE^i(V)<\dim_\kk E^i(V)<\infty$, hence $E$ is reflexive. Conversely, if $E$ is reflexive, then as an algebra $E(V)$ is a quotient of $S(QE(V))$ which has finite dimension in each degree, hence $E(V)$ is reflexive.
\end{proof}

The main result of the section is the following theorem. The hypothesis that $\FF_\add$ has homological dimension zero is satisfied in some cases of interest, for example if $\FF=\PP_{\omega,\kk}$ or if $\FF=\Fct(P_R,\Mod)$, for a ring $R$ such that $R\otimes_{\mathbb{Z}}\kk$ has homological dimension zero (see section \ref{subsubsec-setting-add}).

\begin{theorem}\label{thm-classif-indecomp}
Let $\kk$ be a perfect field of positive characteristic $p$, and let $\FF$ be $\PP_{\omega,\kk}$ or $\Fct(\V,\Mod)$, where $\V$ is a small additive category of characteristic $p$. Assume that $\FF_\add$ has homological dimension zero.
Let $E$ be a reflexive exponential functor. 
\begin{enumerate}
\item The exponential functor $E$ is indecomposable if and only if the functor $PQE$ is simple.
\item Assume that $E$ is indecomposable. If $p$ is odd and $PQE$ is concentrated in odd degrees, then $E\simeq \Lambda(PQE)$. If $p=2$ or if $E$ is concentrated in even degrees, then there is an $FV$-word $w$ such that $E\simeq \Sigma(PQE,w)$.
\end{enumerate}
\end{theorem}

\begin{corollary}\label{cor-sgn-indecomp}
Two reflexive indecomposable functors $E_1$ and $E_2$ are isomorphic if and only if there are isomorphisms of graded functors $PE_1\simeq PE_2$ and $QE_1\simeq QE_2$. 
\end{corollary}

\begin{corollary}[Krull-Remak-Schmidt property]\label{cor-KRS} If $E$ is a reflexive exponential functor, then $E$ admits a unique decomposition as a tensor product of an at most countable family of indecomposable exponential functors:
$$E\simeq \bigotimes_{i\in I}E_i\;.$$
\end{corollary}
\begin{proof}
Existence of the decomposition follows from the fact that $QE^i$ is a finite direct sum of simple functors, that $Q(E'\otimes E'')^i=(QE')^i\oplus (QE'')^i$ and that a connected exponential functor $E$ functor is trivial if and only if $QE=0$. By \cite[Chap. I.6 Thm 1]{Gabriel}, uniqueness follows from the fact that $\End_{\FF-\Exp^*_c}(E_i)$ is a division algebra. Indeed, theorem \ref{thm-classif-indecomp} gives an explicit form for $E_i$, and thus the endomorphism algebra of $E_i$ is isomorphic to $\End_{\FF^*}(PQE_i)$. The latter is a division algebra since $PQE_i$ is a simple functor.
\end{proof}

In the rest of the section, we prove proposition \ref{prop-defiEAW} and theorem \ref{thm-classif-indecomp}.
\subsection{Preliminary reductions}
If $p$ is odd, then lemma \ref{lm-splitting-principle} implies theorem \ref{thm-classif-indecomp} for exponential functors $E$ such that $QE$ is concentrated in odd degrees, hence it suffices to prove theorem \ref{thm-classif-indecomp} when $E$ is concentrated in even degrees. If $p=2$, we can also restrict our attention to indecomposable exponential functors concentrated in even degrees, since one can always double the degrees of an exponential functor to obtain an exponential functor concentrated in even degrees. Furthermore, recall the categorical decompositions of proposition \ref{prop-categ-decomp} and lemma \ref{lm-decomp-ord-triv}:
\begin{align*}
&\Fct(\V,\Mod)-\Exp^+_c\simeq \prod_{n\text{ prime to }p}\Fct(\V,\Mod)-\Exp_c\langle 2n\rangle\;,\\
&\PP_{\omega,\kk}-\Exp^+_c\simeq \prod_{a\in\mathbb{N}[p^-1],\,a\ne 0}\PP_{\omega,\kk}-\Exp_c(2a)\;.
\end{align*}
If $E$ is indecomposable, then it lies in one the categories 
$\Fct(\V,\Mod)-\Exp_c\langle 2n\rangle$ or $\PP_{\omega,\kk}-\Exp^+_c(2a)$. But all these categories are equivalent to one another by regrading (see lemma \ref{lm-Ra} for the case of  strict exponential functors). Moreover, the operation of regrading transforms an exponential functor satisfying the conditions of proposition \ref{prop-defiEAW} into another exponential functor satisfying the conditions of proposition \ref{prop-defiEAW}, so we have the following consequence.
\begin{lemma}\label{lm-prem-red}
Proposition \ref{prop-defiEAW} and theorem \ref{thm-classif-indecomp} hold if and only if they hold for all exponential functors $E$ in the subcategories $\Fct(\V,\Mod)-\Exp_c\langle 2\rangle$ and $\PP_{\omega,\kk}-\Exp_c(2)$.
\end{lemma}

\subsection{Graded Dieudonn\'e modules}
In order to treat the case of strict exponential functors and the case of ordinary exponential functors simultaneously, we use the following notations until the end of section \ref{sec-indecomp}.
\begin{itemize}
\item We let $\C$ stand for $\Fct(\V,\Mod)-\Exp_c\langle 2\rangle$ or for $\PP_{\omega,\kk}-\Exp^+_c(2)$ (where $\V$ is a small additive category of characteristic $p$ and $\kk$ is a perfect field of characteristic $p$).
\item We let $\AC$ stand for $\Fct_\add(\V,\Mod)$ in the ordinary case or for $\PP_{1,\kk}$ in the strict case. The objects of $\PP_{1,\kk}$ are direct sums of copies of the functor $I=I^{(0)}$ (in particular all its objects are additive), and there is an equivalence of categories $\Mod\simeq \PP_{1,\kk}$  which sends a vector space $V$ to the functor $V\otimes I$.
\end{itemize}
The category $\C$ can be described in terms of the category of graded Dieudonn\'e modules $\mathcal{D}'_\AC $ from definition \ref{def-DA}.

\begin{lemma}\label{lm-dieu-tout}
There is an equivalence of categories $\mathbf{M}:\C\simeq \mathcal{D}'_\AC $. If $E$ is an exponential functor in $\C$ such that $\mathbf{M}(E)=(M^*,F_*,V_*)$ then 
\begin{align*}
&PE^{2p^s} = {^{(s)}}(\mathrm{Ker}\, V_{s-1})\;,
&QE^{2p^s} = {^{(s)}}(\mathrm{Coker}\,F_{s-1})\;.
\end{align*}
\end{lemma}
\begin{proof}
In the case of ordinary exponential functors, the statement is already proved in lemma \ref{lm-dieu-exp}. In the case of strict exponential functors, we define a similar functor $\mathbf{M}$ as the composition of the equivalences of categories:
$$\PP_{\omega,\kk}-\Exp_c(2) \simeq {^\strict}\HH(2) \simeq \HH_1\simeq \mathcal{D}_{\kk}\simeq \mathcal{D}_{\PP_{1,\kk}}\;.$$
To be more specific, the first equivalence is given by theorem \ref{thm-classif-strict}, the second one forgets the weights (in ${^\strict}\HH(2)$ the weights are a can be retrieved from the degrees by dividing them by $2$), the third one is provided by fact \ref{thm-schoeller-Fp}, and the last is induced by the equivalence $\PP_{1,\kk}\simeq \Mod$, namely, it sends a graded Dieudonn\'e module $M$ in $\Mod$ to  the graded Dieudonn\'e module $M\otimes I$ in $\PP_{1,\kk}$.
The classification of additive strict analytic functors given in lemma \ref{lm-control-strict} and the relation between weights and degrees in $\PP_{\omega,\kk}-\Exp_c(2)$ imply that $PE^{2p^s}$ and $QE^{2p^s}$ are direct sums of copies of $I^{(s)}$. Hence $PE^{2p^s}\simeq PE^{2p^s}(\kk)\otimes I^{(s)}$ and $QE^{2p^s}\simeq QE^{2p^s}(\kk)\otimes I^{(s)}$, and the formula for primitives and indecomposables follows from facts \ref{fact-cor-P} and \ref{fact-cor-Q}.
\end{proof}

We can characterize reflexivity by using graded Dieudonn\'e modules.
\begin{lemma}\label{lm-fin}
If $E$ is an object of $\C$, the following assertions are equivalent.
\begin{enumerate}
\item[(i)] For all integers $i$, $QE^i$ has a finite composition series.
\item[(ii)] For all integers $i$, $\mathbf{M}(E)^i$ has a finite composition series.
\item[(iii)] For all integers $i$, $PE^i$  has a finite composition series.
\end{enumerate}
\end{lemma}
\begin{proof}
Proof of (i)$\Rightarrow$(ii): since $E$ is an object of $\C$, $QE^i=0$ if $i\ne 2p^s$. By lemma \ref{lm-dieu-tout}, $QE^{2p^s}$ is a quotient of $\mathbf{M}(E)^s$, thus if $\mathbf{M}(E)^s$ has a finite composition series for all $s$, then $E$ is reflexive. Proof of (ii)$\Rightarrow$(i): let $\mathbf{M}(E)=(M^*,F_*,V_*)$ and let $E'$ be such that $\mathbf{M}(E')=(M^*,F_*,0)$. 
By lemma \ref{lm-dieu-tout} $QE=QE'$. Since the Verschiebung of $E'$ is trivial, it is primitively generated, hence it is a quotient of $S(QE)$. Thus for all $i$, 
$\mathbf{M}(E)^i$ is a quotient of $\mathbf{M}(S(QE))^i$.

By the description of indecomposables in lemma \ref{lm-dieu-tout}, for all additive $A$ of homogeneous degree $2p^s$, $\mathbf{M}(S(A))^i$ quotient of $^{(-s)}A$ for $i\ge s$ and is zero for $i<s$. Note that $^{(-s)}A$ has a finite composition series if and only if $A$ has a finite composition series (by fact \ref{fact-prop-extscal}\eqref{eqcat} in the ordinary case, and by the classification of additive functors of lemma \ref{lm-control-strict} in the strict case). Hence $\mathbf{M}(E)^i$ is a quotient of 
$$\mathbf{M}(S(QE))^i=\bigoplus_{0\le s} \mathbf{M}(S(QE^{2p^s}))^i=\bigoplus_{0\le s\le i} {^{(-s)}}QE^{2p^s}\;.$$
This proves that (i)$\Rightarrow$(ii). The proof of (ii)$\Leftrightarrow$(iii) is similar. 
\end{proof}

\subsection{String modules}
Let $\A$ be an abelian category.

\begin{definition}
A \emph{string module} is an object $M=(M^*,F_*,V_*)$ of $\mathcal{D}'_\A$ which is indecomposable, and such that in each degree $i$, $M^i$ is either zero or a simple object of $\A$.
\end{definition}

Let $A$ be a simple object of $\A$, let $r$ be a nonnegative integer and let $w$ be a $FV$-word with $\ell$ letters, $\ell\in\mathbb{N}\cup\{+\infty\}$. We define a string module $M_{A,r,w}$ by letting for $X\in \{F,V\}$:
$$M_{A,r,w}^i= \begin{cases}
A & \text{if $r\le i\le \ell+1$,}\\
0 & \text{otherwise.}
\end{cases} \quad 
X_i = \begin{cases}
\Id_A & \text{if $w_{i-r+1}=X$,}\\
0 & \text{otherwise.}
\end{cases}$$

For example, if we represent Dieudonn\'e modules by diagrams with nonzero objects $M^i$ as vertices and with nonzero operators $F_i$ and $V_i$ as rightwards and leftwards arrows respectively, then $M_{A,2,FVV}$ can be depicted as
$$\underbrace{A}_{\deg 2}\xrightarrow[]{\Id_A}A \xleftarrow[]{\Id_A}A \xleftarrow[]{\Id_A}A\;.$$
The following elementary lemma is left to the reader. 
\begin{lemma}\label{lm-string}
A graded Dieudonn\'e module $M$ is a string module if and only if there is a triple $(A,r,w)$ and an isomorphism $M\simeq M_{A,r,w}$. In this case the triple $(A,r,w)$ is unique and $\End_{\D_{\A}}(M)$ is isomorphic to $\End_\A(A)$. 
\end{lemma}

The next theorem is a result of combinatorial representation theory, which will be fundamental for us.

\begin{theorem}\label{thm-classif-combin}
Let $M$ be an indecomposable object of $\mathcal{D}'_\A$ such that for all $i\ge 0$, $M^i$ is a finite direct sum of simple objects of $\A$. Then $M$ is a string module.
\end{theorem}

Theorem \ref{thm-classif-combin} can be proved directly with the classical functorial filtration method introduced by Gelfand and Ponomarev \cite{GelfandPonomarev}, see also \cite{BR}. We rather explain now that theorem \ref{thm-classif-combin} is a consequence of the results of \cite{CB}.

For all simple objects $A$, let $M(A)$ be the submodule of $M$ supported in each degree by the isotopic component of $M^i$ corresponding to $A$. Then $M$ splits as a direct sum $M=\bigoplus M(A)$. Thus, we may assume that $M=M(A)$ for some simple $A$, that is, each $M^i$ is a finite direct sum of copies of a fixed simple object $A$. In other terms, we may replace $\A$ by its full subcategory $\mathrm{Add}(A)$ supported by finite direct sums of copies of $A$. Now let $K=\End_\A(A)$. Since $A$ is simple, $K$ is a (possibly skew) field, and $\mathrm{Add}(A)$ is equivalent to the category of finitely generated $K$-modules. Thus it suffices to prove theorem \ref{thm-classif-indecomp} when $\A=\mathrm{mod}_{K}$. In the latter case, $\mathcal{D}'_\A$ identifies with the representations of the string algebra $\Lambda$ generated by the quiver 
$$\xymatrix{
\bullet^0 \ar@/^/[r]^-{F_0} & \ar@/^/[l]^-{V_0}\bullet^1 \ar@/^/[r]^-{F_1}&
\ar@/^/[l]^-{V_1}\bullet^2 \ar@/^/[r]^-{F_2}&\ar@/^/[l]^-{V_2}\cdots \;.
}$$
with relations $F_iV_i=0=V_iF_i$, which are pointwise finite dimensional, hence pointwise artinian in the sense of \cite{CB}. Since the string algebra $\Lambda$ has no periodic words, there is no band module to consider. Hence, theorem \ref{thm-classif-combin} is a special case of \cite[Thm 1.2]{CB} (there is no need that the field $K$ is commutative in the proof of that theorem).

\subsection{Proof of theorem \ref{thm-classif-indecomp} and proposition \ref{prop-defiEAW}}
By lemma \ref{lm-prem-red}, we can assume that $E$ is an object of $\C$. By lemma \ref{lm-fin} the fact that $E$ is reflexive is equivalent to the fact that the graded Dieudonn\'e module $\mathbf{M}(E)$ is a direct sum of finitely many simple objects of $\AC$ in each degree. By theorem \ref{thm-classif-combin}, $\mathbf{M}(E)$ is then indecomposable if and only if it is a string module. Thus, theorem \ref{thm-classif-indecomp} and proposition \ref{prop-defiEAW} both follow from the following statement.
\begin{lemma}
Let $E$ be an object of $\C$. Then $E$ satisfies the conditions of proposition \ref{prop-defiEAW} relatively to a triple $(A,d,w)$ if and only if $d=2p^r$ for some $r\ge 0$ and $\mathbf{M}(E)$ is the string module associated to the triple $({^{(-r)}}A,r,w)$.
\end{lemma}
\begin{proof}
If $E$ is an object of $\C$ then its primitives and indecomposables are concentrated in degrees $2p^s$, $s\ge 0$. The equivalence then follows from the formulas for primitives and indecomposables given in lemma \ref{lm-dieu-tout}.
\end{proof}

\section{Strict exponential functors versus strict analytic functors}\label{sec-unique-strict}
In this section, we restrict our attention to strict exponential functors over a field $\kk$. We consider the forgetful functor:
$$\OO:\PP_{\omega,\kk}-\Exp_c \to \PP_{\omega,\kk}^*$$
which forgets the structural morphisms of an exponential functor, and just retains the underlying strict analytic functor. This functor is faithful but not full.
The main result of the section is the following theorem.

\begin{theorem}\label{thm-uniqueness-strict}
Let $E$ and $E'$ be two strict exponential functors over a perfect field $\kk$. Assume that for all nonnegative integers $k$ and $j$, the vector space $w_kQE^j(\kk)$ is finite dimensional. If the strict analytic functors $\OO E$ and $\OO E'$ are isomorphic in $\PP_{\omega,\kk}^*$, then $E$ and $E'$ are isomorphic as strict exponential functors. 
\end{theorem}
\begin{remark}
The finite dimension assumption in theorem \ref{thm-uniqueness-strict} is satisfied if the algebra $E(\kk)$ is finitely generated, or if $E$ is reflexive in the sense of definition \ref{def-reflexive} (see lemma \ref{lm-virtual-reflex} below).
\end{remark}

Observe that we are not claiming that $\OO$ reflects isomorphisms. In general, there are lots of isomorphisms of the graded strict analytic functors $\OO E$ and $\OO E'$ which are not compatible with the exponential structure.
\begin{example}
Let $\kk$ be a field of cardinal greater than $2$, and $\lambda\in\kk\setminus\{0,1\}$, and let $E$ be a strict exponential functor over $\kk$ and let $i$ be such that $w_iE\ne 0$. Define $f_{i,\lambda}\in \End_{\PP_{\omega,\kk}^*}(\OO E)$ by letting $f_{i,\lambda}: w_kE\to w_kE$ be the identity if $k\ne i$ and multiplication by $\lambda$ if $k=i$. Then $f_{i,\lambda}$ is an isomorphism of graded strict analytic functors, but it is not a morphism of strict exponential functors.
\end{example}

We will prove in section \ref{sec-counter} that theorem \ref{thm-uniqueness-strict} does not hold without the assumption that $\kk$ is perfect. 
Also, the following example shows that theorem \ref{thm-uniqueness-strict} does not hold without the finiteness condition.

\begin{example}\label{ex-infinity}
Let $E_i$ be the quotient of the symmetric algebra $S$ by the ideal generated by the $p^i$-th powers. Let $E=\bigotimes_{i\ge 1} E_i$ and let $E'=E\otimes S$. Then $E$ and $E'$ are not isomorphic (since $E$ has non nilpotent elements). However, for all $r>0$ and all $k<p^r$ we have isomorphisms:
\begin{align*}
w_kE &\simeq \bigoplus_{1\le i\le r} w_i(E_1\otimes\cdots\otimes E_r)\otimes w_{k-i}\big(\textstyle\bigotimes_{r<j}E_{j}\big)\\
&\simeq \bigoplus_{1\le i\le r} w_i(E_1\otimes\cdots\otimes E_r)\otimes w_{k-i}(S^{\otimes \mathbb{N}})\\
&\simeq \bigoplus_{1\le i\le r} w_i(E_1\otimes\cdots\otimes E_r)\otimes w_{k-i}\big(\textstyle S\otimes \bigotimes_{r<j}E_{j}\big) \simeq w_kE'\;.
\end{align*}
\end{example}

 We don't know if an analogue of theorem \ref{thm-uniqueness-strict} holds for ordinary exponential functors. However, some (weaker) results in this direction will be proved later in section \ref{sec-uniqueness}.

The remainder of the section is devoted to the proof of theorem \ref{thm-uniqueness-strict}. We first observe that theorem \ref{thm-uniqueness-strict} is a straightforward consequence of lemma \ref{lm-classif-carzero} and remark \ref{Rk-classif-carzero} if $\kk$ is a field of characteristic zero. Therefore, we assume that $\kk$ is a perfect field of positive characteristic $p$. The proof goes as follows.
\begin{enumerate}
\item In section \ref{subsec-virt-reflex} we observe that it suffices to prove theorem \ref{thm-uniqueness-strict} when $E$ is reflexive (i.e. $\dim E(V)^i$ is finite for all $i\ge 0$, see lemma \ref{lm-reflex-Hopf}).
\item By the results of section \ref{sec-indecomp} any reflexive $E$ decomposes uniquely as a tensor product of graded exponential functors $E_i$, each of these being characterized by its primitives and indecomposables. Grouping these graded additive functors together gives the signature $\sigma(E)$, which characterizes the exponential functor $E$ up to isomorphism, see section \ref{subsec-sign}.
\item In section \ref{subsec-propE} we establish a fundamental property of the indecomposable reflexive exponential functors $E$, namely their summands of homogeneous degree $n$ are either zero or indecomposable. 
\item Finally in section \ref{subsec-proofthm} we achieve the proof of theorem \ref{thm-uniqueness-strict} by showing that the signature of $E$ can be recovered from the graded functor $\OO E$. This relies on the results of section \ref{subsec-propE}.
\end{enumerate}
The obstruction for proving an analogue of theorem \ref{thm-uniqueness-strict} in the case of connected  ordinary exponential functors is that we don't know how to prove the indecomposability property of proposition \ref{prop-indecomp-summand} in the ordinary case.

\subsection{Virtually reflexive exponential functors}\label{subsec-virt-reflex}

Let $E$ be a strict exponential functor. Let $\mathbf{E}$ be the strict exponential functor such that 
$w_k\mathbf{E}^{i}= w_kE^{i-2k}$ and the exponential structure of $\mathbf{E}$ is given by that of $E$, e.g. the multiplication 
${\mu}:w_k\mathbf{E}^{i}\otimes w_\ell\mathbf{E}^{j}\to w_{k+\ell}\mathbf{E}^{i+j}$ is equal to the multiplication 
$\mu: w_kE^{i-2k}\otimes w_\ell E^{j-2\ell}\to w_{k+\ell}E^{i+j-2k-2\ell}$.

\begin{definition}
We call a strict exponential functor $E$ \emph{virtually reflexive} if $\mathbf{E}$ is reflexive.
\end{definition}

The next lemma gives a clearer portrayal of virtually reflexive exponential functors. It shows in particular that reflexive exponential functors are virtually reflexive.
\begin{lemma}\label{lm-virtual-reflex} Let $E$ be a strict exponential functor.
The following assertions are equivalent.
\begin{enumerate}
\item[(i)] The strict exponential functor $E$ is virtually reflexive.
\item[(ii)] For all nonnegative integers $k$ and $i$ and all $V$, $\dim_\kk\;w_kE^i(V)<\infty$.
\item[(iii)] For all nonnegative integers $k$ and $i$, $\dim_\kk\;w_kE^i(\kk)<\infty$.
\item[(iv)] For all nonnegative integers $k$ and $i$, $\dim_\kk\;w_kQE^i(\kk)<\infty$.
\item[(v)] For all nonnegative integers $k$ and $i$, $\dim_\kk\;w_kPE^i(\kk)<\infty$.
\end{enumerate}
\end{lemma}
\begin{proof}
For all $d$, the exponential formula yields an decomposition, in which the direct sum is taken over all pairs of $d$-tuples $(k_1,\dots,k_d)$ and $(i_1,\dots,i_d)$ such that $\sum k_i=k$ and $\sum d_i=d$.
$$w_kE^i(\kk^d)\simeq \bigoplus w_{k_1}E^{i_1}(\kk)\otimes\cdots\otimes w_{k_d}E^{i_d}(\kk)\;.$$
Since the direct sum is finite, this proves that (iii)$\Rightarrow$(ii). (ii)$\Rightarrow$(iii) is obvious.

Next we prove that (i)$\Leftrightarrow$ (iii')$\Leftrightarrow$ (iv').
Since $w_kE^i(\kk)=w_k\mathbf{E}^{i-2k}$, we can replace $E$ by $\mathbf{E}$ in assertions (iii), (iv) without changing their meaning. 
We also observe that by construction, for all $i$, we have a \emph{finite} decomposition:  
$$\widetilde{E}^i=\bigoplus_{0\le k\le i/2}w_k\mathbf{E}^i\;.$$
Therefore assertions (iii) and (iv) are respectively equivalent to:
\begin{enumerate}
\item[(iii')] For all nonnegative integer $i$, $\dim_\kk\;\mathbf{E}^i(\kk)<\infty$.
\item[(iv')] For all nonnegative integers $i$, $\dim_\kk\;Q\mathbf{E}^i(\kk)<\infty$.
\end{enumerate}
 Since additive strict analytic functors are direct sums of Frobenius twist functors $I^{(r)}$, and since $I^{(r)}(\kk)$ has dimension one, $QE^i$ has a finite composition series if and only if $\dim_\kk QE^i(\kk)<\infty$. Thus  (i)$\Leftrightarrow$ (iii') by definition of `reflexive', and (i)$\Leftrightarrow$ (iv') by lemma \ref{lm-reflex-Hopf}. 

Finally we prove (i)$\Leftrightarrow$ (v) by duality. Namely, let $E^\sharp$ be the restricted dual of $E$, i.e. $w_kE^{\sharp\,i}(V)= (w_kE^i(V^*)^*)$ where $^*$ is $\kk$ linear duality (and the structure maps of $E^\sharp$ are obtained from the ones of $E$ by dualizing). Then $E^\sharp$ satisfies (iii) if and only if $E$ satisfies (iii). Since $w_kPE^i(\kk)= w_kQE^{\sharp\,i}(\kk^*)^*$, the fact that equivalence (iii)$\Leftrightarrow$ (iv) holds for$E^\sharp$ proves that the equivalence (iii)$\Leftrightarrow$ (v) holds for $E$.
\end{proof}

\begin{lemma}\label{lm-reduction-reflexive}
Theorem \ref{thm-uniqueness-strict} holds if and only if it holds for all reflexive strict exponential functors $E$.
\end{lemma}
\begin{proof}
Lemma \ref{lm-virtual-reflex} shows that reflexive functors are virtually reflexive. Conversely, if $\OO E\simeq \OO E'$ then both $E$ and $E'$ are virtually reflexive, hence by changing the gradings $\OO \mathbf{E}\simeq \OO \mathbf{E}'$, hence $ \mathbf{E}\simeq \mathbf{E}'$ since $\mathbf{E}$ is reflexive, hence $E\simeq E'$ by coming back to the original gradings.
\end{proof}

\subsection{The signature of a strict exponential functor}\label{subsec-sign}
Let $\kk$ be a perfect field of positive characteristic $p$. By theorem \ref{thm-classif-indecomp} and corollary \ref{cor-KRS}, any nontrivial reflexive $E$ has a unique decomposition as an at most countable tensor product of indecomposable strict exponential functors:
$$E\simeq \bigotimes_{i\in I} E_i\;.$$
Moreover, it follows from proposition \ref{prop-defiEAW} that each $E_i$ is determined up to isomorphism by the pair of graded additive functors $(PE_i,QE_i)$.
Thus the virtually reflexive exponential functor $E$ is determined up to isomorphism by its signature $\sigma(E)$ defined as follows.

\begin{definition}
The \emph{signature} of a virtually reflexive $E$ is the multiset (i.e. non ordered, possibly with some elements repeated) $\sigma(E)$ defined by:
$$\sigma(E)= \{(PE_i,QE_i), i\in I\}\;.$$ 
\end{definition}

Thus, in order to prove theorem \ref{thm-uniqueness-strict}, it suffices to show that one may compute $\sigma(E)$ from the graded strict analytic functor $\OO E$. However, in practice, we won't have a direct way to compute $\sigma(E)$ from $\OO E$, but we will rather be able to compute some kinds of truncations of $\sigma(E)$. To explain this in detail, we need to introduce further notations. 
\begin{convention}\label{conv-multiset}
In the sequel of section \ref{sec-unique-strict}, the term `multiset' will always refer to a multiset of pairs $(A,B)$ in which $A$ and $B$ are graded additive strict analytic functors. 
\end{convention}
Let $(A,B)$ be a pair such that $A$ and $B$ are additive strict analytic functors. For all $k\ge 0$, the truncation $\tau_k(A,B)$ is given by 
$$\tau_k(A,B)=
\big(\bigoplus_{0\le i\le k} A^i\,, \bigoplus_{0\le i\le k} B^i\,\big)\;.$$
If $\sigma$ is a multiset, we define its $k$-th \emph{fake truncation} $\phi_k(\sigma)$ as the multiset
$$\phi_k(\sigma)=\left\{ \tau_k(A,B)\;\left|\; (A,B)\in\sigma\,, (A^{k},B^{k})\ne (0,0) \right.\right\}\;.$$
\begin{example}
If $\sigma=\{(I\oplus I^{(1)},I)\,,\, (I^{(3)},I\oplus I^{(3)})\}$ where each $I^{(k)}$ is placed in degree $k$, then we have:
\begin{align*}
&\phi_0(\sigma)=\{(I,I)\,,\,(0,I)\}\;,\\
&\phi_1(\sigma)=\{(I\oplus I^{(1)},I)\}\;,\\
&\phi_2(\sigma)=\emptyset\;,\\
&\phi_3(\sigma)=\{(I^{(3)},I\oplus I^{(3)}\})\;,\\
&\phi_k(\sigma)=\emptyset \text{ if $k\ge 4$.}
\end{align*}
\end{example}

The next example shows that in general the sequence of all the fake truncations $\phi_k(\sigma)$ is not sufficient to reconstruct $\sigma$.
\begin{example}\label{ex-infinity-multiset}(compare example \ref{ex-infinity})
For all $i\in \mathbb{N}\cup\{\infty\}$, we let: 
$$A_i=\bigoplus_{0\le j<i+1} I^{(j)}\;,$$
where each $I^{(j)}$ is placed in degree $j$.
We consider the two distinct multisets $\sigma$ and $\sigma'$ defined by
$$\sigma=\{ (A_i,I)\;,\; i\in\mathbb{N}\}\qquad \sigma'=\{ (A_i,I)\;,\; i\in\mathbb{N}\cup\{\infty\}\,\}\;.$$
Then for all $k\ge 0$, $\phi_k(\sigma)=\phi_k(\sigma')$ is the multiset which contains the pair $(A_k,I)$ with countable multiplicity and nothing else.
\end{example}

Now we observe that if $E$ is reflexive, then the multiset $\sigma(E)$ is almost finite in the following sense. 
\begin{definition}
We say that a multiset is \emph{almost finite} if for all $i$ it contains only a finite number of pairs $(A,B)$ (counted with multiplicities) such that $(A^i,B^i)\ne (0,0)$.
\end{definition}

We let $\mathcal{S}$ (like `signature') be the set of almost finite multisets. Thus, if $E$ is reflexive, then $\sigma(E)\in \mathcal{S}$, whereas the multisets of example \ref{ex-infinity-multiset} are not elements of $\mathcal{S}$. The fake truncations induce a map:
$$\begin{array}{cccc}
\phi: &\mathcal{S}&\to &\mathcal{S}^{\mathbb{N}}\\
& \sigma & \mapsto & (\phi_k(\sigma))_{k\ge 0}
\end{array}\;.$$
\begin{proposition}\label{prop-inj}
The map $\phi$ is injective.
\end{proposition}
\begin{proof}
Let $\sigma=\{(A_i,B_i),i\in I\}\in\mathcal{S}$. Given a pair $(A,B)$ we let $m(A,B,k)$ be the multiplicity of $\tau_k(A,B)$ in the multiset $\tau_k(\sigma)=\{\tau_k(A_i,B_i),i\in I\}$. Since $\sigma$ is almost finite, the multiplicity $m(A,B)$ of $(A,B)$ in $\sigma$ is finite, the sequence of multiplicities $m(A,B,k)$ is nonincreasing with respect to $k$, and 
$$m(A,B)=\lim_{k\to \infty}m(A,B,k)\;.$$ 
Thus two almost finite multisets $\sigma_1$ and $\sigma_2$ are equal if and only if $\tau_k(\sigma_1)=\tau_k(\sigma_2)$ for all $k\ge 0$. Now one easily determines $\tau_k(\sigma_i)$ from $\phi_k(\sigma_i)$ and $\tau_{k-1}(\sigma_i)$, hence $\tau_k(\sigma_1)=\tau_k(\sigma_2)$ for all $k\ge 0$ if and only if $\phi(\sigma_1)=\phi(\sigma_2)$.
\end{proof}
\begin{corollary}\label{cor-signature}
Let $E$ be a reflexive strict exponential functor over a perfect field $\kk$. The multisets $\phi_k(\sigma(E))$, $k\ge 0$, uniquely determine $E$ up to isomorphism (among the class of reflexive strict exponential functors).
\end{corollary}

\subsection{Special exponential functors and indecomposability}\label{subsec-propE}

We first need a combinatorial lemma of independent interest. This lemma has an obvious dual for commutative algebras (involving multiplication and indecomposables), whose formulation and proof is left to the reader.
\begin{lemma}\label{lm-sec-coalg}
Let $C$ be a cocommutative coalgebra over an arbitrary  field $\kk$. Assume that $C$ is nonnegatively graded, in such a way that the comultiplication preserves the degrees and the augmentation induces an isomorphism $C^0\simeq \kk$. Let $\rho$ be an integer greater or equal to $2$ such that the primitives of $C$ are concentrated in degrees $\rho^r$, $r\ge 0$. 
For all positive integers $k$ with $\rho$-adic decomposition $k=\alpha_0+\alpha_1\rho+\dots+\alpha_n\rho^n$, the iterated coproduct $\Delta:C\to C^{\otimes \sum_i\alpha_i}$ induces an injective morphism of vector spaces: 
$$\Delta_\alpha \,:\,C^k\to (C^1)^{\otimes \alpha_0}\otimes\dots \otimes (C^{\rho^n})^{\otimes\alpha_n}\;.$$
Moreover, for all $n\ge 1$ there is an exact sequence of vector spaces:
$$0\to PC^{\rho^n}\to C^{\rho^n}\xrightarrow[]{\Delta} (C^{\rho^{n-1}})^{\otimes \rho}\;.$$
\end{lemma}
\begin{proof}
For all $k$, the definition of the primitives yield an exact sequence: 
$$0\to PC^k\to C^k\xrightarrow[]{\Delta} \bigoplus_{0<i<k}C^i\otimes C^{k-1} \qquad(*)$$
Since the primitives are concentrated in weights $\rho^r$, $r\ge 0$, using the exact sequences $(*)$ for smaller indices, we can embed all the terms $C^j$ appearing on the right hand side and such that $j$ is not a power of $\rho$ into sums of tensor products $C^\ell\otimes C^{j-\ell}$. Doing this repeatedly, and using cocommutativity to reorder the factors of the tensor products, we obtain an exact sequence:
$$0\to PC^k\to C^k\xrightarrow[]{\Delta} \bigoplus_{\beta} (C^1)^{\otimes \beta_0}\otimes\dots \otimes (C^{\rho^n})^{\otimes\beta_n}\quad(**)$$
where the sum is taken over all the $(n+1)$-tuples of nonnegative integers $(\beta_0,\dots,\beta_n)$ such that $\sum\beta_i\rho^i=k$ and $\sum\beta_i>1$. (If an $(n+1)$-tuple appears more than once, we can remove the extra occurrences in the sum without changing the exactness of the sequence).

We can now prove the first assertion. This assertion is trivial if $k$ is a power of $\rho$, so we assume that $k$ is not a power of $\rho$. Let $\beta$ be an $(n+1)$ tuple such that $\sum\beta_i\rho^i=k$ and $\sum\beta_i>1$. 
Elementary arithmetic shows that it is possible to decompose each $\beta_i$ as a sum of nonnegative integers $\beta_i=\beta_{i,i}+\dots+ \beta_{i,n}$ such that 
for $0\le i\le n$ we have 
$$\alpha_i\rho^i= \sum_{0\le j\le i}\beta_{j,i}\rho^j\;.$$
The coproduct of $C$ induces a commutative diagram, in which the bottom horizontal map is any permutation which puts the terms of the tensor product in the right order:
$$\xymatrix{
C^k\ar[rr]^-{\Delta_\alpha}\ar[d]^-{\Delta_\beta}&& \bigotimes_{0\le i\le n} (C^{\rho^i})^{\otimes \alpha_i}\ar[d]^-{\bigotimes_{0\le i\le n}\Delta}\\
(C^1)^{\otimes \beta_0}\otimes\dots \otimes (C^{\rho^n})^{\otimes\beta_n}\ar[rr]^-\simeq&&
\bigotimes_{0\le i\le n}\big(\bigotimes_{0\le j\le i}(C^{\rho^j})^{\otimes \beta_{j,i}}\big)
}$$
Thus, the morphism $\Delta$ appearing in the exact sequence $(**)$ is injective (indeed $k$ is not a power of $p$, hence $PC^k=0$) and it factors through $\Delta_\alpha$. 
Thus, $\Delta_\alpha$ must be injective. 

Next we prove the second assertion. So $k=p^n$, and by a similar reasoning the morphism $\Delta$ appearing in the exact sequence $(**)$ factors through 
$$\Delta: C^{\rho^n}\to (C^{\rho^{n-1}})^{\otimes \rho}\;.$$
Thus, the kernel of the latter is contained in $PC^{\rho^n}$. Finally, an easy explicit computation shows that this kernel is exactly $PC^{\rho^n}$.
\end{proof}

\begin{definition}\label{def-special}
We call a connected exponential functor $E$ \emph{special} if $PQE$ is simple.
\end{definition}
Recall that $PQE$ is the image of the canonical map $PE\to QE$. By theorem \ref{thm-classif-indecomp} the indecomposable reflexive strict exponential functors are special if $\kk$ is a perfect field.
\begin{lemma}\label{lm-prem-prop}
If $\kk$ is an arbitrary field. If $E$ is special with $PQE$ concentrated in degree $d$, then $E$ is indecomposable, $E^d=PQE$ and $E^i\ne 0$ if and only if $i$ is a multiple of $d$.  
\end{lemma}
\begin{proof}
If an arbitrary nontrivial $E$ is connected then let $m(E)$ be the minimal positive degree such that $E^{m(E)}\ne 0$. Then $E^{m(E)}$ is additive and nonzero. As a first consequence, any nontrivial connected $E$ has $PQE\ne 0$. Since $PQ(E_1\otimes E_2)=PQE_1\oplus PQE_2$, this shows that special exponential functors are indecomposable. As a second consequence, if $E$ is special then $E^{m(E)}=PQE$. Thus for all $k$, $E^d(V)^{\otimes k}$ is a nontrivial direct summand of $E^{kd}(V^{\oplus k}$, hence $E^{kd}\ne 0$. Conversely, we prove by induction on $k$ that $E^i=0$ for $kd<i<(k+1)d$. This holds for $i=0$ since $d=m(E)$. If this holds for all $\ell<k$, then the degree $i$ part of the isomorphism $E(V\oplus W)\simeq E(V)\otimes E(W)$ shows that $E^i$ is additive, so $E^i\subset PQE^i=0$, and the property holds for $k$.  
\end{proof}

\begin{proposition}\label{prop-indecomp-summand}
Let $E$ be a connected strict exponential functor over an arbitrary field $\kk$. Assume that $E$ is special, with $PQE$ concentrated in degree $d$. 
Then for all $k\ge 0$, $\End_{\PP_{\omega,\kk}}(E^{kd})=\kk$.
\end{proposition}
\begin{proof}
We prove the statement by induction on $k$. Lemma \ref{lm-prem-prop} describes $E^k$ in degrees $k\le 1$ hence we see directly that the statement holds (simple additive functors have endomorphism rings of dimension one). Let $k>1$ and assume that $\mathrm{End}_{\PP_{\omega,\kk}}(E^{id})$ is computed for all $i<k$. By our assumptions on $E$, $PE^{kd}=0$ or $QE^{kd}=0$. We assume that $PE^{kd}=0$ (the proof for $QE^{kd}=0$ is similar, using a dual form of lemma \ref{lm-sec-coalg}). We claim that there is a tuple $(\alpha_1,\dots,\alpha_n)$ with $n>1$ and $\sum\alpha_i=kd$ and a monomorphism 
$$\iota: E^{kd}\to E^{\alpha_1}\otimes\dots\otimes E^{\alpha_n}\;.$$ 
If $d$ is odd and $p\ne 2$ then $E$ is an exterior algebra, and we let $\iota$ be the inclusion $E^{kd}\subset (E^d)^{\otimes k}$. If $d$ is even or if $p=2$, then $E$ is commutative in the ungraded sense, and $\iota$ is given by lemma \ref{lm-sec-coalg}. Thus there is a monomorphism:
$$\End_{\PP_{\omega,\kk}}(E^{dk})\to \Hom_{\PP_{\omega,\kk}}(E^{dk},E^{\alpha_1}\otimes\dots\otimes E^{\alpha_n})\;.\qquad(*)$$
By \cite[Thm 1.7]{FFSS} the vector space on the right hand side of $(*)$ is isomorphic to the following direct sum, indexed by all $n$-tuples of nonnegative integers $\beta=(\beta_1,\dots,\beta_n)$ such that $\sum\beta_i=dk$:
$$\bigoplus_\beta \bigotimes_{i=1}^n\Hom_{\PP_{\omega,\kk}^*}(E^{\beta_i},E^{\alpha_i}) \quad(**)$$
Now by lemma \ref{lm-prem-prop} $E$ is indecomposable. Hence by proposition \ref{prop-categ-decomp}, there is an integer $a\in\mathbb{N}[p^{-1}]$ such that the degrees of $E$ are $a$ times its weights.
Since there are no nontrivial morphisms between homogeneous strict analytic functors of different weights, all the terms of the sum $(**)$ are zero except the term indexed by $\beta=\alpha$. Now by the induction hypothesis, the term of the sum indexed by $\beta=\alpha$ has dimension $0$ or $1$. Thus by injectivity of $(*)$ the endomorphism space of $E^{dk}$ has dimension $0$ or $1$. But it cannot have dimension zero since $E^{kd}\ne 0$ by lemma \ref{lm-prem-prop}.
\end{proof}

\begin{corollary}\label{cor-iso-fct}
Assume that $E_1$ and $E_2$ are special. Let $n$ be such that $E_1^n\ne 0$. If there is an isomorphism of functors $E_1^n\simeq E_2^n$ then if $0\le i\le n$ there is an isomorphism $E_1^i\simeq E_2^i$.
In particular, we have $QE_1^{i}\simeq QE_2^i$ and $PE_1^i\simeq PE_2^i$.
\end{corollary}
\begin{proof}
Assume that $PQE_i$ is concentrated in degree $d_i$. Then $n=n_id_i$ by lemma \ref{lm-prem-prop}. We have:
$$E_i^{n}(V\oplus\kk)\simeq\bigoplus_{0\le k\le n_i}E^{d_ik}_i(V)\otimes E^{n-d_ik}_i(\kk)\;.$$
The functors $E^{d_ik}$ are indecomposable, thus this isomorphism is a decomposition into indecomposable summands. If $f:E_1^n\to E_2^n$ is an isomorphism, then it induces an isomorphism between $E_1^{n}(V\oplus\kk)$ and $E_2^{n}(V\oplus\kk)$. 
But the $E^{d_ik}$ have local endomorphism rings (of dimension one!), hence the Krull-Remak-Schmidt theorem \cite[Chap I.6 Thm 1]{Gabriel} applies: there are isomorphisms between the $E^{d_1k}_1$ and the $E_2^{d_2\ell}$. 

But an isomorphism of strict analytic functors must preserve the weights. Since the exponential functors $E_i$ are indecomposable, the weights are proportional to the degrees: each $E_i^{d_ik}$ has weight $a_id_ik$ for some $a_i\in \mathbb{N}[p^{-1}]$ by proposition \ref{prop-categ-decomp}.  
In particular the $E_i^{d_ik}$, $0\le k\le n_i$ are $n_i+1$ pairwise distinct indecomposables. The uniqueness in the Krull schmidt theorem implies that $n_1=n_2$, hence $d_1=d_2=:d$. Moreover, since $E^n_1$ is isomorphic to $E^n_2$, one must have $a_1=a_2$. Thus $E^{d_1k}_1$ and $E^{d_2\ell}_2$ have the same weight if and only if $k=\ell$. Thus the Krull-Remak-Schmidt theorem actually implies that $E^{dk}_1\simeq E_2^{dk}$ for all $k$ such that $0\le kd\le n$. Moreover, if $i$ is not divisible by $d$ then $E_1^i=0=E_2^i$ by lemma \ref{lm-prem-prop}.
\end{proof}

\subsection{Proof of theorem \ref{thm-uniqueness-strict}}\label{subsec-proofthm}
By lemma \ref{lm-reduction-reflexive} it suffices to prove theorem \ref{thm-uniqueness-strict} when $E$ is reflexive. By corollary \ref{cor-signature}, it suffices to prove that we can compute all the fake truncations $\phi_k(\sigma(E))$, $k\ge 0$ from the graded strict analytic functor $\OO E$.

Let $F$ be a graded analytic functor, such that $\dim_\kk F^k(V)<+\infty$ for all $V$. We now describe how to extract a multiset $\alpha_k(F)$ of pairs of graded additive functors from $F$. 
We first decompose $F^k$ into a direct sum of indecomposable functors $F_i^k$. 
$$F^k=\bigoplus_{i\in I}F_i^k\;.$$
Such a decomposition exists and is unique by the next proposition.

\begin{proposition}[Krull-Remak-Schmidt property]\label{prop-KRS-P}
Let $G$ be a strict analytic functor, such that for all $k$ and all vector spaces $V$, $\dim_\kk w_kG(V)<\infty$. Then $G$ admits a unique decomposition into a direct sum of indecomposable graded strict analytic functors.
\end{proposition}
\begin{proof}
For all $k$, evaluation on $\kk^k$ induces an equivalence of categories $\PP_{k,\kk}\simeq S(k,k)-\mathrm{Mod}$, where $S(k,k)$ is a Schur algebra. As a consequence, any functor $G$ of $\PP_{k,\kk}$ with finite dimensional values has a unique decomposition as a finite direct sum of indecomposable functors, whose endomorphism ring is a local ring (see e.g. \cite[Chap 1, 1.4]{Benson}). Applying this result for each summand $w_kG$ of $G$, we obtain the result.
\end{proof}

We let $J\subset I$ the subset of indices $i$ such that $F_i^k$ contains an additive functor or has an additive functor as a quotient. Now we let $K\subset J$ be the subset of the indices $i\in J$ such that $F_i^k$ is isomorphic to $E^k$ for some reflexive indecomposable strict exponential functor $E$. For all $i\in K$ we let 
$$\pi_k(F_i^k)= \left(\bigoplus_{0\le i\le k}PE^i,\bigoplus_{0\le i\le k}QE^i\right)\;.$$
Observe that this is well-defined (it does not depend on the choice of $E$ such that $E^k\simeq F_i^k$) by corollary \ref{cor-iso-fct}.
We let $\alpha_k(F)$ be the multiset
$$\alpha_k(F)=\{\pi_k(F_i^k)\;, i\in K\;\}\;.$$
In particular, $\alpha_k(F)$ is empty if $K$ is empty. The following proposition finishes the proof of theorem \ref{thm-uniqueness-strict}.
\begin{proposition}
For all reflexive strict exponential functor $E$ and all $k\ge 0$ we have $\phi_k(\sigma(E))=\alpha_k(\OO E)$.
\end{proposition}
\begin{proof}
We decompose $E$ as a tensor product of indecomposable functors: $E\simeq \bigotimes_{i\in I}E_i$. Then $E^k$ decomposes as
$$E^k\simeq \bigoplus_{i\in I}E_i^k\;\oplus\; \text{ other summands.}$$
The other summands in this decomposition are direct summands of tensor products $E_{i_1}^{k_1}\otimes\cdots\otimes E_{i_n}^{k_n}$ with at least two nonconstant factors. By Pirashvili's vanishing lemma (fact \ref{lm-cancel}), such functors (hence their indecomposable summands) cannot have nontrivial additive subfunctors or quotients. Hence:
$$\alpha_k(\OO E)=\alpha_k(\bigotimes_{i\in I}\OO E_i)=\alpha_k(\bigoplus_{i\in I}\OO E_i)=\bigcup_{i\in I}\alpha_k(\OO E_i)\;.$$
On the other hand 
$$\phi_k(\sigma(E))=\bigcup_{i\in I}\phi_k(\sigma(E_i))\;.$$
Thus it remains to check that for all $i$, $\alpha_k(\OO E_i)=\phi_k(\sigma(E_i))$, which follows from the definition of $\alpha_k$.
\end{proof}

\section{Filtrations of exponential functors}\label{sec-filtr}

In this section, we return to the general setting, i.e. $\kk$ is a commutative ring and $\FF$ denotes either $\PP_{\omega,\kk}$ or $\Fct(\V,\Mod)$, where $\V$ is a small additive category. In this section we give give basic results regarding filtrations of exponential functors, and two fundamental examples, namely (a variant of) the coradical filtration, and the augmentation filtration. Theorems \ref{thm-corad} and \ref{thm-aug} show that these two filtrations do not depend on the exponential structure of $E$, but only on the underlying functor. We use these results to establish some properties of graded functors which admit an exponential structure in section \ref{subsec-exp-struct}. 
Throughout section \ref{sec-filtr}, we rely on the notion of Eilenberg-Mac Lane degree of a functor, recalled in appendix \ref{app-Fct}. The terms `Eilenberg-Mac Lane degree' will often be shortened into `EML-degree', and the EML-degree of a functor $F$ will be denoted by $\deg_{\EML}F$.

\subsection{Filtered exponential functors}
Let $V_1,V_2$ be two graded $\kk$-modules, equipped with filtrations by graded submodules 
$\cdots\subset F_k V_i\subset F_{k+1}V_i\subset \cdots$. A morphism $f:V_1\to V_2$ preserves the filtrations if $f(F_kV_1)\subset F_kV_2$ for all $k$.
The tensor product filtration on $V_1\otimes V_2$ is defined as usual by 
$$F_k(V_1\otimes V_2)=\sum_{i+j=k}F_iV_1\overline{\otimes} F_jV_2\;,$$
where the symbol `$\overline{\otimes}$' indicates that we are considering the image of $F_iV_1\otimes F_jV_2$ in $V_1\otimes V_2$.
An \emph{algebra filtration} of a graded $\kk$-algebra $A$ is a filtration by graded $\kk$-submodules such that the multiplication $\mu:A\otimes A\to A$ and the unit $\eta:\kk\to A$ preserve the filtrations ($\kk$ is endowed with the filtration such that $F_0\kk=\kk$, $F_{-1}\kk=0$). 
Similarly a \emph{coalgebra filtration} of a graded $\kk$-coalgebra $C$ is a filtration by graded $\kk$-submodules such that the comultiplication and the counit preserve the filtrations. Finally a \emph{Hopf filtration} of a graded Hopf $\kk$-algebra is a filtration which is both an algebra and a coalgebra filtration, and such that the antipode preserves the filtration. These notions evidently generalize to functorial graded algebras, coalgebras and Hopf algebras, e.g. an algebra filtration of a functorial graded $\kk$-algebra $A$ is a filtration by graded subfunctors $F_iA$, such that for all $V$, the $F_iA(V)$ yield an algebra filtration of $A(V)$.

\begin{proposition}\label{prop-caract-Hopf-filtr}
Let $E$ be an exponential functor, filtered by a family of graded subfunctors $F_kE$, $k\in \Z$. The following assertions are equivalent:
\begin{enumerate}
\item[(F1)] The $F_kE$, $k\in \Z$,  yield a Hopf filtration.
\item[(F2)] The $F_kE$, $k\in \Z$,  yield an algebra filtration  and a coalgebra filtration.
\item[(F3)] The isomorphisms $\phi:E(V)\otimes E(W)\xrightarrow[]{\simeq} E(V\oplus W)$ and $u:\kk\xrightarrow[]{\simeq} E(0)$ are isomorphisms of filtered modules for all $V,W$.
\end{enumerate}
\end{proposition}
\begin{proof}
To see that (F1)$\Leftrightarrow$(F2), we simply observe that $\chi_V=E(-\Id_V)$ being defined using the functoriality of $E$, $\chi$ automatically preserves all subfunctors of $E$, in particular it preserves any filtration of $E$. Now we prove (F2)$\Leftrightarrow$(F3).
By lemma \ref{lm-defbis} the isomorphisms $\phi$ and $u$  can be reconstructed from the multiplication and the unit of $E$ and vice-versa. Hence the $F_iE$ form an algebra filtration if and only if $\phi$ and $u$ preserve the filtrations. Similarly, the $F_iE$ form a coalgebra filtration if and only if $\phi^{-1}$ and $u^{-1}$ preserve the filtrations. Thus  (F2) is equivalent to (F3).
\end{proof}

Given a filtered graded $\kk$-module, we let $\mathrm{gr}\, V=\bigoplus_{k\in\Z}F_kV/F_{k-1}V$. This is a bigraded object with $\mathrm{gr}_k V^i= F_kV^i/F_{k-1}V^i$, but we only consider it as a graded object with homogeneous summand of degree $i$ given by $\mathrm{gr} V^i = \bigoplus_{k\in \Z}\mathrm{gr}_k V^i$. If $V_1$ and $V_2$ are two filtered graded $\kk$-modules, there is a canonical surjective map
 $\mathrm{gr}\, V_1 \otimes \mathrm{gr}\, V_2\to \mathrm{gr}\, (V_1\otimes V_2)$. Thus for all filtered graded algebras $A$, $\mathrm{gr}\,A$ is again a graded algebra, with multiplication ($\mu$ denotes the multiplication of $A$):
$$\mathrm{gr}\,A\otimes \mathrm{gr}\,A\xrightarrow[]{\mathrm{can}} \mathrm{gr}\,(A\otimes A)\xrightarrow[]{\mathrm{gr}(\mu)}\mathrm{gr}\,A\;.$$
All this holds for functorial graded algebras as well. Thus, for all exponential functors $E$ equipped with a Hopf filtration, $\mathrm{gr}\, E$ is a functorial graded algebra. 
\begin{lemma}\label{lm-grExp}Let $E$ be an exponential functor equipped with a Hopf filtration. Then $\mathrm{gr}\,E$ is an exponential functor if and only if the canonical map $\mathrm{gr}\,E\otimes \mathrm{gr}\,E\to  \mathrm{gr}\,(E\otimes E)$ is an isomorphism.
\end{lemma}
\begin{proof}
Since $\phi$ and $u$ are isomorphisms of filtered functors, $\mathrm{gr}(u)$ and $\mathrm{gr}(\phi)$ are isomorphisms. So condition (A1) of lemma \ref{lm-defbis} is satisfied. Moreover
\begin{align*}
\mathrm{gr}(\mu)\circ \mathrm{can}\circ \big(\mathrm{gr} \,E(\iota_V)\otimes \mathrm{gr}\, E(\iota_W)\big)& = \mathrm{gr}(\mu)\circ \mathrm{gr}(E(\iota_V)\otimes E(\iota_W))\circ \mathrm{can}\\
&= \mathrm{gr}(\phi)\circ \mathrm{can}
\end{align*} 
is an isomorphism if and only if $\mathrm{can}:\mathrm{gr}\, E(V)\otimes \mathrm{gr}\,E(W)\to \mathrm{gr}\,(E(V)\otimes E(W))$ is an isomorphism for all pairs $(V,W)$. The latter is an isomorphism if and only if it is an isomorphism for all pairs $(V,V)$. Thus condition (A2) of lemma \ref{lm-defbis} is satisfied if and only if the canonical map $\mathrm{gr}\,E\otimes \mathrm{gr}\,E\to  \mathrm{gr}\,(E\otimes E)$ is an isomorphism.
\end{proof} 
\begin{remark}\label{rk-hyp-sat} It follows from \cite[Chap III, Exercises, \S 2 ,6)]{BourbakiAlgCom} that
the canonical map $\mathrm{gr}\,E\otimes \mathrm{gr}\,E\to  \mathrm{gr}\,(E\otimes E)$ is an isomorphism if  $\mathrm{gr}\,E(V)$, $E(V)$ and $\bigcup_{k\in\Z}F_{k} E(V)$ are flat $\kk$-modules for all $V$.
\end{remark}

\begin{remark}\label{rk-short}
Let $\FF=\Fct(\Proj_R,\Mod)$, and let $\FF-\Exp^{\mathrm{filtr}}_c$ denote the category whose objects are ordinary exponential functors equipped with a Hopf filtration such that the canonical map $\mathrm{gr}\,E\otimes \mathrm{gr}\,E\to  \mathrm{gr}\,(E\otimes E)$ is an isomorphism, and whose morphisms are the morphisms of exponential functors which preserve the filtrations. Similarly let $\HH^{\mathrm{filtr}}$ be the category of graded bicommutative Hopf algebras equipped with a nice Hopf filtration. There a commutative diagram, in which the bottom horizontal equivalence of categories is provided by theorem \ref{thm-classif-ord}:
$$\xymatrix{
\FF-\Exp^{\mathrm{filtr}}_c \ar[d]^-{\mathrm{gr}}\ar[rr]^-{\mathrm{ev}_R}&& {_R}\HH^{\mathrm{filtr}}\ar[d]^-{\mathrm{gr}}\\
\FF-\Exp_c\ar[rr]^-{\mathrm{ev}_R}_-\simeq && {_R}\HH
}$$
As we shall see below in propositions \ref{prop-stat-pol} and \ref{prop-stat-copol}, filtrations of exponential functors are typically infinite. However, after evaluation on $R$, they may become finite objects. Hence, the top right corner of the diagram typically yield shorter computations of associated graded objects.
\end{remark}

\subsection{Two fundamental filtrations}\label{subsec-fund-filtr}
We first consider a variant of the classical coradical filtration of a graded $\kk$-coalgebra $(C,\Delta,\epsilon)$. To be more specific, assume that $C$ has a coaugmentation $\eta:\kk\to C$ (i.e. $\eta$ is a morphism of coalgebras). The $\kk$-module $C$ then has a decomposition as a direct sum 
$$C=\mathrm{Im}\,\eta\oplus \ker\epsilon =\kk\oplus \overline{C} \;.$$
We let $\pi = \Id-\eta\epsilon$ be the projection onto the augmentation ideal $\overline{C}$. The reduced comultiplication $\overline{\Delta}:\overline{C}\to \overline{C}^{\otimes 2}$ is defined as the composition:
$$\overline{C}\hookrightarrow C\xrightarrow[]{\Delta} C\otimes C\xrightarrow[]{\pi^{\otimes 2}}\overline{C}\otimes \overline{C}\;.$$
One checks that $\overline{\Delta}$ is coassociative, its $k$-fold iteration $\overline{\Delta}_{k+1}:\overline{C}\to \overline{C}^{\otimes k+1}$ equals the composition $\pi^{\otimes k+1}\circ \Delta_{k+1}$.
We define an increasing filtration of the $\kk$-module $C$ by letting $P_kC=\kk\oplus \ker \overline{\Delta}_{k+1}$: 
$$\kk=P_0 C \subset P_1 C\subset\dots\subset  P_i C\subset \dots\subset P_{+\infty} C=\bigcup_{k\ge 0}P_k C\subset C\;.$$
\begin{remark}\label{rk-sweedler}
Assume that $C$ is a flat $\kk$-module. Then 
$P_kC$ equals the kernel of the composition:
$$C\xrightarrow[]{\Delta}C\otimes C\xrightarrow[]{\mathrm{can}\otimes \pi} (C/P_{k-1}C)\otimes \overline{C}\;.$$
In particular, if $\kk$ is a field, this filtration identifies with the filtration $\bigwedge^k(\kk 1)$, $k\ge 0$ of \cite{Sweedler}. Moreover \cite[Rk a) and Prop 11.1.1 p. 226]{Sweedler} imply that $C$ is pointed irreducible if and only if this filtration is exhaustive - in this case the filtration coincides with the coradical filtration of $C$.
\end{remark}
 
\begin{lemma}\label{lm-cog}
If $\mathrm{gr}\, C$ and $C/P_{+\infty} C$ are flat $\kk$-modules, then the filtration $P_kC$, $k\ge 0$ is a filtration of coalgebras.
\end{lemma}
\begin{proof}
Copy the proof of \cite[Thm 9.1.6]{Sweedler}, but use \cite[Chap III, Exercices, \S 2 ,6)c)]{BourbakiAlgCom} instead of \cite[Lm 9.1.5]{Sweedler}.
%
\end{proof}

This construction applies as well to functorial augmented graded coalgebras, in particular to an exponential functor $E$. The definition of $P_kE$ uses the map $\overline{\Delta}_{k+1}$, so  
this filtration seems strongly linked with the coalgebra structure of $E$. The next result generalizes the result on primitives in lemma \ref{lm-additivite} since $P_1E=\kk\oplus PE$. It shows that this filtration does not depend on the structure morphisms of $E$.

\begin{theorem}\label{thm-corad}
Let $E$ be an exponential functor. Then each functor
$P_k E$ is polynomial of EML-degree less or equal to $k$, and moreover it is the largest such subfunctor of the functor $E$.
\end{theorem}
\begin{proof}
By definition of the reduced diagonal, one has a commutative square:
$$\xymatrix{
\overline{E}(V)\ar[d]^-{\overline{\Delta}_E}\ar[r]^-{\iota}& E(V)\ar[d]^-{\Delta_E}\\
\overline{E}(V)^{\otimes {n+1}}&\ar[l]_-{\pi^{\otimes n+1}}E(V)^{\otimes {n+1}}
}$$
where $\iota$ is the canonical inclusion of $\overline{E}$ into $E$ and $\pi$ the canonical projection of $E$ onto $\overline{E}$. Since $E$ is exponential, the $(n+1)$-th cross effect of $E$ is
$$\Cr_{n+1} E(V_1,\dots,V_{n+1})=\overline{E}(V_1)\otimes\dots\otimes \overline{E}(V_{n+1})$$
and the canonical projection of $E(V_1\oplus\dots\oplus V_{n+1})$ onto $\Cr_{n+1}E(V_1,\dots,V_{n+1})$ is $\pi^{\otimes n+1}$. Hence the canonical map $E(V)\to \Cr_{n+1} E(V,\dots,V)$ is equal to $0+\overline{\Delta}_E:\kk\oplus \overline{E}(V)\to \overline{E}(V)^{\otimes {n+1}}$ and by fact \ref{fact-ker}, $P_nE$ is the largest polynomial subfunctor of $E$ of EML-degree less or equal to $n$.
\end{proof}

\begin{corollary}\label{cor-filtr-P}
Let $E$ be an exponential functor. Then the 
filtration $P_kE$, $k\ge 0$, yields an algebra filtration of $E$. 
Moreover, if the $\kk$-modules $\mathrm{gr}\,E(V)$ and $E(V)/P_{+\infty}E(V)$ are $\kk$-flat for all $V$, then this filtration is a Hopf filtration of $E$, and $\mathrm{gr}\,E$ is an exponential functor.
\end{corollary}
\begin{proof}
By fact \ref{fact-tens}, the degree of $P_iE\otimes P_jE$ is less or equal to $i+j$. The multiplication being a morphism of functors, it preserves the polynomial filtrations, hence it sends this tensor product into $P_{i+j}E$. Thus, the $P_kE$, $k\ge 0$, form an algebra filtration of $E$. Moreover if the flatness assumptions are satisfied, the filtration is also a filtration of coalgebras by lemma \ref{lm-cog}, it is a Hopf filtration by proposition \ref{prop-caract-Hopf-filtr}. Finally, the canonical map $\mathrm{gr}\,E\otimes \mathrm{gr}\,E\to \mathrm{gr}\,(E\otimes E)$ is an isomorphism by remark \ref{rk-hyp-sat}, hence $\mathrm{gr}\,E$ is exponential by lemma \ref{lm-grExp}.
\end{proof}

If $A$ is a graded $\kk$-algebra, with augmentation $\epsilon:A\to\kk$, we let $\overline{A}=\ker\epsilon$ be the augmentation ideal of $A$. Let $Q_0A=A$ and let $Q_{-k}A$ be the $k$-th power of the ideal $\overline{A}$ if $k>0$. Then the $Q_{-k}A$ for an algebra filtration of $A$ usually called the \emph{augmentation filtration of $A$}:
$$Q_{-\infty}A=\bigcap_{k\ge 0}Q_{-k}A \subset \cdots\subset Q_{-2}A\subset Q_{-1}A=\overline{A}\subset Q_0A=A\;.$$
\begin{lemma}\label{lm-tens-aug}
For all augmented graded $\kk$-algebras $A$ and $B$, 
the augmentation filtration of $A\otimes B$ coincides with the tensor product of the augmentation filtration of $A$ and of the augmentation filtration of $B$.
\end{lemma}
\begin{proof}
The augmentation ideal $\overline{A\otimes B}$ coincides with the image of the map $\mathrm{can}:\overline{A}\otimes B \oplus A\otimes \overline{B}\to A\otimes B$. Thus $Q_{-k}(A\otimes B)$ is the image of the map
$$ (\overline{A}\otimes B \oplus A\otimes \overline{B})^{\otimes k}\xrightarrow[]{\mathrm{can}^{\otimes k}} (A\otimes B)^{\otimes k}\xrightarrow[]{\mathrm{mult}_{A\otimes B}} A\otimes B\;.$$
Developing the tensor products, we see that this image is the sum of the images of the maps for $0\le i\le k$ 
$$\overline{A}^{\otimes i}\otimes A^{\otimes k-i}\otimes \overline{B}^{\otimes k-i}\otimes B^{\otimes i} \to A^{\otimes k}\otimes B^{\otimes k}\xrightarrow[]{\mathrm{mult}_A\otimes \mathrm{mult}_B}A\otimes B\;.$$
The image of the latter is $Q_{-i}A\overline{\otimes} Q_{-(k-i)}B$, which proves the lemma.
\end{proof}

One can consider augmentation filtrations for exponential functors (we have already used this in the proof of theorem \ref{thm-twist}). The proof of the next theorem is formally dual to the proof of theorem \ref{thm-corad} and is therefore omitted. Observe that theorem \ref{thm-aug} generalizes the result on indecomposables in lemma \ref{lm-additivite} since $QE\oplus \kk = Q_{-1}E/Q_{-2}E\oplus \kk =E/Q_{-2}E$. Theorem \ref{thm-aug} also shows that the augmentation filtration of an exponential functor is independent of the exponential structure of $E$.
\begin{theorem}\label{thm-aug}
Let $E$ be an exponential functor. Then for all positive integer $k$, 
$\deg_\EML E/Q_{-k}E <k$. Moreover, $Q_{-k}E$ is the smallest subfunctor of $E$ among the subfunctors $F$ of $E$ such that the quotient $E/F$ is polynomial of EML-degree less than $k$.
\end{theorem}

\begin{corollary}
The augmentation filtration of an exponential functor $E$ is a Hopf filtration. Moreover if $E(V)$ and $\mathrm{gr}\,E(V)$ are $\kk$-flat modules for all $V$, then $\mathrm{gr}\,E$ is an exponential functor.
\end{corollary}
\begin{proof}
We prove that condition (3) of proposition \ref{prop-caract-Hopf-filtr} is satisfied.
Since $\phi_{V,W}$ is a retract of $\phi_{V\oplus W,V\oplus W}$ for all pairs $(V,W)$, it suffices to prove that $\phi_{V,V}: E(V)\otimes E(V)\xrightarrow[]{\simeq} E(V\oplus V)$ is an isomorphism of filtered modules for all $V$. Let $E'(V)=E(V\oplus V)$. As $\phi_{V,V}$ is an isomorphism of functors, it sends the smallest functor $F$ of $E^{\otimes 2}$ such that $\deg_\EML E^{\otimes 2}/F<k$ to the smallest functor $F'$ of $E'$ such that $\deg_\EML E'/F'<k$. Thus $\phi_{V,V}$ maps $Q_{-k}(E^{\otimes 2})(V)$ isomorphically to $Q_{-k}E'(V)$. But $Q_{-k}(E^{\otimes 2})(V)$ coincides with the $k$-th term of the tensor square of the augmentation filtration of $\overline{E}(V)$ by lemma \ref{lm-tens-aug}, thus $\phi_{V,V}$ is an isomorphism of filtered modules.
\end{proof}

Now we give a concrete application, which parallels the result of Eilenberg and Mac-Lane \cite{EML2} that the singular homology of Eilenberg-Mac Lane spaces are analytic and coanalytic functors.

\begin{example}[Morava K-theory of classifying spaces]\label{ex-Morava}
Let $q=p^d$ for an odd prime $p$, and let $n>0$. For all $L\in \Proj_\Z$, we let
$$E(L)=\overline{K}(n)_{\overline{*}}B(L/qL)\;.$$
By \cite[Thm 5.7]{RavenelWilson} the coalgebra $E(L)$ is isomorphic to $\Gamma_{nj}(L/pL)$ hence its coradical filtration is exhaustive. Thus $E$ is analytic by theorem \ref{thm-corad}. Also, if $n>1$ then the algebra $E(L)$ is a tensor product of truncated power algebras, hence the augmentation filtration is Hausdorff, hence $E$ is coanalytic by theorem \ref{thm-aug}. The exponential functors obtained as associated graded objects of these filtrations can be simply recovered by the technique of remark \ref{rk-short}. For example, if $n=2$ and $q=p$, then $\mathrm{gr}\,E(\kk)=\Fp[x]/x^2$ where `$\mathrm{gr}$' refers to the coradical filtration and $x$ is primitively generated, hence if `$\mathrm{gr}$' refers to the polynomial filtration of $E$ and $S_2$ is the symmetric algebra modulo the $p^2$-th powers then
$$\mathrm{gr}\,E(L)= S_2(L/qL)\;.$$ 
\end{example}

\subsection{On the polynomial and copolynomial filtrations}\label{subsec-exp-struct}

We now investigate in more details the properties of the two fundamental filtrations presented in section \ref{subsec-fund-filtr}. As we have proved that these filtrations coincide with the polynomial and copolynomial filtrations of the underlying functors, the results we give here show that not all functors may be equipped with an exponential structure. 

\begin{proposition}\label{prop-conn}
Let $E$ be a nontrivial connected exponential functor. Let $m$ be the minimal positive integer such that $E^m\ne 0$. For all integers $i$ and $k$ such that $mk\ge i$ we have 
$$ P_kE^i = P_{+\infty}E^i = E^i\quad \text{ and }\quad Q_{-(k+1)}E^i=Q_{-\infty}E^i=0\;.$$
If $\kk$ is a domain and $E^m(V)$ is torsion free for all $V$, then $\deg_{\EML} E^{mk} = k$ for all $k$. In particular, 
$$ P_{k-1}E^{mk}\varsubsetneq P_kE^{mk}=E^{mk}\quad \text{ and }\quad 0=Q_{-(k+1)}E^{mk}\varsubsetneq Q_{-k}E^{mk}\;.$$
\end{proposition}
\begin{proof}
The first part follows from the fact that $\overline{E}^{\otimes k}$ is zero in degrees less than $km$. Now let $\kk$ be a domain and assume $\deg_{\EML} E^{mk}<k$. Let $\oplus^k$ be the additive functor $\oplus^k(V)=V^{\oplus k}$. Then $\deg_{\EML} E^{mk}\circ {\oplus^k}<k$. This contradicts that $E^{mk}\circ {\oplus^k}$ contains $(E^m)^{\otimes k}$ as a direct summand and the latter has EML-degree $k$ by \ref{fact-tens-2}.
\end{proof}

\begin{proposition}\label{prop-stat-pol}
Let $E$ be an exponential functor.
\begin{enumerate}
\item[(1)] If $E^0(V)$ is $\kk$-flat for all $V$ and $P_1E= \kk$, then $P_{+\infty}E= \kk$.
\item[(2)] If $E(V)$ is $\kk$-flat for all $V$ and $P_iE=P_{i+1}E$, then $P_{+\infty}E=P_iE$. 
\item[(3)] If $\kk$ is a domain, and $E(V)$ is 
$\kk$-flat for all $V$. Then either $P_{+\infty}E=\kk$ or the filtration is strictly increasing:
$$\kk=P_0E\varsubsetneq P_1E \varsubsetneq \dots \varsubsetneq P_{i}E \varsubsetneq P_{i+1}E \varsubsetneq\dots \varsubsetneq P_{+\infty}E\;.$$
\end{enumerate}
\end{proposition}
\begin{proof}
Proof of (2): Since $\overline{E}(V)$ is $\kk$-flat for all $V$, $P_{k+1}E$ is the kernel of
$$ E\xrightarrow[]{\Delta}E\otimes E\xrightarrow[]{\pi\otimes (\Id-\epsilon)}(E/{P_kE})\otimes \overline{E}\;.$$
An easy induction then shows that $P_kE=P_iE$ for all $k\ge i$.
Proof of (1): if $P_1E=\kk$ then $P_1(E^0)=P_1E\cap E^0=\kk$. So $E^0=\kk$ by (2), i.e. $E$ is connected. Now by proposition \ref{prop-conn} if $E$ was nontrivial and $m$ was the minimal positive integer such that $E^m\ne 0$, then $E^m\subset P_1E^m$ contradicting $P_1E=\kk$.
Proof of (3): assume that the filtration is not strictly increasing. Thus $P_{+\infty}E=P_kE$ for some $k\ge 0$ by (2). Hence, for all $\ell> k$ there is a short exact sequence
$$0\to P_kE\to E\xrightarrow[]{\overline{\Delta}_{\ell}} \overline{E}^{\otimes\ell}\;.\qquad (*)$$
Let $\oplus^2$ be the additive functor $V\mapsto V^{\oplus 2}$, and consider the exponential functor $E\circ \oplus^2$. By definition $P_\ell(E\circ \oplus^2)$ is the kernel of $\overline{\Delta}_{\ell}\circ\oplus^2$.
Hence, precomposing the short exact sequence $(*)$ by $\oplus^2$, we see that $P_\ell(E\circ \oplus^2)=(P_kE)\circ \oplus^2$. Thus the biggest polynomial subfunctor of $E\circ\oplus^n$ is $(P_kE)\circ \oplus^2$, which has EML-degree $k$. Now we claim that we must have $k=0$. Indeed, since $\kk$ is a domain $E(V)$ is torsion free for all $V$, hence $P_kE(V)$ is also torsion free. Thus the tensor square of the inclusion $P_kE(V)^{\otimes 2}\to E(V)^{\otimes 2}$ is injective (tensor by the fraction field to prove injectivity). Now by fact \ref{fact-tens-2}, $P_kE^{\otimes 2}$ has EML-degree $2k$. Since $E^{\otimes 2}\simeq E\circ\oplus^2$, we must have $2k\le k$, which implies that $k=0$.
\end{proof}

\begin{proposition}\label{prop-stat-copol}
Let $E$ be an exponential functor.
\begin{enumerate}
\item[(1)] If $Q_{-i}E=Q_{-(i+1)}E$ for some integer $i$, then $Q_{-\infty}E=Q_{-i}E$. 
\item[(2)] If $\kk$ is a domain, and the $\kk$-module $E(V)/Q_{-\infty}E(V)$ is torsion-free  for some $V$. Then either $Q_{-\infty}E=Q_{-1}E$ or the augmentation filtration is strictly decreasing:
$$\dots\varsubsetneq \dots \varsubsetneq Q_{-(i+1)}E \varsubsetneq Q_{-i}E \varsubsetneq\dots \varsubsetneq Q_{-1}E=\overline{E}\;.$$
\end{enumerate}
\end{proposition}
\begin{proof}
The first statement is obvious, the proof of the second one is formally dual to the proof of proposition \ref{prop-stat-pol}(3).
\end{proof}

\begin{remark}
Propositions \ref{prop-stat-pol}(3) and \ref{prop-stat-copol}(2) fail without assumption on the torsion. For example if $\kk=\Z$, $\V=\Proj_{\Z}$, and $A(-)=-\otimes \mathbb{Q}/\Z$, a counter-example is given by the exponential functor
$$E(V)= S(A(V)) =\Z\,\oplus\, A(V)\;.$$
\end{remark}

We end by a few remarks on analytic exponential functors, i.e those for which $E=P_{+\infty}E$.
Proposition \ref{prop-conn} shows that connected exponential functors are analytic. Hence strict exponential functors are analytic (they can be regraded so that they become connected as explained in remark \ref{Rk-classif-carzero}). On the contrary, ordinary exponential functors are not always analytic. For example if $G:\V\to \mathrm{Ab}$ is a nontrivial additive functor, then the $\kk$-group algebra on $G$ yields an exponential functor $\kk G$ such that $\overline{\Delta}_{k+1}$ is injective for all $k$. Hence $P_{+\infty}\kk G=\kk\ne \kk G$. 
\begin{proposition}\label{prop-analytic-complement} Let $E$ be an ordinary exponential functor.
\begin{enumerate}
\item Assume that the $\kk$-modules $\mathrm{gr}\, E(V)$ and $E(V)/P_{+\infty}E(V)$ are $\kk$-flat for all $V$. Then $P_{+\infty}E$ is an exponential functor and the inclusion $P_{+\infty}E\hookrightarrow E$ is a morphism of exponential functors.
\item If $\kk$ is an algebraically closed field, let $G(E):\V\to \mathrm{Ab}$ be the functor of grouplike elements of $E$. There is an isomorphism of exponential functors:
$$E\simeq P_{+\infty}E\otimes \kk G(E)\;.$$
\end{enumerate}
\end{proposition}
\begin{proof}Proof of (1): under the flatness assumptions, the $P_kE$ yield a Hopf filtration of $E$ by corollary \ref{cor-filtr-P}. Moreover by our flatness assumptions for all $k,\ell \in \mathbb{N}\cup \{+\infty\}$ there is a canonical isomorphism $P_kE\otimes P_\ell E\simeq P_kE\overline{\otimes} P_\ell E$. Using that $\colim_{k,\ell} P_kE\otimes P_\ell E = P_{+\infty} E\otimes P_{+\infty} E$, we then see that the Hopf algebra operations of $E$ make  $P_{+\infty}E$ into a functorial Hopf subalgebra of $E$. The result now follows from lemma \ref{lm-stabilite}. Proof of (2): as explained in remark \ref{rk-sweedler}, an exponential functor $E$ is pointed irreducible if and only if it is analytic. Thus $P_{+\infty}E(V)$ is the maximal pointed irreducible Hopf subalgebra of $E(V)$. The result now follows from \cite[Thm 8.1.5]{Sweedler}.
\end{proof}

\section{A weak analogue of theorem \ref{thm-uniqueness-strict}}\label{sec-uniqueness} 
In this section $\FF$ stands for $\Fct(\V,\Mod)$, where $\V$ is a small additive category of characteristic $p$ and $\kk$ is a perfect field of characteristic $p$. 
We consider the functor which forgets the exponential structure:
$$\OO:\FF-\Exp_c\to \FF^*\;.$$
Our aim is to prove a weak analogue of theorem \ref{thm-uniqueness-strict} in this context.

\begin{definition}
We say that an exponential functor $E$ is \emph{absolutely reflexive} if it is connected, and for all $i$, $QE^i$ has a finite composition series, whose factors are absolutely simple (i.e. their endomorphism ring equals $\kk$).
\end{definition}

By reiterating the proof of lemma \ref{lm-fin}, one proves the following characterization of absolute reflexivity.

\begin{lemma}\label{lm-caract}
A connected exponential functor $E$ is absolutely reflexive if and only if $PE^i$ has a finite composition series, whose factors are absolutely simple for all $i$.
\end{lemma}

\begin{remark}
If all the simple objects of $\FF_\add$ are absolutely simple, then $E$ is absolutely reflexive if and only if it is reflexive in the sense of definition \ref{def-reflexive}. This happens e.g. if $\V=\Proj_R$ and $R$ and $\kk$ are finite fields, or if $R$ is a finite dimensional semi-simple $\mathbb{F}_p$-algebra and $\kk$ is algebraically closed.
\end{remark}

Recall that a connected exponential functor $E$ is primitively generated if the canonical map $S_{\pm}(PE)\to E$ is an epimorphism. Equivalently the canonical map $PE\to QE$ is an epimorphism. Similarly, $E$ is \emph{cogenerated by its indecomposables} if the canonical map $E\to \Gamma_{\pm}(QE)$ is a monomorphism. Equivalently the canonical map $PE\to QE$ is a monomorphism. The main result of the section is the following a weak analogue of theorem \ref{thm-uniqueness-strict}.
\begin{theorem}\label{thm-uniqueness-weak}
Let $\kk$ be a perfect field of characteristic $p$, let $\V$ be a small additive category of characteristic $p$ and let $\FF=\Fct(\V,\Mod)$. Assume that $\FF_\add$ has homological dimension zero.
Let $E$ be an absolutely reflexive exponential functor, such that $E$ is primitively generated or cogenerated by its indecomposables. If there is an exponential functor $E'$ such that $\OO E\simeq \OO E'$ then $E\simeq E'$.
\end{theorem}
\begin{corollary}\label{cor-weak}
Let $E$ and $E'$ be exponential functors. Assume that $E$ is absolutely reflexive, and that there is an isomorphism $\OO E\simeq \OO E'$. Then there is an isomorphism of exponential functors $\mathrm{gr}\,E\simeq \mathrm{gr}\,E'$, where $\mathrm{gr}$ refers to the augmentation filtration or to the coradical filtration.
\end{corollary}
\begin{proof}
By theorem \ref{thm-alg} or \ref{thm-aug}, an isomorphism $\OO E\simeq \OO E'$ must preserve the filtration, hence it induces an isomorphism $\OO(\mathrm{gr}\,E)\simeq \OO(\mathrm{gr}\,E')$. The corollary then follows directly from theorem \ref{thm-uniqueness-weak} applied to $\mathrm{gr}\,E$ and $\mathrm{gr}\,E'$. However, in order to be able to apply theorem \ref{thm-uniqueness-weak}, it remains to prove that $\mathrm{gr}\,E$ is absolutely reflexive. If we consider the augmentation filtration then this is obvious since $Q(\mathrm{gr}\,E)=QE$. If we consider the coradical filtration, we use that in this case $P(\mathrm{gr}\,E)=PE$ and we use lemma \ref{lm-caract}. 
\end{proof}
The remainder of the section is devoted to the proof of theorem \ref{thm-uniqueness-weak}.
We mainly copy the proof of theorem \ref{thm-uniqueness-strict}, so we outline the proof and try to emphasize on the differences.

\subsection{The signatures} 
Absolute reflexivity, as well as being primitively generated or cogenerated by indecomposables are properties of the underlying functor of an exponential functor by lemma \ref{lm-additivite}. Thus if $\OO E'\simeq \OO E$ then $E'$ inherits these properties from $E$. 
Thus $E$ and $E'$ both decompose uniquely as a tensor products of indecomposables by theorem \ref{thm-classif-indecomp}. We define their signatures $\sigma(E)$ and $\sigma(E')$ as in section \ref{subsec-sign}. In order to prove that  $E\simeq E'$ it suffices to prove that for all $k\ge 0$, the fake truncations $\phi_k(\sigma(E))$ and $\phi_k(\sigma(E'))$ are equal. 
Thus it suffices to prove that if $E$ is absolutely reflexive and primitively generated, or cogenerated by indecomposables, then $\phi_k(\sigma(E))$ can be recovered from the functor $\OO E$. 

\subsection{An analogue of proposition \ref{prop-indecomp-summand}}
Since $E$ is primitively generated or cogenerated by indecomposables, it follows from corollary \ref{cor-KRS} that $E$ decomposes as a tensor product of indecomposables exponential functors of the form: 
$$S(A[d])\;,\text{ or}\; S_n(A[d])\;,\text{ or}\; \Lambda(A[d])\;,\text{ or}\; \Gamma(A[d])\;,\text{ or}\;\Gamma_n(A[d])\;,\;\qquad (*)$$
where $A[d]$ is a simple additive functor placed in degree $d$ and where $S_n(A[d])=\Sigma(A[d],F^{n-1})$ is the quotient algebra of $S(A[d])$ by the ideal generated by $^{(n)}A[dp^n]\subset S^{p^n}(A[d])$ and $\Gamma_n(A[d])=\Sigma(A[d],V^{n-1})$ is the subalgebra of $\Gamma(A[d])$ generated by $\bigoplus_{i\le p^n}\Gamma^{p^i}(A[d])$. 
Moreover since $E$ is absolutely reflexive, the factors of its decomposition must be absolutely reflexive hence $A$ must be absolutely simple. For these specific exponential functors we are able to prove the analogue of proposition \ref{prop-indecomp-summand}.
\begin{proposition}\label{prop-substitute}
Let $X$ denote one of the exponential functors of the list $(*)$ and assume that $A$ is absolutely simple. Then for all $k$, $\End_{\FF}(X^{kd})=\kk$.
\end{proposition}
\begin{proof}
We treat the case of $X=S_n(A[d])$, the other cases being similar. The functor $X^{dk}$ is a quotient of $A^{\otimes k}$. Thus $\End_{\FF}(X^k)$ is a subspace of $\Hom_{\FF}(A^{\otimes k},X^{dk})$. By \cite[Thm 1.7]{FFSS} the latter is isomorphic to the following direct sum, taken over all $k$-tuples $\beta$ such that $\sum \beta_i=kd$: 
$$\bigoplus_\beta \bigotimes_{i=1}^n\Hom_{\FF}(A, X^{\beta_i}) \;.\quad(**)$$
There is non nonzero morphism between an additive functor and a constant functor. Thus for degree reasons, the only term which contributes to this sum is $\beta=(d,\dots,d)$. Since $X^d=A$ and $A$ is absolutely simple, $(**)$ has dimension $1$. So $\dim_\kk \End_{\FF}(X^{kd})\le 1$. The other inequality is trivial.
\end{proof}

\subsection{Construction of the multisets $\alpha_k(F)$}
Assume that $E$ is  primitively generated (the proof for $E$ cogenerated by its indecomposables is similar) and absolutely reflexive. Fix $k\ge 0$. In order to prove theorem \ref{thm-uniqueness-weak}, we are going to prove that $\phi_k(\sigma(E))$ can be extracted from $\OO E$.

The functor $F=\OO E^k$ decomposes as a direct sum:
$$F= C \oplus \bigoplus_{i\in I}F_i \quad(\dag)$$
in which (i) $I$ is a finite set, (ii) $C$ is a functor which has no nonzero additive subfunctor, (iii) the $F_i$ are functors having a nonzero additive subfunctor, and (iv)  each $F_i$ is isomorphic to $X_i^k$ for some exponential functor $X_i$ in the following list (where $A_i$ is absolutely simple):
$$S(A_i[d_i])\;,\text{ or}\; S_{n_i}(A_i[d_i])\;,\text{ or}\; \Lambda(A_i[d_i])\;.\quad(**)$$
We call a decomposition $(\dag)$ such that the four conditions (i-iv) are satisfied a \emph{$k$-primitive decomposition of $F$}.
We don't know if $\FF$ satisfies the Krull-Remak-Schmidt property. However proposition \ref{prop-substitute} has the following consequence, which we can use as a substitute for proposition \ref{prop-KRS-P}.
\begin{proposition}\label{prop-substitute2}
If a functor $F$ has a \emph{$k$-primitive decomposition of $F$}, it is unique. Namely, if $F= C' \oplus \bigoplus_{i\in J}F'_i$ is another $k$-primitive decomposition, there is a bijection $\xi:I\to J$ and isomorphisms $F_i\simeq F'_{\xi(i)}$ and $C\simeq C'$.
\end{proposition}
\begin{proof}
The functors $X^k$ of the list $(**)$ have local endomorphism rings, hence they satisfy the cancellation property \cite[Prop 2]{Warfield}: if $X^k\oplus G\simeq X^k\oplus G'$ then $G\simeq G'$ (whatever $G$ and $G'$ are). Thus if we have two decompositions we can cancel all the $X^k$ in common, hence we can assume that for all pair $(i,j)\in I\times J$, we have $F_i\not\simeq F'_j$. We claim that in this case $I$ and $J$ are empty, thus $C=F\simeq F=C'$ and the result is proved. 

It remains to prove our claim. Assume that $I$ or $J$ is not empty. Among the $F_i$ and the $F'_j$, take one with maximal Eilenberg-Mac Lane degree, say for example $F_{i_0}$. Then $F_{i_0}$ does not embed in $C'$ (because it has a nontrivial additive subfunctor) and $\Hom_\FF(F_{i_0},F'_j)=0$ for all $j$ by lemma \ref{lm-mini-calcul} below, which contradict the fact that $F_{i_0}$ is a direct summand of $F$. 
\end{proof}

\begin{lemma}\label{lm-mini-calcul}
Let $X_1$ and $X_2$ be two functors in the list $(**)$. Assume that $\deg X_1^k\ge \deg X_2^k$, that $X_1^k$ and $X_2^k$ both have a nonzero additive subfunctor, and that $\Hom_\FF(X_1^k,X_2^k)\ne 0$. Then $X_1^k\simeq X_2^k$. 
\end{lemma}
\begin{proof}
First $X_1^k$ has no quotient of Eilenberg-Mac Lane degree less than $\deg_{\EML}X_1^k$. Thus $X_1^k$ and $X_2^k$ have the same EML-degree $\delta$. Then each $X_i^k$ must be isomorphic to some functor of the form $S^{\delta}(A_i)$ (because they have a nontrivial additive subfunctor). We compute $\Hom_\FF(S^\delta(A_1),S^\delta(A_2))$ in the same way as in proposition \ref{prop-substitute}: it is a subspace of $\Hom_\FF(A_1,A_2)^{\otimes \delta}$. Thus $\Hom_\FF(A_1,A_2)\ne 0$ hence $A_1\simeq A_2$ by simplicity. Hence $X_1^k\simeq X_2^k$.
\end{proof}

Now let $F$ be an arbitrary functor which admits a $k$-primitive decomposition $(\dag)$. For all $i\in I$, let $X_i$ such that $X_i^k\simeq F_i$. Then we let:
$$\pi_k(F_i)= \left(\bigoplus_{0\le \ell\le k}PX_i^\ell,\bigoplus_{0\le i\le k}QX_i^\ell \right)\;.$$
This is well defined because corollary \ref{cor-iso-fct} holds for the functors in the list $(*)$ (replace the use of the weights and of proposition \ref{prop-categ-decomp} in the proof of corollary \ref{cor-iso-fct} by the Eilenberg-Mac Lane degree of a functor and proposition \ref{prop-conn}). Now we let 
$$\alpha_k(F)=\{\pi_k(F_i)\;,\;i\in I\;\}.$$
One easily proves that for $F=\OO E$ we obtain $\alpha_k(F)=\phi_k(\sigma(E))$, which finishes the proof of theorem \ref{thm-uniqueness-weak}.

\section{(Counter)examples over arbitrary fields}\label{sec-counter}
Let $\kk$ be an arbitrary field of positive characteristic $p$, and $\kk^p$ be the subfield of $p$-th powers. In this short section, we describe a family of exponential functors $E_a$ indexed by a parameter $a\in\kk$. This family shows that several results obtained over perfect fields do not hold over imperfect fields.
\begin{enumerate}
\item If $a\not\in\kk^p$ then $E_a$ is indecomposable by lemma \ref{lm-CEX4}, while by lemmas \ref{lm-CEX2} and \ref{lm-additivite}, $(PQE_a)^2=A\oplus A$ is decomposable. Thus theorem \ref{thm-classif-indecomp} does not hold over an imperfect field.
\item If $a\not\in\kk^p$ then $E_a$ is not isomorphic to $E_0$, by lemmas \ref{lm-CEX1} and \ref{lm-CEX4} while the underlying graded functors are isomorphic by lemma \ref{lm-CEX2}. Thus, theorem \ref{thm-uniqueness-strict} does not hold over an imperfect field. Since the $E_a$ are primitively generated, even the weaker statement of theorem \ref{thm-uniqueness-weak} fails over an imperfect field.
\end{enumerate}

Fix a simple additive functor $A$ such that $\mathrm{End}_{\FF_\add}(A)\simeq\kk$ and  let $A[2]$ (resp. $^{(1)}A[2p]$) be a copy of $A$ (resp. $I^{(1)}\circ A$) placed in degree $2$ (resp. $2p$). We let $E=S(A[2]\oplus A[2])$ and we define $E_a$ as the cokernel of the morphism of exponential functors $S({^{(1)}}A[2p])\to E$ induced by 
$$g_a=(\Id,a\Id):{^{(1)}}A\to {^{(1)}}A\oplus {^{(1)}}A=PE^{2p}\subset E^{2p}\;.$$
\begin{lemma}\label{lm-CEX1}
If $a\in\kk^p$, the exponential functor $E_a$ is isomorphic to $T(A[2])\otimes S(A[2])$, where $T(A[2])$ is the quotient of $S(A[2])$ by the ideal generated by the $p$-th powers.
\end{lemma}
\begin{proof}
The morphism of functors $f: A\oplus A\to A\oplus A$ defined by $f(x,y)=(x, y+a^{1/p}x)$ induces an isomorphism $S(f): E\simeq E$, which induces in turn an isomorphism $T(A[2])\otimes S(A[2])\simeq E_a$. 
\end{proof}

\begin{lemma}\label{lm-CEX2}
Let $T(A)$ be the quotient of $S(A)$ by the ideal generated by the $p$-th powers.
For all $a\in\kk$ and all $d\ge 0$, there is an isomorphism in $\FF$:
$$(E_a)^{2d}\simeq \bigoplus_{k=0}^d T^{d-k}(A)\otimes S^{k}(A)\;.$$
\end{lemma}
\begin{proof}
The result holds for $a=1$ by lemma \ref{lm-CEX1}.
By definition, $E_a^{2d}$ is the cokernel in $\FF$ of the following composite morphism $\phi_a^{2d}$:
$${^{(1)}}A\otimes E^{2d-2p}\xrightarrow[]{g_a} E^{2p}\otimes E^{2d-2p}\xrightarrow[]{\mathrm{mult}} E^{2d}\;.$$
Thus it suffices to find automorphisms $\psi^{2d}_{a}$ of $E^{2d}$ for all $d\ge 0$, such that 
$$\phi_a^{2d}\circ ({^{(1)}}A\otimes \psi^{2d-2}_a)= \psi_a^{2d} \circ \phi_1^{2d}\;.$$
Now we have a decomposition
$$E^{2d}= \bigoplus_{0\le i\le  p-1}\;\bigoplus_{0\le k\le \lfloor d/p \rfloor} S^{d-i-kp}(A)\otimes S^{i+kp}(A)$$
and we define a suitable $\psi^{2d}_a$ as the automorphism which is the multiplication by $a^k$ on each summand $S^{d-i-kp}(A)\otimes S^{i+kp}(A)$.
\end{proof}

\begin{lemma}\label{lm-CEX4} If $a\not\in\kk^p$ the exponential functor $E_a$ is indecomposable.
\end{lemma}
\begin{proof}
We first describe the abelian group $H(a,b):=\mathrm{Hom}_{\FF-\Exp_c}(E_a,E_b)$.
A morphism $f:E\to E$ is uniquely determined by its restriction $f^2:E^2\to E^2$. Since $E^2=E_a^2=E_b^2$ any morphism $\overline{f}:E_a\to E_b$ lifts uniquely to a morphism $f:E\to E$ such that $Pf^{2p}(\mathrm{Im}\,g_a)\subset \mathrm{Im}\,g_b$. Since $E^2=A\oplus A$ and $\mathrm{End}_\FF(A)=\kk$, $f^2$ can be described by a $2\times 2$-matrix $M$ with coefficients in $\kk$. Similarly $\mathrm{End}_\FF({^{(1)}}A)=\kk$ and the action of $Pf^{2p}$ on $PE^{2p}= {^{(1)}}A\oplus {^{(1)}}A$ is described by the matrix $M^{[1]}$, where `$^{[1]}$' means that the coefficients of $M$ are raised to the power $p$. Thus an element $\overline{f}\in H(a,b)$ is described by a matrix: 
$$M= \begin{bmatrix}
x & y \\
z & t
\end{bmatrix} \text{ such that there is $\lambda\in\kk$ such that} 
\begin{cases}
x^p+y^pa &=  \lambda,\\
z^p+t^p a &= \lambda b.
\end{cases}
$$
By construction, the convolution product in $H(a,b)$ corresponds to the addition of matrices, while the composition $H(b,c)\times H(a,b)\to H(a,c)$ corresponds to the multiplication of matrices.

Now let $f\in H(a,a)$ such that $f\circ f=f$. The corresponding matrix satisfies $M^2=M$, hence $(M^{[1]})^2=M^{[1]}$. Thus the only eigenvalues of $M^{[1]}$ are $0$ and $1$, so that $\lambda\in\{0,1\}$. Assume that $a\not\in\kk^p$. If $\lambda=0$ the conditions imply that $M^{[1]}=0$, hence $M=0$, and if $\lambda=1$ the conditions imply that $M=t\Id$. So $f$ is either $0$ or an isomorphism.
\end{proof}

\appendix

\section{Frobenius twists}\label{app-VF}

This appendix presents some well-known facts regarding Frobenius twists, and their application to bicommutative Hopf algebras. Here $\kk$ is a field of positive characteristic $p$, and $\mathcal{V}$ is the category of $\kk$-vector spaces.

\subsection{The Frobenius twist of a vector space}\label{subsec-app-VF-norm}
The $r$-th Frobenius twist functor is the functor 
$$I^{(r)}:\mathcal{V}\to \mathcal{V}$$
obtained by scalar extension along the Frobenius morphism $F^r:\kk\to \kk$, defined by $F^r(x)=x^{p^r}$. If $\kk$ is a perfect field, $F^r$ exists for all $r\in\Z$, otherwise we impose that $r$ is nonnegative. We have $I^{(r)}(V)={}_0\kk_r\otimes_\kk V$, where ${}_0\kk_r$ is the abelian group $\kk$, with $\kk$-$\kk$-bimodule structure $\lambda\cdot x\cdot \mu= \lambda x F^r(\mu)$. A tensor of the form $1\otimes v$ is usually denoted by $^{(r)}v$. Note that $I^{(0)}$ is the identity functor of $\mathcal{V}$, which we also denote by $I$.
\begin{notation}
For all vector spaces $V$, the notation $I^{(r)}(V)$ is most often simplified into $^{(r)}V$.
\end{notation}

The Frobenius twist functors satisfy the usual properties of scalar extension recalled in the next fact.
\begin{fact}\label{fact-prop-extscal} see \cite[II \S 5]{BourbakiAlgI}.
\begin{enumerate}
\item\label{fi-univ} {\bf Universal property.} Let $\pi:V\to {}^{(r)}V$ be the map $\pi(v)={}^{(r)}v$. It is $p^r$-linear, i.e. it is additive and $\pi(\lambda v)=\lambda^{p^r}\pi(v)$. 
For all vector spaces $V$, $W$ and all $p^r$-linear maps $f:V\to W$, there is a unique $\kk$-linear map $\overline{f}: {}^{(r)}V\to W$ such that 
$$\overline{f}\circ \pi= f\;.$$
\item\label{fi-colim} {\bf Colimits.} The functor $I^{(r)}$ commutes with colimits and preserves dimension. In particular it is exact and faithful.
\item\label{eqcat} {\bf Transitivity.} There is a natural isomorphism ${I}^{(r)}\circ {I}^{(s)}\simeq {I}^{(r+s)}$. Thus, if $\kk$ is perfect, $I^{(r)}$ is an equivalence of categories.
\item {\bf Tensor products.} The $r$-th Frobenius twist is a strong symmetric monoidal endofunctor of $(\mathcal{V},\otimes,\kk)$.
\item {\bf Duality.} There is a natural map $^{(r)}\Hom_\kk(V,\kk)\to \Hom_\kk(^{(r)}V,\kk)$, which is an isomorphism if $\dim_\kk V<\infty$.
\end{enumerate}
\end{fact}

\begin{remark}
In fact \ref{fact-prop-extscal}\eqref{fi-colim}, exactness comes from the additivity of  $I^{(r)}$ and the fact that exact sequences of $\mathcal{V}$ are split exact. Faithfulness comes from the isomorphism $\mathrm{Im} {}^{(r)}f\simeq {}^{(r)}(\mathrm{Im} f)$ (by exactness) the latter being zero if and only if $\mathrm{Im} f$ is zero (by dimension preserving). Note that if $r>0$ and $\kk$ is not perfect, then $I^{(r)}$ is not full, since $\lambda\Id :{}^{(r)}V\to {}^{(r)}V$ is equal to some $^{(r)}f$ if and only if $\lambda$ is in the image of $F^r$.
\end{remark}

Now we assume that $r\ge 0$. We are going to give a combinatorial description of $^{(r)}V$. The symmetric group $\Si_{p^r}$ acts on $V^{\otimes p^r}$ by permuting the factors. Let $\Gamma^{p^r}V$ be the subspace of invariants and $S^{p^r}V$ the quotient space of coinvariants. There is a composite map (natural with respect to $V$):
$$\Psi_r\;:\; \Gamma^{p^r}V\hookrightarrow V^{\otimes p^{r}}\twoheadrightarrow S^{p^r}V\;.$$
Fix a basis of $V$. If $v_1,\dots,v_{p^r}$ are $p^r$ basis vectors of $V$, let $\mathrm{sym}(v_1,\dots,v_{p^r})$ be the sum of all the elements of the set $\{\, v_{\sigma(1)}\otimes\dots\otimes v_{\sigma(p^r)}\;\left|\;\sigma\in\Si_{p^r}\,\right.\}$. For example, if all the $v_k$ are equal to $v$, then this set has only one element and $\mathrm{sym}(v,\dots,v)=v^{\otimes p^r}$. The elements of the form $\mathrm{sym}(v_1,\dots,v_{p^r})$ generate $\Gamma^{p^r}V$ as a $\kk$-vector space. The class in $S^{p^r}V$ of a tensor $v\otimes\dots\otimes v$ is denoted by $v^{p^r}$ as usual.
Then one has
$$\Psi_r(\mathrm{sym}(v_1,\dots,v_{p^r}))=
\begin{cases}
v^{p^r} & \text{ if all the $v_k$ are equal to $v$,}\\
0 & \text{ otherwise.}
\end{cases}
$$
Thus $\mathrm{Im}\,\Psi_r$ is the subspace of $S^{p^r}V$ spanned by the vectors $v^{p^r}$, $v\in V$.
\begin{fact}\label{fact-GSF}
The vector space $^{(r)}V$ is naturally isomorphic to $\mathrm{Im}\,\Psi_r$. In particular, there is a natural surjective morphism $\Gamma^{p^r}V\to {}^{(r)}V$ and a natural injective morphism $^{(r)}V\to S^{p^r}(V)$.
\end{fact}
\begin{proof}
The surjective $p^r$-linear map $V\to \mathrm{Im}\,\Psi_r$, $v\mapsto v^{p^r}$ induces a surjective $\kk$-linear morphism $^{(r)}V\to \mathrm{Im}\,\Psi_r$. If $\dim_\kk V<\infty$, the latter is an isomorphism for dimension reasons. The general case follows by taking colimits (any vector space is the colimit of its finite dimensional subspaces). 
\end{proof}

\subsection{Frobenius twists as strict polynomial functors}\label{subsec-app-VF-strict} Let $r\ge 0$. Up to a scalar, there is a unique morphism of strict polynomial functors $\Gamma^{p^r}\to S^{p^r}$ (apply e.g. \cite[Thm 2.10]{FS}). As an object of $\PP_{p^r,\kk}$, the $r$-th Frobenius functor $I^{(r)}$ is \emph{defined} as the image of this morphism. It follows from fact \ref{fact-GSF} that the underlying ordinary functor coincides with the usual $r$-th Frobenius functor of section \ref{subsec-app-VF-norm}. Given an object $F$ of $\PP_{d,\kk}$ we often denote by $F^{(r)}$ the composition $F\circ I^{(r)}$. Precomposition by $I^{(r)}$ defines a functor:
$$ -^{(r)}:\PP_{d,\kk}\to \PP_{dp^r,\kk}\;.$$

\begin{fact}\label{fact-thickF}
The functor $-^{(r)}$ is fully faithful. Its essential image is a thick (i.e. stable under subobjects, quotients and extensions) subcategory of $ \PP_{dp^r,\kk}$.
\end{fact}
\begin{proof}
Full faithfulness is proved e.g. in \cite[Appendix A]{BMT}. If $L$ is a simple object, then $L^{(r)}$ is also simple, see e.g. \cite[Prop B.2]{TouzeStein}. Thus the essential image is supported by all the functors whose finite subfunctors have composition factors of the form $L^{(r)}$ for simple objects $L$ of $\PP_{d,\kk}$. This is clearly a thick subcategory.
\end{proof}

\subsection{Frobenius and Verschiebung of Hopf algebras}\label{app-FV-Hopfalg}
Now we introduce Frobenius and Verschiebung maps for bicommutative Hopf algebras. A reference is \cite[II \S 7 1.1 and IV \S 3 4.4]{DG}. 
Given a graded vector space $V=\bigoplus_{i\in \mathbb{N}} V^i$, its Frobenius twist $^{(r)}V$ is the graded vector space with:
$$({^{(r)}}V)^k=\begin{cases} 
{}^{(r)}V^i & \text{ if $k=pi$,}\\
0 & \text{ if $p$ does not divide $k$.}
\end{cases}$$

Let $(A,\mu,\eta)$ be a graded $\kk$-algebra which is commutative in the ungraded sense (thus $A$ is concentrated in even degrees if $p$ is odd). Since the $r$-th Frobenius twist is strong monoidal, the following structure morphisms endow the graded vector space ${^{(r)}}A$ with the structure of a commutative graded $\kk$-algebra:
$$\kk\simeq{^{(r)}}\kk\xrightarrow[]{^{(r)}\eta} {^{(r)}}A\;,\qquad 
{^{(r)}}A\otimes {^{(r)}}A\simeq {^{(r)}}(A\otimes A)\xrightarrow[]{{^{(r)}}\mu}{^{(r)}}A\;.$$
The Frobenius map $\mathbf{F}:{^{(1)}}A \to A$ is the morphism of graded $\kk$-algebras (natural with respect to $A$) defined by $\mathbf{F}(^{(1)}a)=a^p$. One has the following obvious alternative definition.
\begin{fact}
The map $\mathbf{F}$ equals the composition $^{(1)}A\hookrightarrow S^{p}A\xrightarrow[]{\mu} A\;.$
\end{fact}

Similarly, if $(C,\Delta,\epsilon)$ is a  graded coalgebra over $\kk$, which is cocommutative in the ungraded sense, then the following coproduct and augmentation induce a coalgebra structure on the graded vector space $^{(r)}C$:
$$
{^{(r)}}C\xrightarrow[]{^{(r)}\epsilon} {^{(r)}}\kk\simeq \kk\;,\qquad
{^{(r)}}C\xrightarrow[]{{^{(r)}}\Delta}{^{(r)}}(C\otimes C)\simeq{^{(r)}}C\otimes {^{(r)}}C\;.
$$
The Verschiebung map $\mathbf{V}$ is defined as the composition:
$C\xrightarrow[]{}\Gamma^{p}C\twoheadrightarrow {^{(1)}}C$.
\begin{fact}
Verschiebung maps are morphisms of graded coalgebras.
\end{fact}
\begin{proof}
It is possible to check this by using an explicit formula for the effect of $\mathbf{V}$ on elements of $C$. Alternatively, one may use the fact that any graded coalgebra is a union of subcoalgebras which are finite dimensional in each degree to reduce the proof to such coalgebras. Then taking the restricted dual (i.e. the $\kk$-linear dual in each degree) turn coalgebras into algebras and Verschiebung maps into Frobenius maps, so that one can deduce the result from the fact that $\mathbf{F}$ is a morphism of algebras. 
\end{proof}

If $H$ is a graded Hopf algebra over $\kk$ which is bicommutative in the ungraded sense, then one can endow the graded vector space ${^{(r)}}H$ with the structure of a graded Hopf algebra by using the product, unit, augmentation and coproduct defined above, together with the antipode $^{(r)}\chi:{^{(r)}}H\to {^{(r)}}H$. 
\begin{fact}
The Frobenius map $\mathbf{F}:{^{(1)}}H\to H$ and the Verschiebung map $\mathbf{V}:H\to {^{(1)}}H$ are morphisms of graded Hopf algebras.
\end{fact}
\begin{proof}
That $\mathbf{F}$ is a morphism of coalgebras and commutes with antipodes follows from the naturality with respect to the algebra $A$, together with the fact that the Frobenius map of $A\otimes A$ equals the composition $^{(1)}(A\otimes A)\simeq (^{(1)}A)^{\otimes 2}\xrightarrow[]{\mathbf{F}^{\otimes 2}}A\otimes A$. A similar argument shows that $\mathbf{V}$ is a morphism of algebras commuting with antipodes.
\end{proof}

\section{Graded Dieudonn\'e modules}\label{App-Dieu}
In this appendix, we briefly recall some classical results from \cite{Schoeller} regarding the category of graded bicommutative Hopf algebras over a field $\kk$ of positive characteristic $p$, and some immediate consequences of these results.

Let $\HH^+$ denote the category of graded bicommutative Hopf algebras which are connected and concentrated in even degrees. For all $i\ge 1$, let $\HH_i$ be the full subcategory supported by the Hopf algebras whose primitives (or equivalently whose indecomposables) are concentrated in degrees $2ip^k$, $k\ge 0$. By regrading, the $\HH_i$ are all equivalent to one another.
Let $\mathbb{N}_p$ be the set of integers prime to $p$. 
\begin{fact}\label{thm-schoeller1}\cite[Section 2, Th\'eor\`eme]{Schoeller}
There is an equivalence of categories 
$$\HH^+\simeq \prod_{i\in\mathbb{N}_p}\HH_i\;.$$
\end{fact}
Let $W(\kk)$ be the ring of the perfect field $\kk$. Let $\mathcal{D}_\kk$ be the category of graded Dieudonn\'e modules defined as follows. An object of $\mathcal{D}_\kk$ is a nonnegatively graded $W(\kk)$-module $M=\bigoplus_{i\ge 0} M^i$ such that $p^{i+1}M_i=0$ for all $i$, equipped with morphisms\footnote{Although the degrees of $M$ are denoted as exponents, we use indices for the operators $F_i$ and $V_i$, in order to prevent confusion with the closely related notations $F^i=F\circ\cdots\circ F$  and $V^i=V\circ\cdots \circ V$ where $F$ and $V$ are Frobenius and Verschiebung maps for Hopf algebras.} of $W(\kk)$-modules $F_i:M^i\to M^{i+1}$ and $V_i:M^{i+1}\to M^i$, satisfying $F_iV_i=p=V_iF_i$ for all $i$. A morphism in $\mathcal{D}_\kk$ is a (degree preserving) morphism of graded $W(\kk)$-modules commuting with the operators $F_i$ and $V_i$. 

\begin{fact}\label{thm-schoeller2}\cite[Th\'eor\`eme in section 5.2, and section 5.4]{Schoeller}
Let $\kk$ be a perfect field. There is an equivalence of categories: 
$$M:\HH_1\simeq \mathcal{D}_\kk\;.$$
Moreover, if $H$ is primitively generated, then $M(H)$ is given by 
\begin{enumerate}
\item $M(H)^i={^{(-i)}PH^{2p^i}}$,
\item the operators $F_i:M(H)^i\to M(H)^{i+1}$ coincide (up to a twist $^{(-i-1)}$) with the Frobenius map $^{(1)}PH^{2p^i}\to PH^{2p^{i+1}}$, $^{(1)}x\mapsto x^p$,
\item the operators $V_i$ are zero.
\end{enumerate}
\end{fact}

\begin{notation}
We let $M:\HH_n\simeq \mathcal{D}_\kk$ be the equivalence of categories defined by dividing all the degrees of a Hopf algebra by $n$ and then applying the equivalence of categories given in fact \ref{thm-schoeller2}.
\end{notation}

\begin{fact}\label{fact-cor-P}
Let $M(H)=(M,F_*,V_*)$. There is a natural isomorphism 
$$PH^{2np^k}\simeq{^{(k)}}(\ker V_{k-1})\;.$$
\end{fact}
\begin{proof}
We have a sequence of isomorphisms of abelian groups, natural with respect to $V$ and $H$:
\begin{align*}\Hom_\kk(V,PH^{2p^k}) &\simeq \Hom_{\HH_1}(S(V[2p^k]),H)\simeq \Hom_{\mathcal{D}}(M(S(V[2p^k])),M(H))\\
&\simeq \Hom_\kk({^{(-k)}V},\mathrm{Ker}\,V_{k-1})\simeq \Hom_\kk(V,{^{(k)}}\mathrm{Ker}\,V_{k-1})
\end{align*}
Hence the result follows from the Yoneda lemma.
\end{proof}

\begin{fact}\label{fact-cor-Q}
We have an isomorphism, natural with respect to $V$, between $M(\Gamma(^{(k)}V[2np^{k}]))$ and the Dieudonn\'e module which equals $V$ in each degree $i\ge k$, and such that $V_i$ is the identity morphism for $i\ge k-1$, and all the other operators are zero.
As a consequence, if $M(H)=(M,F_*,V_*)$, we have a natural isomorphism:
$$QH^{2np^k}\simeq{^{(k)}}(\mathrm{Coker} F_{k-1})\;.$$
\end{fact}
\begin{proof}
The equivalence $M$ sends the $p$-th truncated polynomial algebra $T(^{(k)}V[2np^k])$ on a vector space $^{(k)}V$  to the Dieudonn\'e module which is equal to $V$ concentrated in degree $k$ (with trivial operators $F_*$ and $V_*$). 
The divided power algebra $H(V)=\Gamma(^{(k)}V[2np^k]$ (i) has a composition series whose factors are the $T({^{(\ell)}}V[2np^\ell])$, $\ell\ge k$, each of these appearing once, and (ii) has primitives equal to $^{(k)}V[2np^k]$. Thus by (i) $M(H(V))^i=V$ in each degree $i\ge k$ and zero otherwise, and by (ii) we know that all the $V_i$ must be injective for $i\ge k-1$. Since the only morphisms $V\to V$, natural with respect to $V$ are homotheties, it follows that these $V_i$ are nonzero scalar multiples of the identity. Since $0=p=V_iF_i$, the $F_i$ are zero.
Finally, up to isomorphism, one can arrange that the $V_i$, $i\ge k-1$, are equal to the identity. The identification of the indecomposables is similar to the proof of fact \ref{fact-cor-P}.
\end{proof}

Given an abelian category $\A$ and a ring $R$,  we let $_R\A$ denote the category of left $R$-modules in $\A$. Its objects are pairs $(A,\phi_A)$ where $A$ is an object of $\A$ and $\phi_A:R\to \End_\A(A)$ is a ring morphism. Its morphisms are the morphisms $f:A\to B$ such that  $f\circ \phi_A(r)=\phi_B(r)\circ f$ for all $r\in R$. 
\begin{notation}\label{notation-AppA}
We write $\HH_n'$ for $_{\Fp}\HH_n$ and $\mathcal{D}'_\kk$ for $_{\Fp}\mathcal{D}_\kk$.
\end{notation}

As any equivalence of abelian categories induces an equivalence between the corresponding categories of $R$-modules, we obtain:
\begin{fact}\label{thm-schoeller-Fp}
Let $\kk$ be a perfect field. There is an equivalence of categories: 
$$M:\HH_n'\simeq \mathcal{D}'_\kk\;.$$
\end{fact}

When $R=\Fp$, the category $\Fp-\A$ identifies with the full subcategory of $\A$ supported by the objects $A$ such that $p\Id_A=0$. Since $\kk$ is perfect, $W(\kk)/pW(\kk)$ is isomorphic to $\kk$. Thus we obtain:
\begin{fact}\label{fact-DieudoDef}
The category $\mathcal{D}'_\kk$ is the full subcategory of $\mathcal{D}_\kk$ supported by the graded $\kk$-modules $M=\bigoplus_{i\ge 0} M^i$ equipped with $\kk$-linear operators $F_i:M^i\leftrightarrows M^{i+1}:V_i$ satisfying the relations $F_iV_i=0=V_iF_i$. 
\end{fact}

An object $M$ in $\mathcal{D}'_\kk$ can be depicted as a diagram of $\kk$-modules, each $M^i$ being placed in degree $i$, and operators $F_i$ and $V_i$
$$\xymatrix{
M^0 \ar@/^/[r]^-{F_0} & \ar@/^/[l]^-{V_0}M^1 \ar@/^/[r]^-{F_1}&
\ar@/^/[l]^-{V_1}M^2 \ar@/^/[r]^-{F_2}&\ar@/^/[l]^-{V_2}\cdots \;.
}$$ 
When drawing such diagrams, we usually only mention the operators $F_i$ and $V_i$ which are nonzero.
Let $\Gamma_m(\kk[2np^k])$ be the kernel of the unique (up to a nonzero scalar multiple) morphism $\Gamma(\kk[2np^k])\to \Gamma(\kk[2np^{km}])$. Since $M$ is exact, the Dieudonn\'e module of $\Gamma_m(\kk[2np^k])$ is depicted by:
\begin{align*}\underbrace{\kk}_{\text{deg } 0}\xleftarrow[]{\Id} \cdots \xleftarrow[]{\Id} \underbrace{\kk}_{\text{deg } m-1}\;.\end{align*}
The Dieudonn\'e module of $S(\kk[2np^m])$ is known by fact \ref{thm-schoeller2}. Up to isomorphism, there is a unique graded Dieudonn\'e module $M(m)$ which has $M(\Gamma_n(\kk[2np^k]))$ as a submodule, $M(S(\kk[2np^m]))$ as a quotient, and which is a nontrivial extension of these modules. Namely, $M(m)$ is depicted by:
\begin{align*}\underbrace{\kk}_{\text{deg } 0}\xleftarrow[]{\Id} \cdots \xleftarrow[]{\Id} \underbrace{\kk}_{\text{deg } m}\xrightarrow[]{\Id}\cdots \xrightarrow[]{\Id}\kk \xrightarrow[]{\Id}\cdots\;.\end{align*}
On easily sees that the modules $M(m)$ form a projective generator of $\mathcal{D}'_\kk$. Coming back to $\HH_n'$ we obtain the following statement.
\begin{fact}\label{fact-Gen-dieu}
Let $\kk$ be a perfect field. For all $m\ge 0$ there is a unique (up to isomorphism) Hopf algebra $H_m$ in $\HH_n'$ fitting into a nonsplit short exact sequence
$$\kk\to \Gamma_m(\kk[2n])\to H_m \to S(\kk[2np^m])\to \kk\;.$$
Moreover the $H_m$, $m\ge 0$, form a generator of $\HH_n'$.
\end{fact}

\section{Additive, polynomial and analytic functors}\label{app-Fct}

In this appendix, we briefly recall some well-known facts on cross-effects and polynomial functors, introduced in \cite{EML2}. 
\begin{fact}\label{fact-idcomp}
Let $\A$ be an additive category. The following assertions are equivalent.
\begin{enumerate}
\item[(i)] All idempotents in $\A$ have a kernel.
\item[(ii)] All idempotents in $\A$ have a cokernel.
\item[(iii)] All idempotents $e:A\to A$ in $\A$ decompose as $e=s\circ r$ for some morphisms $r:A\to B$ and $s:B\to A$ satisfying $r\circ s=\Id_B$. 
\end{enumerate}
\end{fact}
In the sequel, $\A$ denotes an additive category with split idempotents, that is which satisfies the equivalent conditions listed in fact \ref{fact-idcomp}. We also let $\V$ be a (small) additive category, and we denote by $\Fct(\V,\A)$ be the category of all functors from $\V$ to $\A$.

Since $\A$ has split idempotents, any functor $F:\V\to \A$ canonically decomposes as a direct sum $F(V)=F(0)\oplus \overline{F}(V)$, where $\overline{F}:\V\to \A$ is \emph{reduced}, i.e. satisfies $\overline{F}(0)=0$. Letting $\Fct_{\mathrm{red}}(\V,\A)$ be the full subcategory of reduced functors, one has a canonical decomposition
$$ \Fct(\V,\A)\simeq \A\times\Fct_{\mathrm{red}}(\V,\A)\;.$$

For $n\ge 2$, the $n$-th cross effect of $F$ is the functor with $n$ variables denoted by $\Cr_n F$ and defined as the kernel of the canonical morphism:
$$\overline{F}(V_1\oplus\dots\oplus V_n)\xrightarrow[]{(\overline{F}(e_1),\dots,\overline{F}(e_n))} \bigoplus_{i=1}^n\overline{F}(V_1\oplus\dots\oplus V_n)\;, $$
where $e_i$ denotes the composite of the canonical morphisms
$$V_1\oplus\dots\oplus V_n\to  V_1\oplus\dots\oplus\widehat{V_i}\oplus\dots\oplus V_n \to V_1\oplus\dots\oplus V_n\;.$$
The kernel exists and is a direct summand of $\overline{F}(V_1\oplus\dots\oplus V_n)$ because the $e_i$ are commuting idempotents. Moreover, this definition can be proved to be equivalent to the original definition of Eilenberg and Mac Lane \cite[p. 77]{EML2} e.g. by using the characterization given in \cite[Thm 9.6]{EML2}. 
The functor $\Cr_n F$ is $n$-reduced, i.e. it is equal to zero whenever one of its arguments is zero. Thus cross effects define a functor:
$$\Fct(\V,\A)\to \Fct_{n-\mathrm{red}}(\V^{\times n},\A)\;.$$
A functor $\mathrm{D}_n$ going the other way is defined by $(\mathrm{D}_nG)(V)=G(V,\dots,V)$.
\begin{fact}\label{fact-adj}
The pairs $(\Cr_n,D_n)$ and $(D_n,\Cr_n)$ are adjoint pairs.
\end{fact}
A functor $F:\V\to \A$ is called \emph{polynomial of degree less than $n$} (or sometimes polynomial of Eilenberg-Mac Lane degree less than $n$ if we want to avoid confusion with homological degrees) whenever $\Cr_n F=0$. It is polynomial of degree $n-1$ if $\Cr_n F=0$ and $\Cr_{n-1} F\ne 0$. Observe that $\Cr_n F(V_1,\dots,V_n)$ is a direct summand of $D_n\Cr_n F(V)$ for $V=V_1\oplus\dots\oplus V_n$. Thus we obtain a criterion which is sometimes useful in computations.
\begin{fact}\label{def-pol}
$F$ is polynomial of degree less than $n$ if and only if $D_n\Cr_nF=0$.
\end{fact}
When $\A=\kk-\Mod^*$, the category $\Fct(\V,\A)$ is abelian with enough projectives and injectives, and the derived form of fact \ref{fact-adj} is often formulated as follows, see e.g \cite[Thm A.1]{PiraBetley}. 
\begin{fact}[Pirashvili's vanishing lemma]\label{lm-cancel}
Let $F$ be polynomial of degree less than $n$ and let $F_1,\dots,F_n$ be reduced functors. Then 
$$\Ext^*_{\Fct(\V,\kk-\mathrm{Mod}^*)}(F,F_1\otimes\dots\otimes F_n)=0=\Ext^*_{\Fct(\V,\kk-\mathrm{Mod}^*)}(F_1\otimes\dots\otimes F_n,F)\;.$$
\end{fact}

We denote by $\Fct^{< n}(\V,\A)$ the full subcategory of $\Fct(\V,\A)$ whose objects are the polynomial functors of degree less than $n$. We have the following formal consequence of fact \ref{fact-adj}.
\begin{fact}\label{fact-stab}
The cross effect functor $\Cr_n$ commutes with limits and colimits. In particular, the subcategory $\Fct^{< n}(\V,\A)$ of $\Fct(\V,\A)$ is stable by direct summands, and more generally by limits and colimits (when they exist).
\end{fact}
Assume that $\A$ has all cokernels. Then for all $F:\V\to \A$ we let
$q_n F$ be the cokernel of the counit of adjunction $D_n(\Cr_n F)\to F$. This counit of adjunction can be explicitly described as the composition (where $\sigma_V$ is the folding map):
$$\Cr_nF(V,\dots,V)\hookrightarrow F(V^{\oplus n})\xrightarrow[]{F(\sigma_V)} F(V)\;.$$
\begin{fact}\label{fact-quot}
Assume that $\A$ has all cokernels. The functor $q_nF$ is polynomial of degree less than $n$, and  
$$q_n:\Fct(\V,\A)\to \Fct^{< n}(\V,\A)$$
is left adjoint to the embedding $\Fct^{< n}(\V,\A)\hookrightarrow \Fct(\V,\A)$. (Thus $q_n F$ is the largest quotient of degree less than $n$ of $F$.)
\end{fact}

When the counits $D_n\Cr_n F\to F$ have images (e.g. if $\A$ is abelian), these yield a descending filtration of $F$, called the \emph{copolynomial filtration of $F$} \cite{Djament}.

Dually, if $\A$ has all kernels, we let $p_nF$ be the kernel of the unit of adjunction $F(V)\to D_n\Cr_n(V)$. This unit of adjunction  
equals the composition:
$$ F(V)\xrightarrow[]{F(\delta_V)} F(V^{\oplus n})\twoheadrightarrow \Cr_nF(V,\dots,V)\;.$$ 
\begin{fact}\label{fact-ker}
Assume that $\A$ has all kernels. The functor $p_n$ is polynomial of degree less than $n$, and  
$$p_n:\Fct(\V,\A)\to \Fct^{< n}(\V,\A)$$
is right adjoint to the embedding $\Fct^{< n}(\V,\A)\hookrightarrow \Fct(\V,\A)$. (Thus $p_n F$ is the largest subfunctor of degree less than $n$ of $F$.)
\end{fact}

The subfunctors $p_nF$ yield an ascending filtration of $F$, called the \emph{polynomial filtration of $F$}. The functor $F$ is called \emph{analytic} if its polynomial filtration is exhaustive. 
The following improvement of fact \ref{fact-stab} is a formal consequence of the existence of $p_n$ or $q_n$.
\begin{fact}\label{fact-create}
Assume that $\A$ has all cokernels (resp. all kernels). Then the embedding $\Fct^{< n}(\V,\A)\hookrightarrow \Fct(\V,\A)$ creates colimits (resp. limits).
\end{fact}

A functor $F$ has degree less than two if and only if in the canonical decomposition
$F(V)=F(0)\oplus \overline{F}(V)$, the functor $\overline{F}$ is additive. Thus, all facts recalled  above for the category $\Fct^{<2}(\V,\A)$ also apply to the full subcategory  $\Fct_{\add}(\V,\A)$ of additive functors.

The next result may be thought of as the core of the Eilenberg-Watts classification of additive functors \cite{Eilenberg, Watts}. For this result, we take the category $\mathrm{P}_R$ of finitely generated projective $R$-modules as source category of our additive functors. We denote by ${_R}\A$ the category of $R$-modules in $\A$. For example the category of $R$-modules in $\Mod$ is isomorphic to the category $_R\Mod$ of $(R,\kk)$-bimodules.
\begin{fact}\label{fact-EW}
Evaluation on $R$ yields an equivalence of categories 
$$\Fct_\add(\mathrm{P}_R,\A)\to {_R}\A\;.$$
\end{fact}

To be more specific, the inverse of the equivalence is constructed as follows. The category ${_R}\A$ can be equivalently thought of as the category of functors from the category $*_R$ with one object and the ring $R$ as endomorphisms, to $\A$. Since $\A$ is additive with split idempotents, any functor $F:*_R\to \A$ extends to an additive functor $\widetilde{F}:(*_R)_\oplus^\natural\to \A$, where $(*_R)_\oplus$ is the additive hull of $*_R$ and $^\natural$ refers to its idempotent completion. But $(*_R)_\oplus^\natural$ is equivalent to $\mathrm{P}_R$ (see e.g. \cite[Lm 3.12]{IvoGoncalo}).

\begin{remark}
An analogous classification for polynomial functors of higher degree was proved by Pirashvili 
\cite{PiraPol}, but we shall not use this here.
\end{remark}

Now we take $\A=\Mod$. Thus we can define the tensor product of two functors $F,G:\V\to \Mod$ by $(F\otimes G)(V)=F(V)\otimes G(V)$. Using the definition of cross effects, one can prove
\begin{fact}\label{fact-tens}
If $\deg F\le k$ and $\deg G\le \ell$ then $\deg (F\otimes G)\le k+\ell$.
\end{fact}
The inequality in fact \ref{fact-tens} may very well be strict. For example, if $\kk=\Z$, $\V=\Proj_\Z$ then $A(-)=-\otimes \mathbb{Q}/\Z$ has degree one, while $A^{\otimes 2}=0$. 
\begin{fact}\label{fact-tens-2}
If $\kk$ is a commutative domain and $F$ and $G$ are polynomial functors such that $F(V)$ and $G(V)$ are torsion-free for all $V$, then $\deg (F\otimes G)=\deg F+\deg G$.  
\end{fact} 
\begin{proof}
Let $d=\deg F$ and $e=\deg G$. The functor of $d+e$ variables:
$$C(V_1,\dots,V_{d+e})=\Cr_dF(V_1,\dots,V_d)\otimes \Cr_eG(V_{d+1},\dots,V_{d+e})$$
is $(d+e)$-reduced, and it is a direct summand of $(F\otimes G)(V_1\oplus\dots\oplus V_{d+e})$. Hence it is a subfunctor of $\Cr_{d+e}(F\otimes G)$. 
By our hypotheses $\Cr_dF$ and $\Cr_eG$ are both nonzero, and they have values in torsion free modules. Since $\kk$ is a domain, the tensor product of two nonzero torsion-free modules is nonzero (tensor with the fraction field to see it). Thus $C$ is nonzero, hence $\Cr_{d+e}(F\otimes G)$ is nonzero, hence $\deg (F\otimes G)\ge \deg F+\deg G$. The other inequality is provided by fact \ref{fact-tens}. 
\end{proof}

The notion of degree in the sense of Eilenberg and Mac Lane can also be developed for strict polynomial functors. Let $\kk$ be a commutative ring and let $\PP_{d,\kk}$ be the category of homogeneous strict polynomial functors of weight $d$ over $\kk$. There is a forgetful functor 
$$\U: \PP_{d,\kk}\to \Fct(\Proj_\kk,\Mod)\;.$$
The degree of $F\in \PP_{d,\kk}$ is defined as that of $\U F$. Since the projective generator $\Gamma^{d,n}$ of $\PP_{d,\kk}$ has degree $d$, it follows from fact \ref{fact-stab} that the degree of an object of $\PP_{d,\kk}$ is always less or equal to $d$. Note that the inequality can be strict, e.g. if $\kk$ is a field of characteristic $p$ then $I^{(1)}\in \PP_{p,\kk}$ is additive, hence it has degree $1$. 

In fact it is possible to define cross effects directly on the level of strict polynomial functors, in such a way that cross effects commute with the forgetful functor $\U$, see e.g. \cite[Section 3.2]{TouzeFund}. Thus one can also say that $F$ has degree less than $n$ if $\Cr_nF$ is zero. Then one easily sees that facts \ref{fact-adj} through \ref{fact-create} and facts \ref{fact-tens} and \ref{fact-tens-2} stay valid for strict polynomial functors, provided we replace $\Fct(\V,\A)$ by $\PP_{d,\kk}$ and $\Fct^{<n}(\V,\A)$ by $\PP_{d,\kk}^{<n}$ (the full subcategory supported by the strict polynomial functors of degree less than $n$).

\end{document}